\documentclass[11pt]{amsart}
\usepackage{geometry}

\usepackage[dvipsnames]{xcolor}
\definecolor{solarized}{RGB}{253,246,227}
\definecolor{bleudefrance}{rgb}{0.19, 0.55, 0.91}

\usepackage{eucal,mathrsfs}
\usepackage{hyperref}
\hypersetup{
colorlinks = true,
allcolors = blue,
}

\usepackage{amsmath,amsfonts,amssymb,amsthm,color,dsfont,amsbsy}
\usepackage{marginnote}
\usepackage{changepage}
\usepackage{amscd}
\usepackage{mathtools}
\usepackage{bm}
\usepackage{multirow}

\usepackage{tikz}
\usetikzlibrary{cd, calc, knots, arrows}
\usetikzlibrary{decorations.markings}
\tikzset{->-/.style={decoration={
markings,
mark=at position #1 with {\arrow{>}}},postaction={decorate}}}
\tikzset{-<-/.style={decoration={
markings,
mark=at position #1 with {\arrow{<}}},postaction={decorate}}}
\tikzset{-<<-/.style={decoration={
markings,
mark=at position #1 with {\arrow{-latex}}},postaction={decorate}}}
\usetikzlibrary{shapes.geometric}

\usepackage{graphicx}

\usepackage{tablefootnote}

\newtheorem{thm}{Theorem}[section]
\newtheorem{cor}[thm]{Corollary}
\newtheorem{prop}[thm]{Proposition}
\newtheorem{lemma}[thm]{Lemma}

\newtheorem{question}[thm]{Question}

\theoremstyle{definition}
\newtheorem{remark}[thm]{Remark}
\newtheorem{observation}[thm]{Observation}
\newtheorem{exmp}[thm]{Example}
\newtheorem{convention}[thm]{Convention}

\newtheorem{definition}[thm]{Definition}

\numberwithin{equation}{section}

\newcommand{\Isom}{\operatorname{I}}
\newcommand{\PSL}{\operatorname{PSL}}
\newcommand{\BP}{\mathbb{P}}

\def\CC{\mathcal{C}}

\def\FF{{\mathcal{F}}}

\def\BC{{\mathbb{C}}}

\def\BN{{\mathbb{N}}}
\def\BR{{\mathbb{R}}}
\def\BZ{{\mathbb{Z}}}

\def\BH{{\mathbb{H}}}

\def\BS{{\mathbb{S}}}
\def\kk{{\Bbbk}}
\def\BL{{\mathbb{L}}}

\newcommand{\BM}{\mathbb{M}}

\def\a{{\alpha}}

\def\ld{{\lambda}}
\def\Ld{{\Lambda}}
\newcommand\Lda[1]{\Lambda_{(#1)}}
\newcommand\tLd[1]{\tilde{\Lambda}_{(#1)}}

\def\De{{\Delta}}
\def\D{{\Delta}}

\def\be{{\beta}}
\def\b{\be}  

\def\e{{\varepsilon}}

\def\w{{\omega}}

\def\o{{\,\otimes\,}}

\def\s{{\sigma}}

\def\1{{\mathds{1}}}

\def\to{\rightarrow}

\def\Mp2{{\operatorname{Mp}_2(\BZ)}}

\def\SU{{\operatorname{SU}}}
\def\SO{{\operatorname{SO}}}

\def\op{{\operatorname{op}}}
\def\cop{{\operatorname{cop}}}

\def\Hom{{\operatorname{Hom}}}

\def\End{{\operatorname{End}}}

\def\Rep{{\operatorname{Rep}}}

\newcommand\Irr{{\operatorname{Irr}}}

\def\id{{\operatorname{id}}}

\def\Tor{{\operatorname{Tor}}}
\def\Tr{{\operatorname{Tr}}}

\renewcommand\subset{ \subseteq }

\def\mod\text{mod}

\usepackage{diagbox}
\usepackage{graphics}

\newcommand\ol[1]{\overline{#1}}

\newcommand{\ot}{\otimes}

\renewcommand{\mod}[1]{{\ \operatorname{mod}\ {#1}}}

\def\rhu{\rightharpoonup}
\def\lhu{\leftharpoonup}

\usepackage{tikz}
\usetikzlibrary{intersections,decorations.pathmorphing,decorations.fractals}
\usepackage{caption}
\usepackage{enumitem}
\usepackage{fullpage}
\usepackage[normalem]{ulem}

\newcommand{\pairing}[1]{{\left\langle {#1} \right\rangle}}

\newcommand{\dd}[2]{d_{#1}^{[#2]}}
\newcommand{\ff}[2]{f_{#1}^{[#2]}}

\newcommand{\ls}{{\ld  S}}
\newcommand{\frm}{\mathfrak{f}}
\newcommand{\frmL}{\frm_{\operatorname{L}}}
\newcommand{\frmR}{\frm_{\operatorname{R}}}
\newcommand{\Sym}{{\operatorname{Sym}}}
\newcommand{\sot}[1]{^{\otimes {#1}}}
\def\rev{{^{\operatorname{rev}}}}
\def\tP{{\tilde{P}}}
\def\tnu{\tilde{\nu}}
\usepackage{soul}

\title{Generalizations of Frobenius-Schur indicators from Kuperberg invariants}

\author{Liang Chang}
\address{Nankai University, School of Mathematical Sciences and LPMC, Tianjin, China}
\email{changliang996@nankai.edu.cn}
\author{Siu-Hung Ng}
\address{Louisiana State University, Department of Mathematics, Baton Rouge, LA 70803}
\email{rng@math.lsu.edu}
\author{Yilong Wang}
\address{Beijing Institute of Mathematical Sciences and Applications (BIMSA), Beijing, China}
\email{wyl@bimsa.cn}

\begin{document}
\begin{abstract}
We introduce an approach to produce gauge invariants of any finite-dimensional Hopf algebras from the Kuperberg invariants of framed 3-manifolds. These invariants are generalizations of Frobenius-Schur indicators of Hopf algebras. The computation of Kuperberg invariants is based on a presentation of the framed 3-manifold in terms of Heegaard diagram with combings satisfying certain admissibility conditions.  We provide framed Heegaard diagrams for two infinite families of small genus 3-manifolds, which include all the lens spaces, and some homology spheres. In particular, the invariants of the lens spaces $L(n,1)$ coincide with the higher Frobenius-Schur indicators of Hopf algebras. We compute the Kuperberg invariants of all these framed 3-manifolds, and prove that they are invariants of the tensor category of representations of the underlying Hopf algebra, or simply gauge invariants. 
\end{abstract}
\maketitle

\tableofcontents

\section{Introduction}\label{sec:intro}
The 2nd Frobenius-Schur (FS-) indicators $\nu_2(V)$ for complex irreducible representations $V$ of a finite group $G$ were introduced more than a century ago \cite{FS1} for studying the fundamental group of closed non-orientable surfaces $\Sigma$, namely
\[
|\Hom(\pi_1(\Sigma), G)| =\frac{1}{|G|^{\mathcal{X}(\Sigma)-1}}\sum_{V \in \Irr(G)} (\nu_2(V) \dim(V))^{\mathcal{X}(\Sigma)} 
\]
where $\Irr(G)$ denotes the set of complex irreducible representations of $G$, and $\mathcal{X}(\Sigma)$ is the Euler characteristic of $\Sigma$. The Frobenius-Schur Theorem asserts that $\nu_2(V) \in \{0,1,-1\}$ for any irreducible representation $V$ and that $\nu_2(V)=1$ (resp.~$-1$) if $V$ admits a nondegenerate $G$-invariant symmetric (resp.~skew symmetric) bilinear form on $V$, and 0 otherwise. The notion of Frobenius-Schur (FS-)indicators was first generalized to any finite-dimensional representation of semisimple Hopf algebras \cite{LM00}, and the corresponding Frobenius-Schur theorem was established therein. Around the same time, some generalizations of the 2nd FS-indicator were introduced for $C^*$-fusion categories \cite{FGSV99,FS03}, semisimple quasi-Hopf algebras \cite{MN05}, and pivotal tensor categories \cite{NS07}. 

Higher FS-indicators $\nu_n(V)$ for any finite-dimensional representation $V$ of a semisimple Hopf algebra $H$ and positive integer $n$ were also introduced in \cite{LM00}, and were studied extensively in \cite{KSZ}.  It was discovered in \cite{MN05, NS08} that the $n$-th indicators, for any positive integer $n$,  are invariants of the tensor category of representations of semisimple (quasi-)Hopf algebras. More precisely, if $H$ and $H'$ are semisimple (quasi-)Hopf algebras over $\BC$ such that there exists an equivalence of tensor categories $\FF:\Rep(H) \to \Rep(H')$, then $\nu_n(V) = \nu_n(\FF(V))$ for all $V \in \Rep(H)$. In particular, if the tensor categories of complex representations of finite groups $G_1$, $G_2$ are equivalent, then any representation of $G_1$ has the same indicators as its image in $\Rep(G_2)$ under the equivalence. The same statement holds for semisimple Hopf algebras. 

The FS-indicator $\nu_n(H)$ of the regular representation of $H$ admits many interesting interpretations in various contexts. For example, if $H = \BC G$ for a finite group $G$, then 
\[
\nu_n(H) = \#\{x  \in G\mid x^n=1\}\,.
\]
In particular, $\nu_2(\BC G)$ is the number of involutions in $G$. In general, for any semisimple Hopf algebra $H$ over $\BC$, 
\[
\nu_2(H) = \Tr(S)
\]
where $S$ is the antipode of $H$ (cf.\ \cite[p.19]{KSZ}). Moreover, $\nu_n(H)/\dim(H)$ is the Reshetikhin-Turaev invariant of the lens space $L(n,1)$ associated to the modular fusion category $\Rep(D(H))$, where $D(H)$ is the Drinfeld double of $H$. If $\Rep(H)$ and $\Rep(H')$ are equivalent fusion categories, then $\Rep(D(H))$ and $\Rep(D(H'))$ are equivalent modular fusion categories (equipped with the canonical pseudounitary pivotal structure). Since the Reshetikhin-Turaev invariants of 3-manifolds are constructed from categorical data, it is immediate to see that  $\nu_n(H) = \nu_n(H')$. Thus, one can conclude that the FS-indicator $\nu_n(H)$ is an invariant of the tensor categories $\Rep(H)$. 

A formula for $\nu_n(H)$ was obtained in terms of the integral and antipode of a semisimple Hopf algebra $H$ for each positive integer $n$ in \cite[p.19]{KSZ}. However, the value of this formula, again denoted by $\nu_n(H)$, can be computed for any \emph{nonsemisimple} finite-dimensional Hopf algebra $H$ over any base field and each positive integer $n$. It was proved algebraically in \cite{KMN12} that $\nu_n(H) = \nu_n(H')$ if $\Rep(H)$ and $\Rep(H')$ are equivalent as tensor categories which are not necessarily semisimple. Throughout this paper, we will simply call a quantity $f(H)$ obtained from a finite-dimensional Hopf algebra $H$ a \emph{gauge invariant} if $f(H) = f(H')$ for any Hopf algebra $H'$ such that $\Rep(H)$ and $\Rep(H')$ are equivalent tensor categories. 

Analogous to the gauge invariance of $\nu_n(H)$ for any semisimple Hopf algebras $H$ over $\BC$ viewed as the Reshetikhin-Turaev invariant of lens space $L(n,1)$, one can consider the Hennings-Kauffman-Radford (HKR) invariants of closed oriented 3-manifolds \cite{Hennings1996, Kauffman1995}. The HKR invariant $Z_{\operatorname{HKR}}(M,H)$ of a 3-manifold $M$ is defined for any finite-dimensional unimodular ribbon factorizable Hopf algebra over $\BC$ (not necessarily semisimple), and is shown to be a gauge invariant of $H$ \cite[Thm.\ 2]{KC17}. In this paper, we explore the gauge invariance  of another type of 3-manifold invariants, called the \emph{Kuperberg invariants},  for any finite-dimensional Hopf algebra. 

Given a finite-dimensional Hopf algebra $H$, Kuperberg introduced in \cite{Kup96} a topological invariant $K(M, \frm, H)$ for any closed oriented 3-manifold $M$ equipped with a framing $\frm$. The 3-manifold $M$ can be considered as a mapping class of the boundary surface $\Sigma$ of some handlebody. The corresponding Heegaard diagram is a planar presentation of $\Sigma$ decorated with the so-called upper and lower gluing curves determined by the mapping class corresponding to $M$ (see Section \ref{sec:def-Kup} for details). The framing $\frm$ is depicted in terms of a 2-dimensional vector field, called a \emph{combing}, on the Heegaard diagram with a \emph{twist front} satisfying the \emph{admissible condition} \eqref{eq:adm}, which is crucial to guarantee the correctness of the combinatorial data needed for computing $K(M, \frm, H)$.  One can construct from these combinatorial data an element $w \in H^{\ot g}$ and $w^* \in {H^*}^{\ot g}$ in terms of the integrals, the antipode and the distinguished grouplike elements of $H$ and $H^*$, where $g$ is the genus of $\Sigma$. The invariant $K(M, \frm, H)$ is a scalar in the base field of $H$ obtained by evaluating $w$ against $w^*$, see \eqref{eq:KInv}. 
By \cite{Kup96}, the scalar $K(M, \frm, H)$ is independent of the choice of framed Heegaard diagrams of $(M, \frm)$. 
For example, it is shown in Corollary \ref{cor:Kup-and-nu} that the lens space $L(n,1)$ with a specific framing $\frm_n$ has the Kuperberg invariant given by

\[K(L(n,1), \frm_n, H) = \ld S\left(\sum_{(\Ld)}\Lda{1}\Lda{2}\cdots \Lda{n}\right)\]
for any finite-dimensional Hopf algebra $H$ with the antipode $S$, where $\ld \in H^*$ is a right cointegral and $\Ld \in H$ is a left integral such that $\ld(\Ld)=1$. 
Here $\sum_{(\Ld)}\Lda{1}\Lda{2}\cdots \Lda{n}$ is the product of the $n$-folded comultiplication of $\Ld$, or simply the $n$-th Sweedler power $\Ld^{[n]}$ of $\Ld$. When $H$ is ribbon factorizable over $\BC$, the Kuperberg invariant for some lens spaces were computed in \cite{CW13}.
This invariant of the framed 3-manifold $(L(n,1), \frm_n)$ for each positive integer $n$ coincides with the aforementioned Frobenius-indicator $\nu_n(H)$ for any finite-dimensional Hopf algebra $H$. 
In particular, the Kuperberg invariant $K(L(n,1), \frm_n, H)$ is a gauge invariant of $H$, and this alludes to the following question:

\begin{question}\label{q1}
For any given framed 3-manifold $(M, \frm)$, must the Kuperberg invariant $K(M, \frm, H)$ be gauge invariant for any finite-dimensional Hopf algebra $H$?
\end{question}
Inspired by the categorical invariance of the Frobenius-Schur indicator $\nu_{n,k}(V)$ for any positive integers $n,k$ and object $V$ in a spherical fusion category $\CC$, one can consider $\nu_{n,k}(\CC) = \sum_{V \in \Irr(\CC)} \dim(V)\nu_{n,k}(V)$, which is an invariant of spherical fusion categories \cite{NS07}.  When $\CC=\Rep(H)$ for any semisimple Hopf algebra $H$ over $\BC$, $\nu_{n,k}(\CC) = \nu_{n,k}(H)$, which is a gauge invariant for semisimple Hopf algebras over $\BC$. This observation motivates our investigation of Question \ref{q1} for the lens spaces $L(n,k)$ for any coprime positive integers $n,k$, which are 3-manifolds of genus $1$. We have investigated Question \ref{q1} for framed 3-manifolds of genus $\le 2$ and provide an affirmative answer for all framed 3-manifolds of genus 1.

\begin{thm}[Theorem \ref{t:genus_1}]\label{thm:1}
Let $M$ be a 3-manifold of genus 1, and $\frm$ an arbitrary framing on $M$. For any finite-dimensional Hopf algebra $H$, $K(M, \frm, H)$ is a gauge invariant of $H$.
\end{thm}
We also consider a family of genus 2 framed large Seifert manifolds $\BM_{m,n}$ with a fixed choice of framing $\frm_{m,n}$ parametrized by $(m,n) \in\BN^2$ which includes some classical 3-manifolds such as the Poincar\'e homology sphere, and we have proved
\begin{thm}[Theorem \ref{t:genus_2}]\label{thm:2}
For any $(m,n) \in \BN^2$ and finite-dimensional Hopf algebra $H$, the Kuperberg invariant $K(\BM_{m,n}, \frm_{m,n}, H)$ is a gauge invariant of $H$.
\end{thm}
Since the sequence $\{\BM_{m,n}\}_{m,n}$ contains lens spaces of the form $L(k,1)$ for $k \ge 3$ (see Remark \ref{rmk:Mmn}), the theorem justifies that $K(\BM_{m,n}, \frm_{m,n}, H)$ is a genus 2 generalization of Frobenius-Schur indicators of $H$. In particular, when $H$ is semisimple, $K(\BM_{m,n}, \frm_{m,n}, H)$ coincides with a special case of the genus 2 generalization of the FS-indicator called the topological indicator defined in \cite{Atfd2}.

In the proof of both Theorems \ref{thm:1} and \ref{thm:2}, we first obtain a framed Heegaard diagram of the framed 3-manifold under study. If the 3-manifold is of genus $g$, then the framed Heegaard diagram consists of $g$ upper curves, $g$ lower curves, a combing, and $g$ twist fronts satisfying the admissible condition (see Section \ref{sec:def-Kup}). From the diagram, we tabulate the combinatorial data of the angles made at the intersections among the upper curves, lower curves and twist fronts, which enables us to compute the Kuperberg invariant. Then using Radford's formula for traces \cite{Radford1976}, we rewrite the Kuperberg invariant as the trace of a linear operator $P_H: H^{\o g} \to H^{\o g}$. This operator $P_H$, which depends on $(M, \frm)$, can be defined for any finite-dimensional Hopf algebra $H$. 

The last step above plays a key role for establishing our major results. Let $H$ and $K$ be finite-dimensional Hopf algebras.  Recall from \cite{NS07} that $\Rep(H)$ and $\Rep(K)$ are equivalent tensor categories if and only if $K \cong H_F$ as Hopf algebras, where $H_F$ is the Drinfeld twist of $H$ by a \emph{2-cocycle} $F \in H^{\o 2}$ (cf. \eqref{eq:2-cocycle}). More precisely, $H_F=H$ as an algebra, with the comultiplication 
\[\Delta_F(h) = F \Delta(h) F^{-1}\,.\]
To prove that $K(M, \frm, H)$ is a gauge invariant of $H$, it suffices to show that 
\[
K(M, \frm, H) = K(M, \frm, H_F)
\]
for any 2-cocycle $F \in H^{\o 2}$. Since $H = H_F$ as vector spaces, we need to prove 
\begin{equation}\label{eq:0}
    \Tr(P_H) = \Tr(P_{H_F}).
\end{equation}
for any 2-cocycle $F \in H^{\o 2}$. 

In order to establish the equality \eqref{eq:0}, one needs to understand the properties of the sequence $\{F_n \}_{n\in \BN}$, associated to any 2-cocycle $F$, which is defined in the following way: $F_1 := 1$, $F_2 := F$ and $F_{n+1}= (1 \o F_n) (\id \o \Delta^n)(F)$ where $\Delta^n$ is the $n$-folded comultiplication.  An elaborated list of properties of the sequence $\{F_n\}_n$ are proved in Section \ref{sec:technical}, and these technical results are essential to the proofs of our main theorems.

Throughout this paper, Hopf algebras are always assumed to be finite-dimensional over an arbitrary field unless otherwise stated. The rest of this paper is organized as follows. In Section 2, we briefly review Hopf algebras, 2-cocycles and their Drinfeld twists. Some basic results on general Hopf algebras will be established in this section. In Section 3, we explain how the Kuperberg invariant of a framed 3-manifold $(M,\frm)$ of genus $g$ is defined. This encompasses how the combinatorial data is derived from a framed Heegaard diagram of $(M,\frm)$, and how this data is applied to construct the elements $w \in H^{\otimes g}$ and $w^* \in H^{*\o g}$  so that $K(M, \frm, H) = w^*(w)$. The examples of computations for the 3-manifolds $\BS^3$, $\BS^1 \times \BS^2$ and the quaternion orbifold $\BS^3/Q_8$ of $\BS^3$ are discussed in detail. In Section 4, we focus on the Kuperberg invariants of lens spaces $L(n,k)$. We start with a certain choice of framed Heegaard diagrams of $L(n,k)$ representing a framing $\frm_R$ to obtain an expression of the corresponding Kuperberg invariant $K(L(n,k),\frm_R,H)$ in terms of the integrals $\ld \in H^*$, $\Ld \in H$ and its antipode $S$ of a finite-dimensional Hopf algebra $H$. This expression is shown to be the trace of a linear operator $P_H : H \to H$, whose gauge invariance is established in Theorem \ref{t:Lnk-fR} by using the properties of $F_n$ for any 2-cocycle $F$ of $H$. By investigating the Kuperberg invariant for different spin classes and Hopf degrees of framings, we prove the gauge invariance of $K(L(n,k),\frm,H)$ for arbitrary framing $\frm$ in Theorem \ref{t:genus_1}.
We dedicate Section 5 to the proofs of all technical results on the sequence $\{F_n\}_n$ for any 2-cocycle $F$ of $H$. These results are used throughout this paper.  Inspired by the gauge invariance of $K(L(n,k), \frm, H)$ for any Hopf algebra $H$, we define a family of gauge invariants $\tilde\nu_{n,k}(H)$ for any positive coprime integers $n,k$. These gauge invariants of $H$ are generally different, but $\tilde\nu_{n,k}(H) = K(L(n,k), \frm, H)$ when $H$ is semisimple over $\BC$. Finally, in Section 7, we study a family of framed 3-manifolds $(\BM_{m,n}, \frm_{m,n})$ of genus 2, and prove the gauge invariance of the Kuperberg invariant $K(\BM_{m,n}, \frm_{m,n}, H)$ for any Hopf algebra $H$.

\section*{Acknowledgment}
We thank Shuang Ming, Dylan Thurston and Hao Zheng for fruitful discussions. L.\ C.\ is supported by NSFC Grant 12171249. S.-H.\ N.\ is partially support by the Simons Foundation MPS-TSM-00008039. Y.\ W.\ is supported by NSFC Grant 12301045, Beijing Natural Science Foundation Key Program Grant Z220002, and the BIMSA startup fund.

\section{Hopf algebras and their Drinfeld twists}\label{sec:alg-prelim}
In this paper, we consider finite-dimensional Hopf algebras over an arbitrary base field $\kk$ unless stated otherwise. We briefly recall some basic notions on general finite-dimensional Hopf algebras in this section. The readers are referred to \cite{Montgomery1993}, \cite{Sweedler1968} or \cite{Rad94} for more details. For the purpose of defining and computing the Kuperberg invariants in the later sections, we will point out the differences between the conventions in this paper and that in Kuperberg's paper \cite{Kup96}.

\subsection{Integrals of Hopf algebras}
Let $H$ be a finite-dimensional Hopf algebra over a field $\kk$ with multiplication $m: H \o H \to H$, comultiplication $\De: H \to H \o H$, counit $\e: H \to \kk$ and antipode $S \in \End_\kk(H)$. We will use the Sweedler notation $\De(h) = \sum_{(h)} h_{(1)} \o h_{(2)}$ for any $h \in H$ or even  $\De(h) =  h_{(1)} \o h_{(2)}$ with the summation suppressed. 

Since $H$ is finite-dimensional, its dual space $H^*$ is also a Hopf algebra with multiplication $\De^*$, identity $\e$, comultiplication $m^*$, counit given by evaluation at $1_H$ and antipode $S^*$.  An element $x \in H$ is called \emph{grouplike}  if $x \ne 0$ and $\Delta(x)=x \o x$. An element $\b \in H^*$ is grouplike if and only  $m^*(\b)(h \o k) = \b(h) \b(k)$ for $h, k\in H$, or equivalently, $\b$ is an algebra homomorphism. For any grouplike element $\b \in H^*$, we will simply denote its \emph{convolution} inverse in $H^*$ by $\b^{-1}$.

There  are natural left and right $H^*$-actions $\rightharpoonup$ and $\leftharpoonup$ on $H$ respectively given by $f \rightharpoonup  h := \sum_{(h)} h_{(1)} f(h_{(2)})$ and $h \leftharpoonup  f := \sum_{(h)} f(h_{(1)}) h_{(2)}$ for $f \in H^*$ and $h \in H$. Since $H \cong H^{**}$ as Hopf algebras via the canonical isomorphism of vector spaces $j: H \to H^{**}$, there are $H$-actions $\rightharpoonup$ and $\leftharpoonup$ on $H^*$, namely
$h \rightharpoonup f = \sum_{(f)} f_{(1)} f_{(2)}(h)$ and $f \leftharpoonup h = \sum_{(f)} f_{(1)}(h) f_{(2)}$.

The iterated comultiplications on $H$ are defined inductively by 
\begin{equation}\label{eq:coprod}
\D^{(0)} = \e \cdot 1_H\,,\ \De^{(1)}=\id_H\,,\ \De^{(2)}=\De\,,\ \De^{(n+1)}=(\id\o \De^{(n)})\circ\De \text{ for $n\geq 2$}\,.
\end{equation}
For simplicity, we will use the abused notation $\De^n$ for $\De^{(n)}$, and the suppressed Sweedler notation of $\De^n(h) = h_{(1)} \o \cdots \o h_{(n)}$ for any $h \in H$. 

A right (resp.~left) \emph{integral} in $H$ is an element $\Ld^R$ (resp.~$\Ld^L$) in $H$ such that for all $h \in H$, $\Ld^R h = \e(h) \Ld^R$ (resp.~$h\Ld^L=\e(h)\Ld^L$).
A right (resp.\ left) \emph{cointegral} of $H$ is an integral $\ld^R$ (resp.~$\ld^L$) in $H^*$, i.e., $h \leftharpoonup \ld^R =\ld^R(h)\cdot 1$ (resp.\ $(\ld^L \rightharpoonup h) = \ld^L(h)\cdot 1$) for all $h \in H$.

Since the space of left (resp.\ right) integrals in $H$, denoted by $\int^l_H$ (resp.\ $\int^r_H$), is 1-dimensional, any $\Ld^L \in \int^l_H\setminus\{0\}$ determines a \emph{distinguished grouplike element} $\a \in H^*$ by $\Ld^L h=\alpha(h)\Ld^L$ for all $h \in H$. We say that $H$ is \emph{unimodular} if $\alpha = \e$. Similarly, the distinguished grouplike element $g\in H$ is determined by 
\begin{equation}\label{eq:ldR}
h \rightharpoonup \ld^R=\ld^R(h)g \quad \text{for all }  h \in H\,.
\end{equation}
The celebrated Radford formula \cite{Radford1976} remarkably relates the antipode and these distinguished grouplike elements:
\begin{thm}[Radford \cite{Radford1976}]
Let $H$ be a finite-dimensional Hopf algebra over any field $\kk$ with antipode $S$. Then
\[
S^4 (h) =\a \rhu g h g^{-1} \lhu \a^{-1} \quad\text{ for all } h \in H,
\]
where $g \in H$ and $\a \in H^*$ are the distinguished grouplike elements.
\end{thm}

Since the cyclic group algebra of any grouplike element is a Hopf subalgebra of $H$, by the Nichols–Zoeller freeness theorem \cite{Nichols1989}, we have the orders $\alpha$ and $g$ are divisors of $\dim(H)$, and $\alpha(g) \in \kk$ is a $\dim(H)$-th root of unity.

It was proved  in \cite{Rad90} that $(H, \lhu)$ is free right $H^*$-module generated by $\Ld^L$ and hence $\ld^R(\Ld^L) \ne 0$. From now on, we choose $\ld^R \in \int^r_{H^*}$ and $\Ld^L \in \int^l_H$ such that 
\[\ld^R(\Ld^L) = 1\,,\] 
and call them a pair of \emph{normalized integrals for $H$}. The following lemma is well-known, and the statement will serve as a reference of our notation for the remainder of this paper.  The proofs are provided for the sake of completeness.
\begin{lemma}\label{lem:int} Suppose the integrals $\ld^R \in H^*$ and $\Ld^L \in H$ are normalized,  and that $g \in H$, $\a \in H^*$ are the distinguished grouplike elements.
Define 
\[\Ld^R := S(\Ld^L) \in H \quad\text{and}\quad \ld^L := \ld^R\circ S^{-1} \in H^*\,.\]
Then:
\begin{enumerate}[label=\rm{(\roman*)}]
\item $\Ld^R$ is a right integral of $H$, and $\ld^L$ is a left integral of $H^*$;
\item $\ld^R(\Ld^R) = \ld^L(\Ld^R) = 1$;
\item $\ld^L(\Ld^L) = \a(g)$.
\end{enumerate}
\end{lemma}
\begin{proof}
The first statement follows immediately from the fact that $S$ is an algebra and coalgebra anti-automorphism of $H$. By \cite[Prop.~3]{Rad94}, we have 
\[\ld^R(\Ld^R) = \ld^R(S(\Ld^L)) = \ld^R(g\Ld^L) = \e(g)\ld^R(\Ld^L) = 1\,.\]
The equality $\ld^L(\Ld^R) = 1$ is a direct consequence of the definition of $\Ld^R$ and $\ld^L$. Finally, we have
\[\ld^L(\Ld^L) = \ld^R(S^{-1}(\Ld^L)) =\a(\Ld_1^L) \ld^R(\Ld_2^L)=\alpha(g)\,. \qedhere\]
\end{proof}

From \eqref{eq:ldR} and Lemma \ref{lem:int}, we have $\ld^R \rightharpoonup \Ld^R = \ld^R(\Ld^R)g = g$,
which is consistent with the definition of the distinguished grouplike element of $H$ in \cite[p.~113]{Kup96}. Moreover, by definition, for any $h \in H$, we have 
\begin{equation}\label{eq:hLdR}
h\Ld^R = h \cdot S(\Ld^L) = S(\Ld^L \cdot S^{-1}(h)) = \alpha(S^{-1}(h)) S(\Ld^L) = \alpha^{-1}(h) \Ld^R\,.
\end{equation}
Therefore, our $\alpha$ coincides with the inverse of Kuperberg's convention  \cite[p.~113]{Kup96} for the distinguished grouplike element in $H^*$. 

For any integer $n$, define
\begin{equation}\label{eq:def-ld-n}
\Ld_{n-\frac{1}{2}} :=
\alpha^{-n}\rightharpoonup S(\Ld^L)
\quad\text{and}\quad
\ld_{n-\frac{1}{2}}:=g^n \rightharpoonup \ld^R\,.
\end{equation}
Then it is easy to derive the following equations from definition and \cite[Prop.~3]{Rad94}:
\begin{equation}\label{eq:twisted-ld}
\begin{split}
&\Ld_{-\frac{1}{2}} = \a^0 \rhu \Ld^R =\Ld^R = S(\Ld^L)\,,
\quad \Ld_{\frac{1}{2}} = \a^{-1} \rhu S(\Ld^L) = \a^{-1}\rhu(\a\rhu\Ld^L) = \Ld^L\,,\\
&\ld_{-\frac{1}{2}} = g^0 \rhu \ld^R = \ld^R\,,
\quad \ld_{\frac{1}{2}} = g \rhu \ld^R = \ld^R \circ S^{-1} = \ld^L\,.
\end{split}\end{equation}

We will simply write $\ld \in \int_{H^*}^r$ and $\Ld \in \int^l_H$ instead of $\ld^R$ and $\Ld^L$ respectively. However, the conventions stated in Lemma \ref{lem:int} remain. The following theorem on integrals, essentially proved in \cite{Rad94}, will be repeatedly used in many of our proofs in this paper. In particular, statement (v) of the following theorem is often called the \emph{Radford trace formula}.
\begin{thm}\label{t:t2} 
    Let $H$ be a finite-dimensional Hopf algebra over $\kk$,  $\ld \in \int^r_{H^*}$ and $\Ld \in \int_H^l$ be a pair of normalized integrals, i.e., $\ld(\Ld) = 1$. For any $a, b \in H$ and $\kk$-linear map $X \in \End_\kk(H)$, we have
\begin{enumerate}[label=\rm{(\roman*)}] 
     \item $\ld(ab) = \ld(S^2(b\lhu \alpha) a)$,
    \item $\ld S(ab) = \ld S(bS(S(a) \lhu \a))$,
    \item $\Lda{1} \o a\Lda{2} = S(a)\Lda{1}\o \Lda{2}$,
    \item $\Lda{1}a \o \Lda{2} = \Lda{1}\o \Lda{2} S(a \lhu \a)$,
    \item $\Tr(X) = \ld\left(S(\Lda{2})X(\Lda{1})\right) = \ld\left(S(X(\Lda{2}))\Lda{1}\right)$,
    \item $\ld S\left(aX\left(\Ld_{(1)}\right)\Ld_{(2)}\right) = \ld S\left(X\left(\Ld_{(1)}S(a)\right) \Ld_{(2)}\right)$.
\end{enumerate}
\end{thm}
\begin{proof} Statements (i), (iii), (iv) and (v) are proved in \cite{Rad90, Rad94}. Statement (ii) follows immediately from (i), and statement (vi) is a consequence of (ii) and (iv). 
\end{proof}

The $n$-th \emph{Sweedler power} $P^{(n)}(x)$ of an element $x \in H$ is defined by $P^{(n)}(x) = x_{(1)} \cdots x_{(n)}$ with $P^{(0)}= 1\cdot\e$. The statement (vi) of the preceding theorem can be generalized as follows, and will be used in Section \ref{sec:Kmn}.

\begin{cor} \label{c:t2} Assuming the conditions of Theorem \ref{t:t2}. For any positive integer $n$,  $x \in H$, and  linear operator $Y : H \to H$, we have
\begin{enumerate}[label=\rm{(\roman*)}]
    \item $\ld S\Big( xY(\Lda{1})P^{(n)}(\Lda{2})\Big) = 
    \lambda  S\left(Y\left(S^2(x_{(1)})\Lda{1}S(x_{(3)})\right)  S^2(x_{(2)}) P^{(n)}\left(\Lda{2}\right) \right)$\,,
    \item $\ld  S\Big( P^{(n-1)}( \Lda{1})Y(\Lda{2})x\Lda{3}\Big)= 
    \ld  S\Big(P^{(n-1)}(\Lda{1}) S^2(x_{(2)}) Y(S(x_{(1)}) \Lda{2} S^2(x_{(3)}))\Lda{3}\Big)$. 
\end{enumerate}
\end{cor}
\begin{proof}
(i) For $n=1$, 
\[\begin{split}
 &\lambda S\left(x Y\left(\Lda{1}\right) \Lda{2}\right) = \lambda S\left(x Y\left(\Lda{1}\right)  S(\Lda{2}) \Lda{3} \Lda{4}\right)  \\  
=\, & \lambda S\left(Y\left(\Lda{1} S(x_{(3)})\right)  S\left(\Lda{2} S(x_{(2)})\right)  \Lda{3} S(x_{(1)}) \Lda{4}\right) \text{ by Theorem \ref{t:t2} (vi)}\\
=\,& \lambda S\left(Y\left(S^2(x_{(1)}) \Lda{1} S(x_{(5)})\right)  S(S^2(x_{(2)})\Lda{2}  S(x_{(4)})) S^2(x_{(3)})\Lda{3} \Lda{4} \right) \text{ by Theorem \ref{t:t2} (iii)}\\
=\,&\lambda S\left(Y\left(S^2(x_{(1)}) \Lda{1} S(x_{(3)})\right)  S^2(x_{(2)}) \Lda{2}\right) \text{ by the properties of the antipode}.
\end{split}\]
Assume the statement holds from some positive integer $n$. Note that
\[
\ld S\Big( xY(\Lda{1})P^{(n+1)}(\Lda{2})\Big) =   \ld S\Big( x(Y*\id)(\Lda{1})P^{(n)}(\Lda{2})\Big)  
\]
where $Y * \id$ is the convolution product of $Y$ and $\id_H$. By the induction assumption, we find
\[\begin{split}
 &\ld S\Big( xY(\Lda{1})P^{(n+1)}(\Lda{2})\Big)   =   \ld S\Big((Y*\id)(S^2(x_{(1)}) \Lda{1}S(x_{(3)}))S^2(x_{(2)})P^{(n)}(\Lda{2})\Big)  \\  
=\, & \ld S\Big((Y(S^2(x_{(1)}) \Lda{1}S(x_{(5)})) S^2(x_{(2)}) \Lda{2}S(x_{(4)})S^{2}(x_{(3)})P^{(n)}(\Lda{3})\Big) \\
=\,& \ld S\Big((Y(S^2(x_{(1)}) \Lda{1}S(x_{(3)})) S^2(x_{(2)}) P^{(n+1)}(\Lda{2})\Big).
\end{split}\]
Statement (ii) follows similarly by induction on $n$ where the case $n=1$ can be proved as follows: 
\[
\begin{split}
& \lambda S \left(Y\left(\Lda{1}\right) x\Lda{2}\right) = \lambda S\left(\Lda{1} S(\Lda{2}) Y\left(\Lda{3}\right) x \Lda{4}\right) \\
=\,&\lambda S\left(S(x_{(3)}) \Lda{1} S(S(x_{(2)})\Lda{2})    Y\left(S(x_{(1)}) \Lda{3}\right)   \Lda{4}\right) \text{ by Theorem \ref{t:t2} (iii)}\\
=\,& \lambda S\left(\Lda{1} S^2(x_{(3)}) S(\Lda{2}  S^2(x_{(4)})) S^2(x_{(2)})   Y\left(S(x_{(1)}) \Lda{3} S^2(x_{(5)})\right)   \Lda{4}\right)\text{ by Theorem \ref{t:t2} (vi)}\\
=\,& \lambda S \left(S^2(x_{(2)})   Y \left(S(x_{(1)}) \Lda{1} S^2(x_{(3)})\right)   \Lda{2}\right) \text{ by the properties of the antipode}. \qedhere
\end{split}
\]
\end{proof}

\subsection{Gauge transformation}
Before proceeding, we introduce some notations involving tensors.  Let $V$ be any finite-dimensional $\kk$-linear space. For any $n \ge 1$, we write an $n$-fold tensor $v \in V^{\ot n}$ as 
\[v = \sum_{i} v^{[1]}_{i} \ot v_{i}^{[2]} \ot \cdots \ot v^{[n]}_i\] 
or simply $v = v^{[1]}_i \ot \cdots \ot v^{[n]}_i$ with the summation notion suppressed. 
We may further abbreviate $v = v^{[1]}_i \ot \cdots \ot v^{[n]}_i$ as $v = \bigotimes_{t = 1}^n v^{[t]}_i$.  For instance, we write $v = v^{[1]}_i \ot \bigotimes_{t = 2}^{n-1} v^{[t]}_i \ot v_i^{[n]}$ to separate the first and last components from the others. Moreover, a permutation  $\s$ on $\{1,\dots, n \}$, denoted by  $\sigma\in \operatorname{Sym}_n$, can be considered as a linear automorphism on $V^{\ot n}$ given by 
\[\sigma v =  v^{[\s 1]}_{i} \ot v_{i}^{[\s 2]} \ot \cdots \ot v^{[\s n]}_i = \bigotimes_{t=1}^{n} v^{[\sigma t]}_i\,.\]

Now let $H$ be a finite-dimensional Hopf algebra over $\kk$. The iterative product of $H\sot{n}$ will be simply denoted by $m$, i.e., $m: H\sot{n} \to H$,  $m(v)=\prod_{t=1}^n v_i^{[t]}$ for any $v \in H\sot{n}$ (where the sum over $i$ will always be suppressed as before). We will use the convention $m(v) = v$ for any $v \in H$.   For $h \in H$, we denote $\De^n(h) = \bigotimes_{i=1}^n h_{(i)}$ by the Sweedler notation. In particular, $\De^{n+l}(h) = \bigotimes_{i=1}^n h_{(i)} \o \bigotimes_{i=n+1}^{n+l} h_{(i)}$ for any positive integers $n, l$.

A \emph{gauge transformation} of $H$ is an invertible element $F \in H^{\ot 2}$ such that $(\e \ot \id)F = (\id\ot\e)F = 1$. By changing the comultiplication of $H$ into \begin{equation*}
	\D_F(h) = F \D(h) F^{-1}\quad\text{for all $h \in H$}\,,
\end{equation*}
one obtains a quasi-Hopf algebra $H_F$ which is identical to $H$ as an algebra with the same counit $\e$ (see, for example, \cite[Chap.~XV]{Kas95}). This quasi-Hopf algebra $H_F$ is a Hopf algebra if and only if $F$ is a \emph{2-cocycle} (see \cite{KMN12}), which means $F$ is a gauge transformation such that the following equality holds
\begin{equation}\label{eq:2-cocycle}
1 \ot 1 = (1 \ot F)\cdot (\id \ot \D)(F) \cdot (\D \ot \id)(F^{-1}) \cdot (F^{-1} \ot 1)\,.   
\end{equation}
Let $F \in H \ot H$ be a 2-cocycle, and write $F = \ff{i}{1} \ot \ff{i}{2}$ with inverse $F^{-1} = \dd{i}{1} \ot \dd{i}{2}$. It can be easily derived from \eqref{eq:2-cocycle} that $u := \ff{i}{1}S(\ff{i}{2})$ is invertible with inverse $u^{-1} = S(\dd{i}{1})\dd{i}{2}$. As mentioned above, in this case, $H_F$ is a Hopf algebra with antipode
\begin{equation}\label{eq:S_F}
    S_F(h) := uS(h)u^{-1}\,,\quad\text{for all $h \in H$}\,.
\end{equation}
We call $H_F$ the (Drinfeld) twist of $H$ by the 2-cocycle $F$. The iterative comultiplication $\D_F^n$ of $H_F$ is the conjugation of $\D^n$ by $F_n \in H\sot{n}$, which is defined inductively as follows (cf. \cite{KMN12}): $F_1 = 1_H$ and 
\begin{equation}\label{eq:F-n-plus-1}
    F_{n+1}=(F_n\o 1)(\id \o \D^n)(F) \quad \text{for any integer }n \ge 1\,.
\end{equation}
The following lemma on $F_n$ was proved in \cite[Lem.\ 2.3]{KMN12}.
\begin{lemma}\label{lem:F-n-plus-1} 
Let $F$ be a 2-cocycle of a finite-dimensional Hopf algebra $H$. Then we have $F_{n+1}=(F_n\o 1)(\D^n \o \id)(F)$ and $\D_F^{n}=F_n\D^{n}F_n^{-1}$  for any integer $n \ge 1$. \qedhere
\end{lemma}

Extending the notations of $F$ and $F^{-1}$, we will write $F_n = \ff{i}{1} \ot\cdots \ot \ff{i}{n}$ with inverse $F_n^{-1} = \dd{i}{1} \ot\cdots \ot\dd{i}{n}$. Together with the Sweedler notation, we find 
\begin{equation}\label{eq:D_F(h)}
    \D^n_F(h) = \ff{i}{1} h_{(1)} \dd{j}{1} \ot \cdots \ot \ff{i}{n}h_{(n)} \dd{j}{n} \quad\text{ for } h \in H.
\end{equation}

Let $\Rep(H)$ be the category of finite-dimensional $H$-modules over $\kk$. It is a finite tensor category in the sense of \cite{EO04}. Following \cite{KMN12}, we give the definition of gauge invariants as follows.

\begin{definition}
Two finite-dimensional Hopf algebras $H$ and $H'$ over a base field $\kk$ are called \emph{gauge equivalent} if $\Rep(H)$ is equivalent to $\Rep(H')$ as tensor categories. A quantity $f(H)$ assigned to each finite-dimensional Hopf algebra $H$ over $\kk$ is called a \emph{gauge invariant} if $f(H) = f(H')$ whenever $H$ is gauge equivalent to $H'$, i.e., $f(H)$ is an invariant of the tensor category $\Rep(H)$.
\end{definition}

It is shown in \cite{NS08} (see also \cite{Sch96,EG02tri}) that two finite-dimensional Hopf algebras $H$ and $H'$ over $\kk$ are gauge equivalent if and only if there exists a 2-cocycle $F \in H \ot H$ such that $H' \cong H_F$ as bialgebras. Hence, gauge invariants of Hopf algebras are those remain unchanged under Drinfeld twists. The goal of this paper is to prove that the Kuperberg invariants of certain framed 3-manifolds are gauge invariants of Hopf algebras.

\section{Framed 3-manifolds and their Kuperberg invariants}\label{sec:def-Kup}
In this section, we review the construction of the Kuperberg invariant for closed oriented framed 3-manifolds. Following \cite{Kup96}, the invariant is defined by a diagrammatic presentation of any framed 3-manifold, called a \emph{framed Heegaard diagram}.

\subsection{Framed 3-manifolds}
Let $M$ be a closed oriented 3-manifold. It is well-known (see, for example, \cite[Chap.~9]{Rol76}) that $M$ admits at least one \emph{Heegaard splitting}, which is a decomposition $M = B_1 \cup_f B_2$ into a union of two handlebodies $B_1$ and $B_2$ of the same genus, say $g$,  glued along their common boundary surface, say $\Sigma_g$, via a homeomorphism $f$. Such a splitting can be encoded by a Heegaard diagram $(\Sigma_g, \mu, \eta)$ of $M$, where $\mu=\{\mu_1,\ldots, \mu_g\}$ and $\eta=\{\eta_1,\ldots, \eta_g\}$ are two collections of disjoint non-separating simple closed curves (or circles) on $\Sigma_g$, and $g$ is called the genus of the Heegaard diagram. Here, the curves are used to indicate how the handles of $B_1$ and $B_2$ are attached which equivalently describe the homeomorphism $f$. 

To obtain $M$ from a Heegaard diagram $(\Sigma_g, \mu, \eta)$, first note that $\BS^3 \setminus \Sigma_g$ has two connected components whose closures are handlebodies. We arbitrarily call one of them the \emph{upper} handlebody and the other one \emph{lower}, and glue these two handlebodies along $\Sigma_g$. The gluing is made so that the meridians of upper handlebody are attached to the curves $\mu_i$ and the meridians of lower handlebody are attached to the curves $\eta_j$. These curves are called the upper curves and the lower curves correspondingly. 
It is well-known that any closed oriented 3-manifold can be constructed in this way \cite{Rol76}. Conversely, according to the Reidemeister-Singer Theorem \cite{Rei33, Sin33} (see also \cite[Thm.\ 4.1]{Kup91} for a simple proof), any two Heegaard diagrams of a 3-manifold are equivalent by a sequence of \emph{stabilization moves}. A stabilization move is a modification of a Heegaard diagram $(\Sigma_g, \mu, \eta)$ by the following procedure: Take the connected sum of $\Sigma_g$ with a torus at a disk in $\Sigma_g$ that is disjoint from all of its upper and lower curves, then add a pair of upper and lower curves at the attached torus along its meridian and longitude (see Figure \ref{fig:stab}). 
\begin{figure}
\centering
\includegraphics[width=350pt]{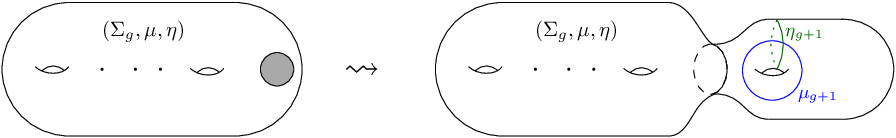}
\caption{The stabilization move.}
\label{fig:stab}
\end{figure}

We will use the following planar presentation of Heegaard diagrams (see, for example, \cite[Lec.~1]{Sav12}). Note that it suffices to indicate where the upper and lower curves are located on the surface $\Sigma_g$. We will draw $\Sigma_g$ minus a point as a plane with $2g$ open disks removed, whose boundaries are depicted by black circles in our presentation. It is understood, although not shown in the picture, that there are $g$ handles (above the plane) attached to the $g$ pairs of circles, which are called the \emph{attaching circles}. We arrange the diagram so that a pair of attaching circles joined by a handle is far away from the other pairs. The orientation of $M$ is induced by the standard orientation of $\BR^3$ where the Heegaard diagram sits as the $xy$-plane. 
Figure \ref{fig:L21-plain} is a Heegaard diagram of the lens space $L(2, 1)$ (see Section \ref{sec:lens-space} for definition), which is homeomorphic to $\mathbb{RP}^3$, where the two black circles represent the attaching circles of the handle, the upper curve is colored in blue, and the lower curve is colored in green.
\begin{figure}[ht]
\centering
\includegraphics[width=200pt]{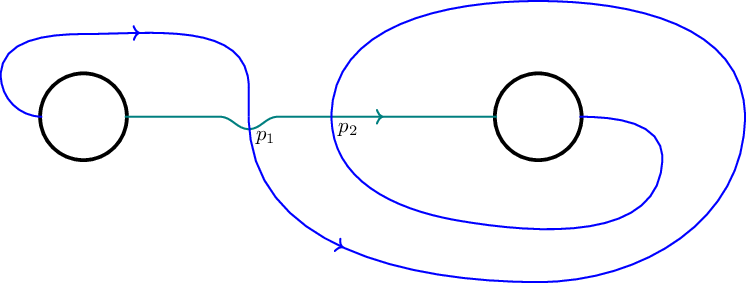}
\caption{A Heegaard diagram of $L(2,1)$.}
\label{fig:L21-plain}
\end{figure}

Equip $M$ with a Riemannian metric.  A \emph{combing} on $M$ is a unit-length tangent vector field up to homotopy, and a \emph{framing} $\frm$ on $M$ consists of three orthonormal combings. Combings and framings can be drawn on Heegaard diagrams of $M$. Given a Heegaard diagram $(\Sigma_g, \mu, \eta)$, a combing of $\Sigma_g$ is a vector field on $\Sigma_g$ with $2g$ singularities of index $-1$, one on each curve, and one singularity of index $+2$ disjoint from all curves. The singularity of index $-1$ on a curve is called the \emph{base point} of the curve and the two outward-pointing vectors are tangent to the curve. It is shown by Kuperberg \cite{Kup96} that a combing on $\Sigma_g$ can be extended to a combing of $M$, and any combing of $M$ is homotopic to an extension of some combing on $\Sigma_g$. 
In this paper, we use gray dashed lines with arrows in a Heegaard diagram to depict the flows of a combing on $\Sigma_g$ which is demonstrated as an example on the left part of Figure \ref{fig:combing}.

\begin{figure}[ht]
\centering
\includegraphics[width=0.5\linewidth]{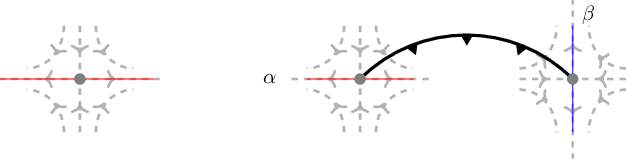}
\caption{\small The gray dashed lines with arrows represent the combing $b_1$, and the gray bullet on each curve stands for the base point. Left: $b_1$ points outwards from the base point along an upper or a lower curve. Right: $\alpha$ is a lower curve, $\beta$ is an upper curve. The black arc with small triangles represents the twist front.}
\label{fig:combing}
\end{figure}

To illustrate a framing $\frm$ on a Heegaard diagram $(\Sigma_g, \mu, \eta)$ of $M$, we start with a combing $b_1$ on $\Sigma_g$, which extends to a combing on $M$. Then, we describe a combing $b_2$ on $\Sigma_g$, which is orthogonal to $b_1$ in $\BR^3$ (not necessarily tangent to $\Sigma_g$), using \emph{twist fronts}, as instructed in \cite{Kup96}.
Twist fronts are arcs in the Heegaard diagram on which $b_2$ is normal to $\Sigma_g$ while pointing from the lower to the upper handlebody. A twist front is transversely oriented in the direction that $b_2$ rotates by the right-hand rule relative to $b_1$, indicated by small triangles based on the arc, for any local segment of the arc homotopic to an integral curve of $b_1$. When viewing from the upper handlebody,  the triangles point counterclockwise near the base point of the upper curve, and they point clockwise near that of the lower curve, see Figure \ref{fig:combing}. We then define a \emph{2-combed Heegaard diagram} of a 3-manifold $M$ to be a Heegaard diagram $(\Sigma_g, \mu, \eta)$ of $M$ that is decorated with a combing $b_1$ on $\Sigma_g$ and a twist front, see Figure \ref{fig:S3} for example.

As is mentioned above, the combing $b_1$ on $\Sigma_g$ extends to a combing on $M$. However, unlike $b_1$, the combing $b_2$ on the surface $\Sigma_g$ described by twist fronts may not be extended to a combing of 3-manifold $M$ orthogonal to $b_1$. 
The \emph{admissibility condition} is described as follows. Orient all the lower and upper curves arbitrarily. 
For each of these curves $c: [0,1] \to \Sigma_g$ with $c(0) = c(1)$, and any $t \in (0,1]$, let $\theta_c(p)$ denote the aggregated counterclockwise rotation of the tangent of $c$ relative to $b_1$ from the base point $c(0)$ to $p = c(t)$ in units of $1=360^\circ$. We denote $\theta_c(c(1))$ by $\theta_c$. Note that $\theta_c$ is always a half integer. Similarly, let $\phi_c(p)$ denote the aggregated right-handed rotation of $b_2$ around $b_1$ from the base point to $p$, and $\phi_c$ denote the total rotation when it is back to the base point. Alternatively, $\phi_c(p)$ is the aggregated number of signed crossings of $c$ with twist fronts: It is positive if the triangle on the twist front and the orientation of $c$ coincide at the intersection, and negative otherwise. 
The base point of $c$ contributes only half of the sign to $\phi_c$ when $c$ returns. By \cite{Kup96}, $b_2$ can be extended to a combing on $M$ if and only if it satisfies the following admissibility conditions:
\begin{equation}\label{eq:adm}
\text{$\theta_\mu=-\phi_\mu$ for each upper curve $\mu$, and $\theta_{\eta}=\phi_{\eta}$ for each lower curve $\eta$.}
\end{equation}
A 2-combed Heegaard diagram of $M$ satisfying the above admissibility condition is called a \emph{framed Heegaard diagram} of $M$. A framed Heegaard diagram $D$ affords a pair of orthogonal combings $(b_1, b_2)$ on $\Sigma_g$ that can be extended to $M$, and consequently yields a framing $\frm := (b_1, b_2, b_1\times b_2)$ on $M$. We call $\frm$ the \emph{diagram framing} associated to the framed Heegaard diagram $D$. Conversely, any framing on $M$ admits a (not necessarily unique) framed Heegaard diagram by restriction to a Heegaard surface.

\subsection{The Kuperberg invariant}
In \cite{Kup96}, for a fixed choice of Hopf algebra, Kuperberg defined a numerical invariant for framed 3-manifolds using framed Heegaard diagrams, and showed that this invariant only depends on the homeomorphism class of the framed 3-manifolds. In this section, we briefly recall the definition of this invariant.

Let $D$ be a framed Heegaard diagram of a 3-manifold $M$ consisting of upper and lower curves $\{\mu_1, ..., \mu_g\}$ and $\{\eta_1, ..., \eta_g\}$ respectively. Applying the stabilization move (Figure \ref{fig:stab}) if necessary, we may assume that $g \ge 1$. Let $I$ be the set of all intersection points between the lower and upper curves, and $n:=|I|$. For any $p \in I$, there exist unique $i,j$ such that  $p \in \eta_i \cap \mu_j$. If $p$ is the $x$-th intersection point starting from the base point of $\eta_i$, and the $y$-th point from the base point of $\mu_j$, then we write $l(p) = (i,x)$ and $u(p) = (j,y)$. It is immediate to see that $l, u : I \to \BN^2$ are injective maps. We order $I$ in two different ways via the lexicographical orderings of $l(I)$ and $u(I)$, and denote the corresponding ordered lists of intersection points by $I_*=\{p_1, \dots, p_{n}\}$ and $I^*=\{p^1, \dots, p^{n}\}$. Then we can define a permutation $\s_D \in \operatorname{Sym}_{n}$ on $\{1, \ldots, n\}$ via $p_{\sigma_{D}(r)} = p^{r}$ for all $1 \le r \le n$.  This in turn determines a linear isomorphism, again denoted by $\s_D: H^{\ot n} \to H^{\ot n}$, that permutes tensor components.

Let $T:H \to H$ be the linear map $T(x) := \alpha^{-1}(x_{(1)}) \cdot S^{-2}(x_{(2)}) \cdot \alpha(x_{(3)})$, where we use the suppressed Sweedler notation (see Section \ref{sec:alg-prelim}), $S$ is the antipode of $H$, and $\alpha \in H^*$ is the distinguished grouplike element  (cf.~\cite[p.~118]{Kup96}). Note that $T$ and $S$ commute.
The framed Heegaard diagram $D$ of $M$ gives rise to a generalization of the Sweedler power operator
\begin{equation}\label{eq:PMDH}
P(M, D, H) := \left(\bigotimes_{j=1}^{g} m_j \right)\circ \s_D \circ \left(\bigotimes_{r=1}^{n} S^{s_r}T^{t_r}\right)\circ\left(\bigotimes_{i=1}^{g} \D^{|\eta_i \cap I|}\right): H^{\ot g} \to H^{\ot g}\,,
\end{equation}
where for any $p_r \in I_*$, we define $s(p_r)$ and $t(p_r)$ by 
\begin{equation}\label{eq:sr-tr}
s(p_r) := 2(\theta_{\eta_i}(p_r)-\theta_{\mu_j}(p_r))+\dfrac{1}{2}\,,\quad 
t(p_r) :=\phi_{\eta_i}(p_r)-\phi_{\mu_j}(p_r)\,,
\end{equation}
and we may simply write $s_r$ and $t_r$ for $s(p_r)$ and $t(p_r)$ respectively. For $1\le j\le g$, $m_j: H^{\ot |\mu_j\cap I|} \to H$ is the multiplication (corresponding to the partition of $I^*$ by the upper curve $\mu_j$) . If a lower curve $\eta_i$ has no intersection with any upper curve, it is understood that $\D^{|\eta_i \cap I|}(\Ld_{\theta(\eta_{i})})=\e(\Ld_{\theta(\eta_{i})})1_{H}$, see \eqref{eq:coprod}.

Recall the conventions of $\Ld_{n-\frac{1}{2}}$ and $\ld_{n-\frac{1}{2}}$ defined in \eqref{eq:def-ld-n} for any integer $n$.

\begin{definition}\label{def:Kup}
Let $H$ be a finite-dimensional Hopf algebra over a field $\kk$, $M$ a 3-manifold with framing $\frm$ and $D$ a framed Heegaard diagram of $M$ representing $\frm$. The Kuperberg invariant of the framed manifold $(M, \frm)$ based on $H$ is defined to be the scalar
\begin{equation}\label{eq:KInv}
K(M, \frm, H) := \left\langle \bigotimes\limits_{j=1}^{g}\ld_{-\theta(\mu_j)}, \, P(M, D, H)\left(\bigotimes\limits_{i=1}^{g}\Ld_{\theta(\eta_i)}\right)\right\rangle \in \kk\,,
\end{equation}
where $g$ is the genus of $D$, $\{\mu_i\}_{i=1}^{g}$ and $\{\eta_i\}_{i=1}^{g}$ are upper and lower curves of $D$ respectively, $\ld \in \int^r_{H^*}$ and $\Ld \in \int^l_H$ are a pair of normalized integrals, and the bracket represents the evaluation of an element in $H^{\ot g}$ against an element in $(H^{*})^{\ot g}$.
\end{definition}

By the above definition, an easy way to get the formal expressions of the Kuperberg invariant can be described as follows. Using the notations above, for each lower curve $\eta_i$, we write $|\eta_i \cap I| = n_i$, $L^i := \Ld_{\theta(\eta_i)}$, and $\D^{n_{i}}(L^i) = L^i_{(1)} \ot \cdots \ot L^i_{(n_i)}$, and write 
\[
L = L^1_{(1)} \o \cdots \o L^1_{(n_1)} \o L^2_{(1)} \o \cdots \o L^1_{(n_2)} \o \cdots \o L^g_{(1)} \o \cdots \o L^g_{(n_g)} = L^{[1]} \o \cdots \o L^{[n]}\,.
\]
Now we go through the following process to write down the Kuperberg invariant:
\begin{enumerate}
\item Label the intersection points by the coproduct factors of the $L^i$'s according to the order of the point on the lower curves: If an intersection point $p$ satisfies $l(p) = (i,x)$, then we label $p$ on the diagram by $S^{s(p)}T^{t(p)}(L^i_{(x)})$, where the powers of $S$ and $T$ are given by \eqref{eq:sr-tr}. For example, the first intersection point $p$ on $\eta_i$ is labeled by $S^{s_p}T^{t_p}(L^i_{(1)})$.
\item The permutation $\sigma_D$ and $\bigotimes_{j=1}^g m_j$ is implemented as follows. For each upper curve $\mu_j$, starting from the base point, concatenate the labels on the intersection points on it according to its orientation, we will get a word of symbols $S^{s(p)}T^{t(p)}(L^i_{(x)})$ denoted by $W_{j}$. For example, suppose $|\mu_1\cap I|=k_1$ and the intersection points on $\mu_1$ are $p^1, ..., p^{k_1}$, then $W_1$ is $S^{s(p^1)}T^{t(p^1)}(L^{i_1}_{(x_1)})\cdots S^{s(p^{k_1})}T^{t(p^{k_1})}(L^{i_{k_1}}_{(x_{k_1})})$, where for $j=1, ..., k_1$, $(i_j,x_j)=l(p_{\sigma_D(j)})=l(p^{j})$.
\item Evaluate $\langle\ld_{-\theta(\mu_j)}, W_{j}\rangle$ for each $\mu_j$ and compute their products. Since $L$ is a sum of simple tensors, the Kuperberg invariant should be a summation of these products for each summand of $L$.
\end{enumerate}

\begin{convention}
In the remainder of this paper, we will always work with a pair of normalized integrals, so we will drop the superscripts for left/right (co)integrals of Hopf algebras. More precisely, for any finite-dimensional Hopf algebra $H$, we will always denote a  pair of normalized integrals for $H$ by $\ld \in \int^r_{H^*}$ and $\Ld \in \int^l_H$ such that $\ld(\Ld)=1$.
\end{convention}

The minimal genus of Heegaard diagrams of $M$ is called the genus of $M$. In the following, we give several examples of Kuperberg invariants of 3-manifolds of small genera.
\begin{exmp}\label{ex:S3}
The only 3-manifold of genus 0 is the 3-sphere $\BS^3$. In principle, we should be able to compute the Kuperberg invariant using a framed Heegaard diagram with neither handles nor upper/lower curves. However, to illustrate how to the Kuperberg invariant using the method described above, we apply the stabilization move once to get a framed Heegaard diagram as in Figure \ref{fig:S3}.
\begin{figure}[ht]
    \centering
    \includegraphics[width=250pt]{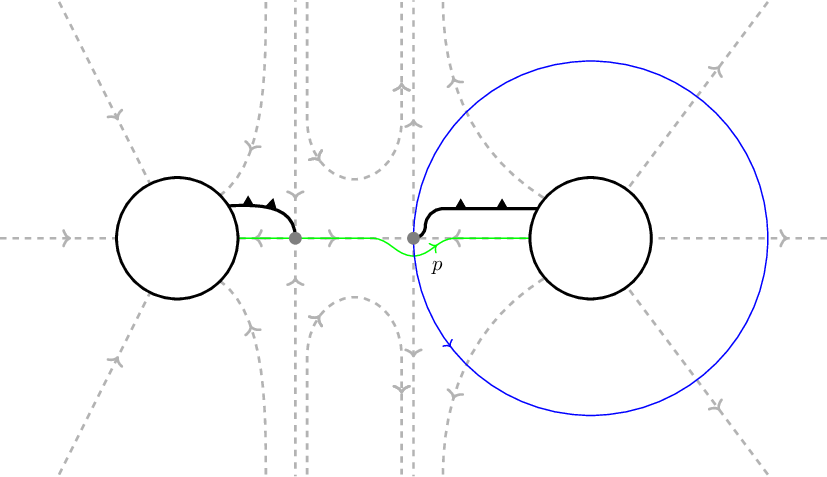}
    \caption{A genus 1 framed Heegaard diagram of $\BS^3$.}
    \label{fig:S3}
\end{figure}
In this Heegaard diagram, the two black circles represent the attaching circles for the handles. The upper curve $\mu$ is colored in blue, and the lower curve $\eta$ is colored in green. The base points of the curves are indicated by the gray bullets. We fix the choice of orientations on the upper and lower curves so that the lower curve goes to the right from its base point and the upper curve goes downwards from its base point. The combing information is encoded by the gray dashed lines (representing the combing $b_1$) and the twist fronts (the arc connecting the two base points, decorated with small triangles). The rotation numbers $\theta$ and $\phi$ at the intersection point $p$ and of the whole curves are given by
\[\begin{split}
& \theta_\eta(p) = \frac{1}{4}\,,\ \theta(\eta) = \frac{1}{2}\,,\ \phi_\eta(p) = 0 \,,\ \phi(\eta) = \frac{1}{2}\,,\\
& \theta_\mu(p) = 0\,,\ \theta(\mu) = \frac{1}{2}\,, \phi_\mu(p) = 0\,,\ \phi(\mu) = -\frac{1}{2}\,.
\end{split}\]
Since $\theta(\eta) = \phi(\eta)$ and $\theta(\mu) = -\phi(\mu)$, Figure \ref{fig:S3} is indeed a framed Heegaard diagram that extends to a framing, denoted by $\frm_0$, on $\BS^3$. Recall from \eqref{eq:coprod} and \eqref{eq:def-ld-n} that $\Delta^{(1)} = \id$, $\ld_{-\frac{1}{2}} = \ld$ and $\Ld_{\frac{1}{2}} = \Ld$, we have, for any finite-dimensional Hopf algebra $H$, 
\[K(\BS^3, \frm_0, H) = \left\langle\ld_{-\frac{1}{2}}, P(\Ld_{\frac{1}{2}})\right\rangle = \ld(S\D^{(1)}(\Ld)) = \ld(S(\Ld)) = 1\,,\]
where the last equality follows from Lemma \ref{lem:int}. 

Fix the above framing $\frm_0$, any other framing of $\BS^3$ (up to homotopy) can be identified with an element in the homotopy set $[\BS^3, \SO(3)]$, which, by \cite[Prop.~2.2]{Kup96}, is completely determined by its (Hopf) degree of the induced map from $H^3(\SO(3), \BZ)$ to $H^3(\BS^3, \BZ)$. According to the proof of \cite[Thm.~4.1]{Kup96}, for any Hopf algebra $H$ with distinguished grouplike elements $g \in H$ and $\alpha \in H^*$ (see \eqref{eq:ldR}), changing the Hopf degree of a framing of a 3-manifold by some $k \in \BZ$ changes its Kuperberg invariant by $\alpha(g)^{-k}$. 

We summarize the above discussions in the following proposition, which is known. The equality in the proposition can be found in \cite{Kup96}, we include it here for completeness and for a demonstration of the computation of Kuperberg invariants. The gauge invariance of $\alpha(g)$ has been proved in \cite[Thm.\ 3.10]{Shi15}.
\begin{prop}\label{prop:S3-inv}
Let $H$ be a finite-dimensional Hopf algebra over an arbitrary base field,  $g \in H$ and $\a \in H^*$ be distinguished grouplike elements. Then for any framing $\frm$ of $\BS^3$, there exists an integer $n \in \BZ$ such that 
\[K(\BS^3, \frm, H) = \alpha(g)^n\,.\]
In particular, $K(\BS^3, \frm, H)$ is a gauge invariant. \qed
\end{prop}
\end{exmp}

\begin{exmp}\label{ex:S2S1} We consider the genus one 3-manifold $\mathbb{S}^2 \times \mathbb{S}^1$ with a framed Heegaard diagram shown in Figure \ref{fig:S2S1}.  
\begin{figure}[ht]
    \centering
    \includegraphics[width=250pt]{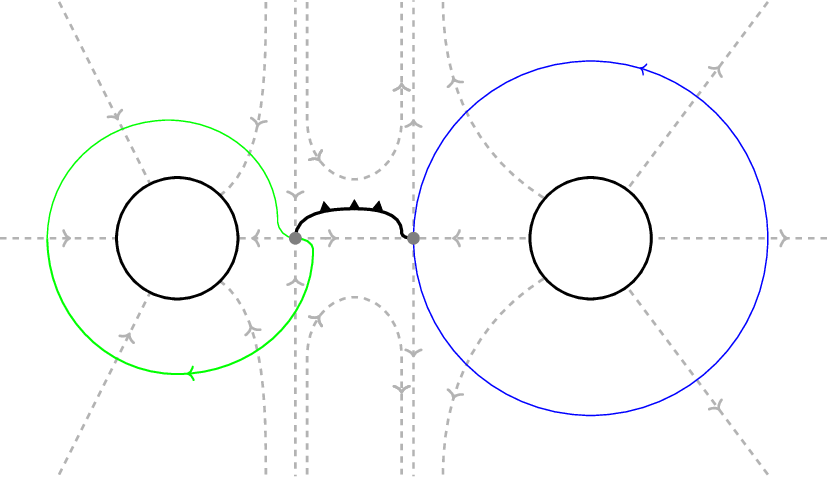}
    \caption{A framed Heegaard diagram of $\BS^2 \times \BS^1$.}
    \label{fig:S2S1}
\end{figure}
As in Example \ref{ex:S3}, the upper curve $\mu$ is colored in blue, and the lower curve $\eta$ is colored in green. Since there is no intersection point, we only have to record the rotation numbers $\theta$ and $\phi$ of the curves:
\[
\theta(\eta) = -\frac{1}{2} = \phi(\eta)\,,\ \theta(\mu) = \frac{1}{2} = -\phi(\mu) \,.
\]
In particular, Figure \ref{fig:S2S1} is a framed Heegaard diagram that determines a framing on $\BS^2\times \BS^1$, denoted by $\frm_1$. By definition, $\D^{(0)} = \e$, and there is no $S$ or $T$ term in $P$ (see \eqref{eq:PMDH}) as there is no intersection point. Therefore, by \cite[Prop.~2.13]{Mon00}, we have, for any finite-dimensional Hopf algebra $H$, 
\[K(\BS^2 \times \BS^1, \frm_1, H) = \pairing{\ld_{-\frac{1}{2}}, P(\Ld_{-\frac{1}{2}})} = \ld(\e(S(\Ld)) = \ld(1)\e(\Ld) = \Tr(S^2)\,,\]
see also \cite{Larson1988a}. 
In particular, by \cite{Larson1988a, Larson1988}, $H$ is semisimple and cosemisimple if and only if $K(\BS^2 \times \BS^1, \frm_1, H) \ne 0$. Note also that by Shimizu (see also \cite[Thm.~2.2]{KMN12}), $K(\BS^2 \times \BS^1, \frm_1, H)$ is a gauge invariant.

In general, we can restrict any framing $\frm$ of $\BS^2 \times \BS^1$ to the genus 1 Heegaard diagram of $\BS^2 \times \BS^1$ (i.e., Figure \ref{fig:S2S1} without the combing and the twist front) to get a framed Heegaard diagram $D_\frm$ of $\frm$ for $\BS^2 \times \BS^1$. Of course, the combing and the twist front on $D_\frm$ is determined by $\frm$, and using $D_\frm$, we can compute the Kuperberg invariant as follows. Assume that in $D_\frm$, we have $\theta(\mu) = -(a-\frac{1}{2})$ for some $a \in \BZ$ and $\theta(\eta) = -(b+1)-\frac{1}{2}$ for some $b \in \BZ$, then 
\[\begin{split}
& K(\BS^2\times\BS^1, \frm, H) = \pairing{\ld_{a-\frac{1}{2}},\e(\Ld_{-b-\frac{1}{2}})} = \pairing{g^a \rhu \ld, \e(\alpha^{b+1}\rhu S(\Ld))}\\ 
=& \pairing{g^a \rhu \ld, \e(\alpha^{b}\rhu \Ld)}
= \ld(g^a) \cdot \alpha^b(\Ld)\,.
\end{split}\]
By definition, for any $a \in \BZ$, we have $(\ld \ot \id)\D(g^a) = \ld(g^a) \cdot g^a = \ld(g^a)\cdot 1$. Since grouplike elements of $H$ are linearly independent, $\ld(g^a) = \ld(1) \delta_{g^a,1}$.  By duality,  $\alpha^b(\Ld) = \e(\Ld) \delta_{\a^b,\e}$, and so
\[
K(\BS^2\times\BS^1, \frm, H) = \ld(1)\e(\Ld) \delta_{g^a,1}\delta_{\a^b,\e}\,.
\]
If $H$ is semisimple and cosemisimple, then $g=1$ and $\a=\e$ and hence 
$K(\BS^2\times\BS^1, \frm, H) = \ld(1)\e(\Ld) =\Tr(S^2)=\dim(H) \in \kk$ by \cite{EG98,Larson1988a}, where $\kk$ is the base field of $H$. Otherwise, $\ld(1)\e(\Ld) =0 =\Tr(S^2)$ and so $K(\BS^2\times\BS^1, \frm, H) =\Tr(S^2)$. Now, we can present the following proposition concerning the Kuperberg invariant of $\BS^2 \times \BS^1$ as follows.
\begin{prop}\label{prop:S2S1-inv}
Let $H$ be a finite-dimensional Hopf algebra over an arbitrary base field. Then for any framing $\frm$ on $\BS^2 \times \BS^1$, 
\[K(\BS^2 \times \BS^1, \frm, H) =\Tr(S^2)\,,\]
where $S$ is the antipode of $H$. 
Consequently, $K(\BS^2 \times \BS^1, \frm, H)$ is a gauge invariant. \qed
\end{prop}
\end{exmp}

\begin{exmp}\label{Ex:Q8}
Let $M=\BS^3/Q_8$ be the quotient manifold of $\BS^3$ by the action of the quaternion group $Q_8$. Note that $\BS^3$ is homeomorphic to $\operatorname{SU}(2)$, which  is the compact group of unit real quaternions, and the quaternion group $Q_8$ of order $8$ is a discrete subgroup of $\operatorname{SU}(2)$. So $\BS^3/Q_8$ is simply the cosets of $Q_8$ in $\operatorname{SU}(2)$. Figure \ref{fig:Q8} is a framed Heegaard diagram of $M$, which is obtained by switching the upper and lower handle bodies in the Heegaard diagram described in \cite[Fig.~7]{GH07}. We then decorate it with a combing and twist fronts.
\begin{figure}[ht]
    \centering
    \includegraphics[width=.5\linewidth]{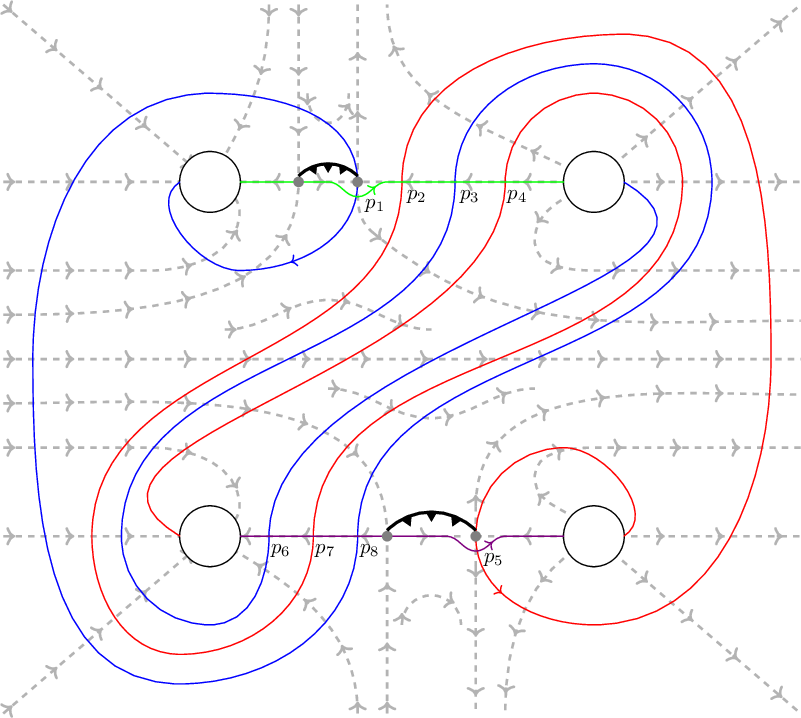}
    \caption{A framed Heegaard diagram of $M = \BS^3/Q_8$.}
    \label{fig:Q8}
\end{figure}
More precisely, in Figure \ref{fig:Q8}, the two pairs of black circles represent the attaching circles for handles, and the lower and upper curves are colored according to Table \ref{tbl:Q8}. The rotation numbers $\theta$ and $\phi$ of the $I$-points can be read from the diagram as we go along each curve from their base points following their orientations, which is recorded in Table \ref{tbl:Q8}. Note that since there is no intersection between the curves the twist fronts, and so the $T$-terms in \eqref{eq:sr-tr} are all identities.
\begin{table}[ht]
\begin{tabular}{|c|c|c||c|c|c|c|c|}
\hline
\multirow{6}{*}{Lower Curves} & \multirow{3}{*}{$\eta_1$ (green)} & $\eta_1 \cap I$ & $p_1$ & $p_2$ & $p_3$ & $p_4$ & Total\\
\cline{3-8}&& $\theta$ & $\frac{1}{4}$ & $\frac{1}{2}$ & $\frac{1}{2}$ & $\frac{1}{2}$ & $\frac{1}{2}$\\
\cline{3-8}& & $\phi$ & $0$ & $0$ & $0$ & $0$ & $\frac{1}{2}$\\
\cline{2-8}& \multirow{3}{*}{$\eta_2$ (violet)} & $\eta_2 \cap I$ & $p_5$ & $p_6$ & $p_7$ & $p_8$ & Total\\
\cline{3-8}& & $\theta$ & $\frac{1}{4}$ & $\frac{1}{2}$ & $\frac{1}{2}$ & $\frac{1}{2}$ & $\frac{1}{2}$\\
\cline{3-8}& & $\phi$ & $0$ & $0$ & $0$ & $0$ & $\frac{1}{2}$\\
\hline
\multirow{8}{*}{Upper Curves} & \multirow{4}{*}{$\mu_1$ (blue)} & $\mu_1 \cap I$ & $p_1$ & $p_6$ & $p_3$ & $p_8$ & Total\\
\cline{3-8} && $\theta$ & $0$ & $-\frac{3}{4}$ & $-\frac{1}{4}$ & $\frac{1}{4}$ & $-\frac{1}{2}$\\
\cline{3-8} && $\phi$ & $0$ & $0$ & $0$ & $0$ & $\frac{1}{2}$\\
\cline{3-8}& & $S$-term in \eqref{eq:sr-tr} & $S$ & $S^3$ & $S^2$ & $S$ & $ $\\
\cline{2-8}& \multirow{4}{*}{$\mu_2$ (red)} & $\mu_2 \cap I$ & $p_5$ & $p_2$ & $p_7$ & $p_4$ & Total\\
\cline{3-8} && $\theta$ & $0$ & $\frac{1}{4}$ & $-\frac{1}{4}$ & $-\frac{3}{4}$ & $-\frac{1}{2}$\\
\cline{3-8} && $\phi$ & $0$ & $0$ & $0$ & $0$ & $\frac{1}{2}$\\
\cline{3-8}& & $S$-term in \eqref{eq:sr-tr} & $S$ & $S$ & $S^2$ & $S^3$ & $ $\\
\hline
\end{tabular}
\caption{Colors and rotation numbers in Figure \ref{fig:Q8}.}
\label{tbl:Q8}
\end{table}

In particular, Figure \ref{fig:Q8} satisfies the admissibility condition for the rotation numbers, so it is a framed Heegaard diagram of $M$, denoted by $D$.  Its corresponding diagram framing is denoted by $\frm$. The element $\Ld_{H^{\otimes n}} = \Ld^{\otimes n}$ is a left integral of $H^{\otimes n}$. We use the simplified (and abused) notation to denote  $\Ld_{H^{\otimes n}}^{[1]} \ot \cdots \ot \Ld^{[n]}_{H^{\ot n}}$ as $\Ld^{1} \ot \cdots \ot \Ld^{n}$, so that different operators in \eqref{eq:PMDH} can be applied to $\Ld$ at different positions. By \eqref{eq:def-ld-n}, we have
\[ \Ld_{\theta(\eta_1)} = \Ld^1_{\frac{1}{2}} = \Ld^1\,,\text{and }
  \Ld_{\theta(\eta_2)} = \Ld^2_{\frac{1}{2}} = \Ld^2
  \]
and
\[\ld_{-\theta(\mu_1)} = \ld_{\frac{1}{2}} = g\rhu \ld\,,\text{and } \ld_{-\theta(\mu_2)} = \ld_{\frac{1}{2}} = g\rhu \ld\,.\]
Note that we do not have to distinguish the cointegrals in the evaluation of the Kuperberg invariant by \eqref{eq:KInv}.

The order of the intersection points are determined by our choice of orientation of the curves on the framed Heegaard diagram $D$ (Figure \ref{fig:Q8}). Using the notations in Section \ref{sec:def-Kup}, we have $I_* = \{p_1, ..., p_8\}$, $I^* = \{p^1, ..., p^8\} = \{p_1, p_6, p_3, p_8, p_5, p_2, p_7, p_4\}$, and the permutation $\sigma_D$ is $(2,6)(4,8)$ in cycle form. Moreover, $l(p_1) = (1,1)$, $l(p_6) = (2,2)$, $l(p_3) = (1, 3)$, $l(p_8) = (2,4)$, $l(p_5) = (2,1)$, $l(p_2) = (1,2)$, $l(p_7) = (2,3)$ and $l(p_4) = (1,4)$. Using these as indices and \eqref{eq:twisted-ld}, we have
\begin{equation}\label{eq:Q8-1}
\begin{split}
K(M, \frm, H) = &\ld\left(S(\Ld^{1}_{(1)})S^3(\Ld^{2}_{(2)})S^2(\Ld^{1}_{(3)})S(\Ld^{2}_{(4)})g\right) \cdot \ld\left(S(\Ld^2_{(1)})S(\Ld^1_{(2)})S^2(\Ld^2_{(3)})S^3(\Ld^1_{(4)})g\right)\\	
=&\ld\left(\Ld^2_{(4)}S(\Ld^1_{(3)})S^2(\Ld^2_{(2)})\Ld^1_{(1)}\right)\cdot\ld\left(S^2(\Ld^1_{(4)})S(\Ld^2_{(3)})\Ld^1_{(2)}\Ld^2_{(1)}\right)\,.
\end{split}\end{equation}
\end{exmp}

\section{Gauge invariance of the Kuperberg invariant of lens spaces}\label{sec:lens-space}

In this section, we study the gauge invariance of the Kuperberg invariants of framed 3-manifolds admitting Heegaard splittings of genus 1, which are called lens spaces. In a genus 1 Heegaard diagram of a lens space, there is only 1 lower curve and 1 upper curve. One can always fix the lower curve so that the lens space is determined (up to homeomorphism) by the homotopy class of the upper curve, which is parameterized by a pair of coprime integers $n, k\in \BZ$. Such a lens space is said to have type $(n,k)$ and is denoted by $L(n,k)$ (see \cite[Chap.~9]{Rol76}). In particular, we have homeomorphisms $L(1,k) \cong \BS^3$ for all $k \in \BZ$, $L(0,1) \cong \BS^2 \times \BS^1$, and  
\begin{equation} \label{eq:lens_iso}
    L(n,k) \cong L(n,-k) \cong L(-n,-k) \cong L(n,k+mn) \quad (\forall m \in \BZ)
\end{equation}
for all coprime integers $n,k$.

By Propositions \ref{prop:S3-inv} and \ref{prop:S2S1-inv}, the Kuperberg invariants of the lens spaces $\BS^3$ and $\BS^2 \times \BS^1$ are gauge invariants of the underlying Hopf algebras for any choice of framings. The main theorem of this section is an extension of these results to all lens spaces, which is stated as follows.

\begin{thm}\label{thm:Lnk}
Let $n$, $k$ be coprime integers, and $H$ a finite-dimensional Hopf algebra over a field $\kk$. For any framing $\frm$ on the lens space $L(n,k)$, the Kuperberg invariant $K(L(n,k),\frm, H)$ is a gauge invariant of $H$.
\end{thm}

In light of \eqref{eq:lens_iso} and Propositions \ref{prop:S3-inv} and \ref{prop:S2S1-inv}, we only have to prove the theorem for $0 < k < n$, which is what we assume in the following.

We start with our presentation of a Heegaard diagram of $L(n,k)$ without combing information, which is shown in Figure \ref{fig:HD-Lnk}. 
Let $r$ be the remainder upon the division of $n$ by $k$. In Figure \ref{fig:HD-Lnk}, the blue lines represent the upper curve and the the horizontal green line is the lower curve. 
A strand labeled by $k$, $r$ or $k-r$ is understood as the corresponding number of parallel strands of the upper curve, and the box is used to denote the place where we separate the strands. The two black circles are the attaching circles of one handle. In particular, there are exactly $n$ strands of the upper curve between the attaching circles and passing through the lower curve, i.e., the upper curve and the lower curve have $n$ intersection points. As before, we denote the set of intersection points between the lower and upper curve by $I$, and an element in $I$ is called an $I$-point. We orient the lower and upper curves so that the lower curves goes from left to right, and the upper curve go through the lower curve from top to bottom. 

\begin{observation}\label{obs:pts}
When moving along the upper curve  with the above choice of orientation, there are $(k-1)$ $I$-points on the lower curve between any $I$-point to its next one on the upper curve.
\end{observation}

\begin{figure}[ht]
\centering
\includegraphics[width=0.65\textwidth]{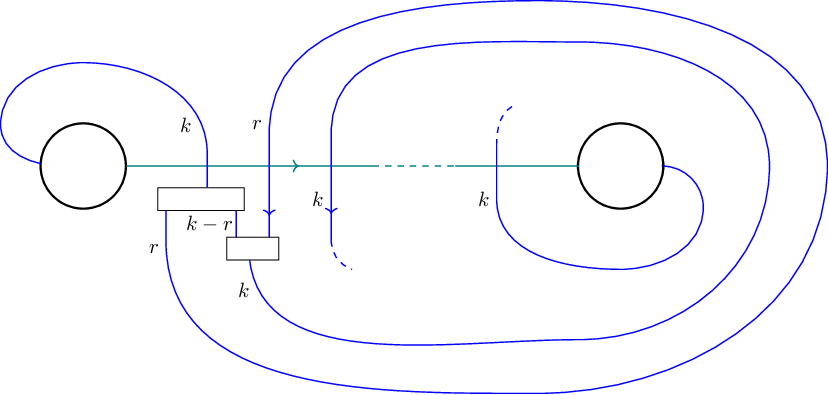}
\caption{A Heegaard diagram of $L(n,k)$.}
\label{fig:HD-Lnk}
\end{figure}

The rest of this section is organized as follows. In Section \ref{subsec:diag-frm}, we introduce two  diagram framings of $L(n,k)$ depending on the parity of $k$ and $n-k$ by explicitly presenting  2-combed Heegaard diagrams. Then, in Section \ref{subsec:gauge-Lnk}, we verify the admissibility condition of these 2-combed Heegaard diagrams, derive their corresponding Kuperberg invariants, and show their gauge invariance in Theorems \ref{t:Lnk-fR} and \ref{t:Lnk-fL}. Finally, combining the above results and an analysis of general framings on lens spaces in Section \ref{subsec:gen-frm}, we establish our main result on the gauge invariance of the Kuperberg invariant of any framed lens space in Theorem \ref{t:genus_1}.

\subsection{Special diagram framings of lens spaces}\label{subsec:diag-frm}
\subsubsection{The diagram framing \texorpdfstring{$\frmL$}{} for odd \texorpdfstring{$k$}{}}\label{subsubsec:fL-pic}
When $k$ is odd, let $k_{1} := \tfrac{k-1}{2}$, and consider the 2-combed Heegaard diagram demonstrated in Figure \ref{fig:fL-no-move}. 
As before, the blue strands represent the upper curve, and the green horizontal line represents the lower curve. It is understood that number of blue strands to the right of the one labelled by $r$ may vary and depends on $n$ and $k$. 
The gray dashed lines indicate the direction of the combing $b_{1}$ on the surface, with base points located at the two gray bullets. We will denote the left (resp.\ right) base point by $\xi_{1}$ (resp.\ $\xi_{2}$). 
In order to get a cleaner presentation, we put the labels of the base points on the vertical dashed lines connecting to them.
In Figure \ref{fig:fL-no-move}, we place the base points so that there is no $I$-point left to $\xi_{1}$, and there are $k_{1}$ $I$-points to the left of $\xi_{2}$, and we isotope the lower curve a little bit so that the $(k_{1}+1)$-th intersection point and $\xi_{2}$ are aligned vertically. The arc connecting the base points, decorated with small triangles, represents the twist front. 
We will show that Figure \ref{fig:fL-no-move} represents a framed Heegaard diagram of $L(n,k)$ when $k$ is odd, and the corresponding diagram framing is denoted by $\frmL$. 

Note that by definition, the formula for the Kuperberg invariant of $L(n,k)$ derived from Figure \ref{fig:fL-no-move} contains several powers of the operator $T$ due to the intersection of the upper curve with the twist front, which may result in computational difficulties in the study of its gauge invariance. 
So our idea is to changes Figure \ref{fig:fL-no-move} into a new framed Heegaard diagram Figure \ref{fig:fL-detail} via the local moves introduced in \cite{Kup96} so that the homotopy class of the diagram framings in these two diagrams are equal, so are their corresponding Kuperberg invariants. In Figure \ref{fig:fL-detail}, the twist front intersects with the upper curve only at the base point. Nevertheless, we will see later in the proof of our main result Theorem \ref{t:genus_1} that Figure \ref{fig:fL-no-move} is still useful.
\begin{figure}[ht]
    \centering
    \includegraphics[width=0.75\textwidth]{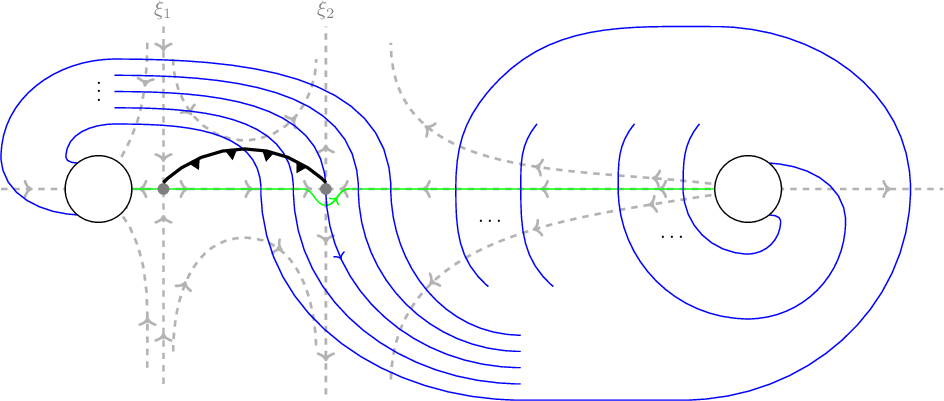}
    \caption{A framed Heegaard diagram of $L(n,k)$ when $k$ is odd. The diagram framing is denoted by $\frmL$. There are $k_1=\frac{k-1}{2}$ $I$-points to the left of $\xi_2$.}
    \label{fig:fL-no-move}
\end{figure}

In \cite[Sec.~2.2]{Kup96}, the \emph{base point isotopy move} among diagram framings of 3-manifolds was introduced. Roughly speaking, if two diagram framings on the Heegaard diagram of a 3-manifold are different by moving the base point of a lower (or upper) curve  along the curve across a $I$-point in such a way that the arc representing the twist front is moved along with the base point, then these two diagram framings are said to be different by a base point isotopy move. See Figure \ref{fig:base-pt-isotopy} for a local picture. It is shown in \cite[Thm.~4.1]{Kup96} that if two diagram framings of a 3-manifold are different by a base point isotopy move, then they represent two homotopic framings on the manifold. As can be seen from the Figure \ref{fig:base-pt-isotopy}, the base point isotopy move can  be described as moving a part of an upper (resp.\ a lower) curve across the base point of an intersecting lower (resp.\ upper) curve. 
\begin{figure}[ht]
\centering
\includegraphics[width=200pt]{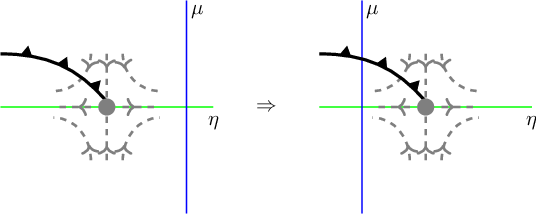}
\caption{The base point isotopy move.}
\label{fig:base-pt-isotopy}
\end{figure}

Now we perform the base point isotopy move to Figure \ref{fig:fL-no-move} by moving the base point $\xi_{1}$ along the lower curve to the right so there is no $I$-point between $\xi_{1}$ and $\xi_{2}$ horizontally. In this process, $\xi_1$ passes through $k_1$ $I$-points, and in particular, when $k=1$, we don't make any change as there is already no $I$-point left to $\xi_1$. The resulting 2-combed Heegaard diagram is demonstrated in Figure \ref{fig:fL-detail}. We will show that this 2-combed Heegaard diagram is a framed Heegaard diagram, and since it represents the same homotopy class of framing as Figure \ref{fig:fL-no-move}, we still denote the diagram framing in Figure \ref{fig:fL-detail} by $\frmL$, and we will use the formula of Kuperberg invariant associated to this diagram framing to prove its gauge invariance. 
\begin{figure}[ht]
\centering
\includegraphics[width=0.8\textwidth]{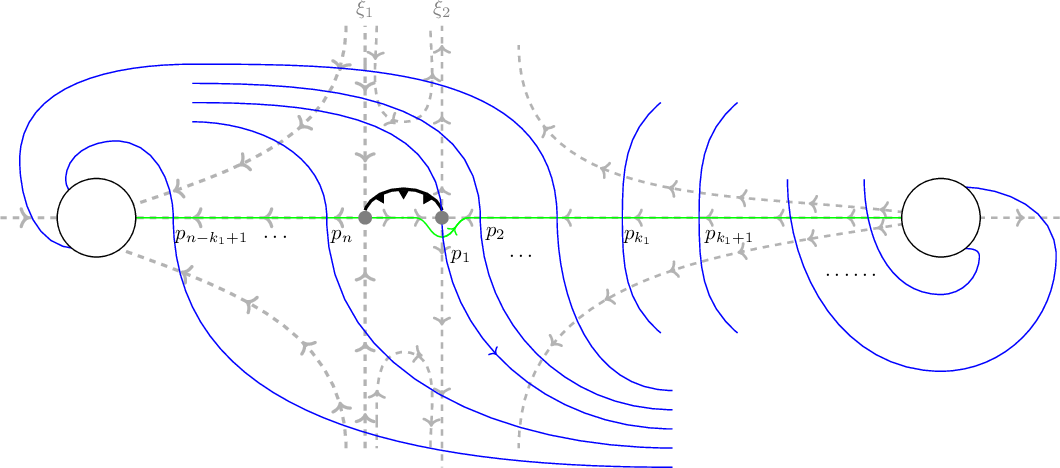}
\caption{When $k$ is odd, we move the base point of $\xi_{1}$ to the right $k_{1}$-times to get this framed Heegaard diagram of $\frmL$, where $k_{1} = \tfrac{k-1}{2}$.}
\label{fig:fL-detail}
\end{figure}

In Figure \ref{fig:fL-detail}, we label the $I$-points as follows: From left to right, the $I$-points to the left of $\xi_{1}$ are labeled by $p_{n-k_{1}+1}$, $p_{n-k_{1}+2}$, ..., $p_{n}$, and those to the right of $\xi_{1}$ are labeled by $p_{1}$, $p_{2}$, ..., $p_{n-k_{1}}$. We take $p_{1}$ to be the starting point of the lower curve, which is oriented towards right at $p_{1}$; and we take the base point of $\xi_{2}$ to be the starting point of the upper curve, at which the upper curve is oriented downwards. Note that by design, when $k=1$, we understand Figure \ref{fig:fL-detail} as having no $I$-point between $\xi_1$ and the left attaching circle, as $k_1 = 0$. We will show in Lemma \ref{lem:fL-diag} that Figure \ref{fig:fL-detail}, and hence Figure \ref{fig:fL-no-move}, is a framed Heegaard diagram.

\subsubsection{The diagram framing \texorpdfstring{$\frm_R$}{} for odd \texorpdfstring{$(n-k)$}{}}
\label{subsubsec:fR-pic}
When $n-k$ is odd, let $k_{0} := \tfrac{n-k-1}{2}$. There is another way to put a 2-combing on the Heegaard diagram (Figure \ref{fig:HD-Lnk}) of $L(n,k)$, depicted in Figure \ref{fig:fR-no-move}, which has the same color scheme as Figure \ref{fig:fL-no-move}. We will show later that Figure \ref{fig:fR-no-move} is a framed Heegaard diagram, and we denote the corresponding diagram framing by $\frmR$. We point out some notable differences between Figure \ref{fig:fR-no-move} and Figure \ref{fig:fL-no-move}: Firstly, in Figure \ref{fig:fR-no-move}, due to the parity of $(n-k)$, we put $\xi_{2}$ closer to the right attaching circle so that there are $k_{0}$ $I$-points to the right of $\xi_{2}$. Secondly, the twist front now connects $\xi_{1}$ and $\xi_{2}$ through the handle, and the small triangles on the twist front are pointing upwards. 
\begin{figure}[ht]
    \centering
    \includegraphics[width=0.7\textwidth]{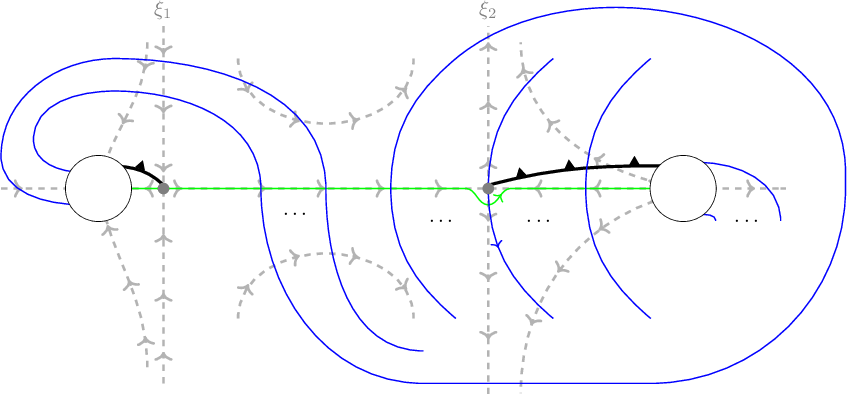}
    \caption{A framed Heegaard diagram on $L(n,k)$ when $(n-k)$ is odd. The diagram framing is denoted by $\frmR$. There are $k_0=\frac{n-k-1}{2}$ $I$-points tot he right of $\xi_2$.}
    \label{fig:fR-no-move}
\end{figure}

As in Figure \ref{fig:fL-no-move}, there are several intersections between the twist front and the upper curve in Figure \ref{fig:fR-no-move}, and we will isotope the upper curve and then perform the base point isotopy move to get an alternative framed Heegaard diagram Figure \ref{fig:fR-detail}, whose associated Kuperberg invariant takes a simpler form. 
We start by pulling $k_0$ strands of the upper curve through the right attaching circle in Figure \ref{fig:fR-no-move}, which is demonstrated in Figure \ref{fig:Lnk-R-move}. 
\begin{figure}[ht]
  \centering
  \includegraphics[width=0.8\textwidth]{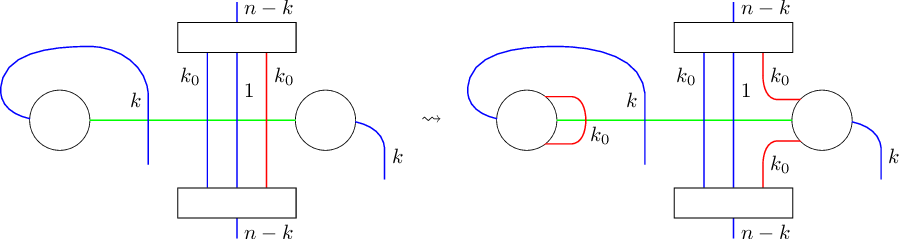}
  \caption{Pulling the right $k_{0}$ strands of the upper curve through the attaching circles.}
  \label{fig:Lnk-R-move}
\end{figure}
More precisely, near the lower curve, Figure \ref{fig:fR-no-move} without the 2-combing is depicted in the left hand side of Figure \ref{fig:Lnk-R-move}. We pull the right $k_0$ strands of the upper curve (colored in red) through the attaching circle on the right, then they will appear from the left attaching circle as illustrated in the right hand side of Figure \ref{fig:Lnk-R-move}.

Now we perform a base point isotopy move (cf.~Figure \ref{fig:base-pt-isotopy}) around the base point $\xi_{1}$ on Figure \ref{fig:fR-no-move}, so that the upper curve moves as described in Figure \ref{fig:Lnk-R-move}. The resulting diagram framing is still denoted by $\frmR$, and is demonstrated in Figure \ref{fig:fR-detail}. 
In Figure \ref{fig:fR-detail}, we label the $I$-points from left to right by $p_{1}$, $p_{2}$, ..., $p_{n}$, and we perturb the lower curve slightly to align $\xi_{2}$ and $p_{n}$ on the same vertical line. 
The $I$-point $p_{1}$ is the  starting point of the lower curve, which is oriented to the right from $p_{1}$; while the starting point of the upper curve is the base point $\xi_{2}$, and we orient the upper curve downwards at $\xi_{2}$. We will show in Proposition \ref{p:S-exponents} that Figure \ref{fig:fR-detail}, and hence Figure \ref{fig:fR-no-move}, is a framed Heegaard diagram.

\begin{figure}[ht]
  \centering
  \includegraphics[width=.7\textwidth]{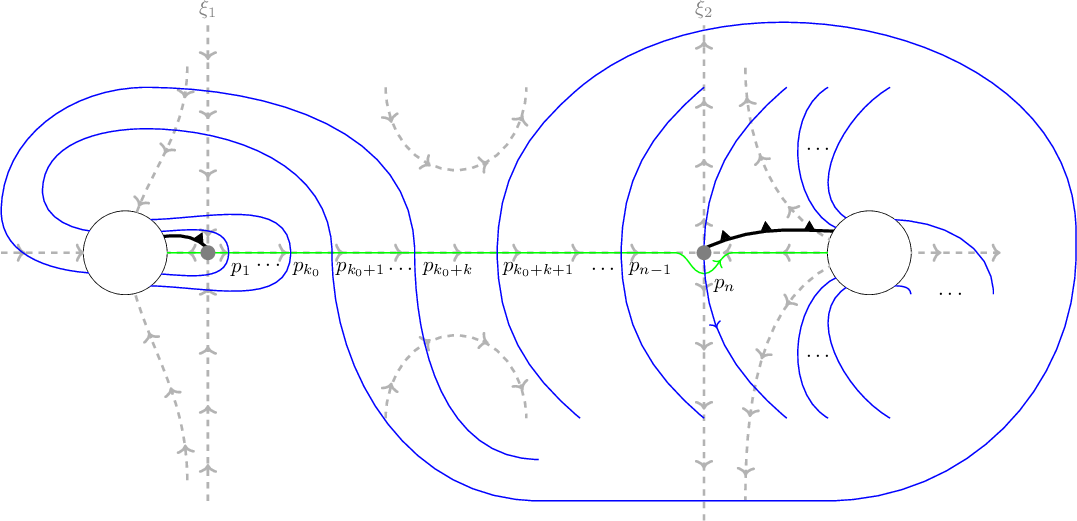}
  \caption{When $(n-k)$ is odd, we perform base point isotopy on Figure \ref{fig:fR-no-move} to get $\frmR$, where $k_{0} = \tfrac{n-k-1}{2}$.}
  \label{fig:fR-detail}
\end{figure}

\begin{remark}\label{rmk:k0}
When $n - k = 1$, $k_0 = 0$, we do not have to conduct any base point isotopy move. In this case, Figure \ref{fig:fR-no-move}, which is the same as Figure \ref{fig:fR-detail}, is already the framed Heegaard diagram for $(L(n,k), \frmR)$.
\end{remark}

\begin{exmp}
Following the above instructions, the diagram framing $\frm_L$ on $L(8, 3)$ is depicted as
\[
\includegraphics[width=0.6\textwidth]{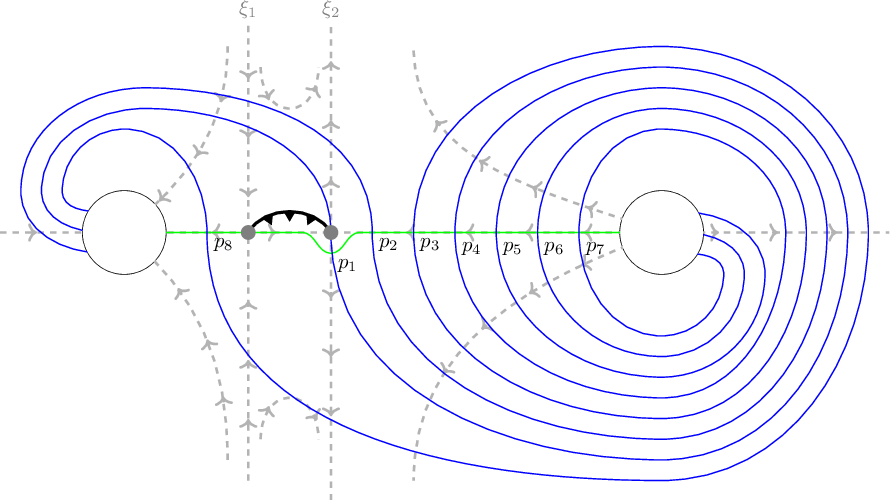}
\]
The rotation numbers of the framed Heegaard diagram of $\frmL$ are recorded in the following table.
\begin{center}
\begin{tabular}{|c|c|c|c|c|c|c|c|c|c|}
\hline
& $p_1$  & $p_2$ & $p_3$ & $p_4$ & $p_5$  & $p_6$ & $p_7$ & $p_8$ & total \\
\hline 
$\theta_\eta$ & $\frac{1}{4}$  & $\frac{1}{2}$ & $\frac{1}{2}$ & $\frac{1}{2}$ & $\frac{1}{2}$  & $\frac{1}{2}$ & $\frac{1}{2}$ & $\frac{1}{2}$  & $\frac{1}{2}$ \\
\hline 
$\theta_\mu$ & $0$ & $-\tfrac{3}{4}$ & $\tfrac{1}{4}$ & $\tfrac{1}{4}$ & $-\tfrac{3}{4}$ & $\tfrac{1}{4}$ & $\tfrac{1}{4}$ & $-\tfrac{3}{4}$ & $-\frac{1}{2}$\\
\hline
$S$-term  & $S$ & $S^3$ & $S$ & $S$ & $S^3$ & $S$ & $S$ & $S^3$ &  \tabularnewline
\hline
$\phi_\eta$ & 0 & 0 & 0 & 0 & 0 & 0 & 0 & 0 & $\frac{1}{2}$\\
\hline
$\phi_\mu$ & 0 & 0 & 0 & 0 & 0 & 0 & 0 & 0 & $\frac{1}{2}$\\
\hline
$T$-term  & $\id$ & $\id$ & $\id$ & $\id$ & $\id$ & $\id$ & $\id$ & $\id$ & \\
\hline
\end{tabular}\,.
\end{center}
The corresponding Kuperberg invariant for any finite-dimensional Hopf algebra $H$ is of the form (recall \eqref{eq:twisted-ld})
\begin{align*}
K(L(8, 3), \frm_L, H)=&(g\rightharpoonup \ld)(S(\Ld_{(1)})S(\Ld_{(4)})S(\Ld_{(7)})S^3(\Ld_{(2)})S^3(\Ld_{(5)})S^3(\Ld_{(8)})S(\Ld_{(3)})S(\Ld_{(6)}))\\
=&\ld(\Ld_{(6)}\Ld_{(3)}S^2(\Ld_{(8)})S^2(\Ld_{(5)})S^2(\Ld_{(2)})\Ld_{(7)}\Ld_{(4)}\Ld_{(1)})\,.
\end{align*}

The diagram framing $\frm_R$ of $L(8,3)$ is depicted as
\[
\includegraphics[width=0.6\textwidth]{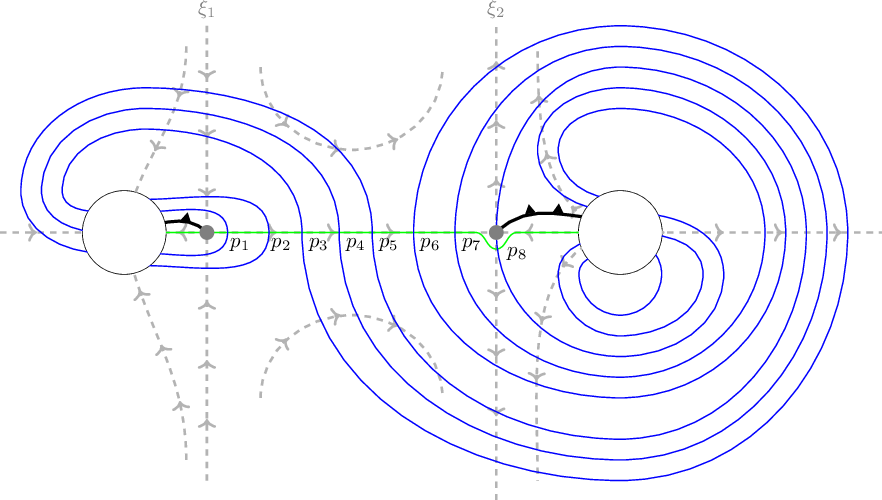}
\]
The rotation numbers of the framed Heegaard diagram for diagram $\frm_R$ are recorded in the following table.
\begin{center}
\begin{tabular}{|c|c|c|c|c|c|c|c|c|c|}
\hline
& $p_1$  & $p_2$ & $p_3$ & $p_4$ & $p_5$  & $p_6$ & $p_7$ & $p_8$ & total \\
\hline 
$\theta_\eta$ & $0$  & $0$ & $0$ & $0$ & $0$  & $0$ & $0$ & $\frac{1}{4}$  & $\frac{1}{2}$ \\
\hline 
$\theta_\mu$ & $\tfrac{3}{4}$ & $\tfrac{3}{4}$ & $-\tfrac{1}{4}$ & $-\tfrac{1}{4}$ & $-\tfrac{1}{4}$ & $\tfrac{3}{4}$ & $\tfrac{3}{4}$ & $0$ & $\frac{1}{2}$\\
\hline
$S$-term  & $S^{-1}$ & $S^{-1}$ & $S$ & $S$ & $S$ & $S^{-1}$ & $S^{-1}$ & $S$ &  \tabularnewline
\hline
$\phi_\eta$ & 0 & 0 & 0 & 0 & 0 & 0 & 0 & 0 & $\frac{1}{2}$\\
\hline
$\phi_\mu$ & 0 & 0 & 0 & 0 & 0 & 0 & 0 & 0 & $-\frac{1}{2}$\\
\hline
$T$-term  & $\id$ & $\id$ & $\id$ & $\id$ & $\id$ & $\id$ & $\id$ & $\id$ & \\
\hline
\end{tabular}\,.
\end{center}
Since $\ld_{-\frac{1}{2}}=\ld$, the corresponding Kuperberg invariant is of the following form:
\begin{equation}\label{eq:L83R}
\begin{split}
K(L(8, 3), \frm_R, H)=&\ld(S(\Ld_{(8)})S(\Ld_{(3)})S^{-1}(\Ld_{(6)})S^{-1}(\Ld_{(1)})S(\Ld_{(4)})S^{-1}(\Ld_{(7)})S^{-1}(\Ld_{(2)})S(\Ld_{(5)}))\\
=&\ld S\left(\Ld_{(5)}S^{-2}(\Ld_{(2)})S^{-2}(\Ld_{(7)})\Ld_{(4)}S^{-2}(\Ld_{(1)})S^{-2}(\Ld_{(6)})\Ld_{(3)}\Ld_{(8)}\right)\,.
\end{split}
\end{equation}
\end{exmp}

\subsection{Gauge invariance for the diagram framings \texorpdfstring{$\frmR$}{} and \texorpdfstring{$\frmL$}{}}\label{subsec:gauge-Lnk}
\subsubsection{Gauge invariance of \texorpdfstring{$K(L(n,k), \frmR, H)$}{} for odd\texorpdfstring{$(n-k)$}{}}
In this subsection, we assume that $(n-k)$ is odd and prove the gauge invariance of the Kuperberg invariant $K(L(n,k), \frmR, H)$ based on the 2-combed Heegaard diagram in Figure \ref{fig:fR-detail}, which will be denoted by $D$. We will first show that $D$ is a framed Heegaard diagram, and give explicit descriptions of the permutation $\s_D$ and the sequence $\{s_i\}_{i=1}^{n}$ in the definition of the Kuperberg invariant (see \eqref{eq:PMDH}, \eqref{eq:sr-tr} and \eqref{eq:KInv}).

\begin{convention}\label{conv:k0}
For the rest of the paper, we adopt the following conventions.
\begin{itemize}
\item An interval of the form $[a, 0]$ with $a>0$ is understood as the empty set. 
\item In tensor components of the form $\bigotimes_{i=1}^{k_0}$, when $k_0 = 0$, the corresponding tensor component is understood as an empty tensor.  
\end{itemize}
\end{convention}

\begin{prop} \label{p:S-exponents}
Let $n > k > 0$ be a pair of coprime integers such that $(n-k)$ is odd, and denote $k_0 = \tfrac{n-k-1}{2}$. Then the 2-combed Heegaard diagram $D$ in Figure \ref{fig:fR-detail} is a framed Heegaard diagram of $L(n,k)$ with $\theta_\mu = \theta_\eta = \phi_\eta = \tfrac{1}{2}$ and $\phi_{\mu} = -\tfrac{1}{2}$. The permutation $\s_D$ is given by $\s_D(i)=\ol {k\cdot (i-1)}$ for $i \in \{1, \dots,  n\}$, where $\ol a$ is the positive residue of $a$ modulo $n$. Moreover, the sequence $\{s_i\}_{i=1}^{n}$ satisfies the following recursive relations with  $s_n=1$:
\begin{enumerate}[label=\rm{(\roman*)}] 
\item If $1\leq i\leq k_0$, then $s_{\ol{i+k}}=s_i+2$;
\item If $k_0 + 1\leq i \leq n-k-1$, then $s_{\ol{i+k}}=s_i-2$;
\item If $n-k \leq i \leq n$, then $s_{\ol{i+k}}=s_i$. 
\end{enumerate}
In particular, $s_n=s_{n-k}=s_k=1$, $s_i$ is an odd integer for all $1 \le i \le n$, and these recursive relations completely determine $\{s_i\}_{i=1}^{n}$. 
\end{prop}

\begin{proof}
By definition, to show that $D$ is a framed Heegaard diagram, it suffices to count the rotation numbers in Figure \ref{fig:fR-detail} and show that $\theta_\eta = \phi_\eta$ for the lower curve $\eta$ (green), and $\theta_\mu = -\phi_\mu$ for the upper curve $\mu$ (blue). 

We first analyze the rotation numbers of the lower curve $\eta$. Start from $\xi_1$ and go along $\eta$ to the right, it is easy to see that for $1 \le i \le n-1$, we have $\theta_{\eta}(p_i) = 0$, and $\theta_{\eta}(p_n) = \tfrac{1}{4}$. 
As we go through the handle and come back to $\xi_1$, the combing $b_1$ and $\eta$ are in opposite direction, so the total rotation number $\theta_\eta$ is equal to $\tfrac{1}{2}$. 
By design, there is no intersection between the twist front and $\eta$, and the direction of the twist front aligns with that of $\eta$ near $\xi_1$. Thus, we have $\phi_\eta = \tfrac{1}{2} = \theta_\eta$.

Now we study the rotation numbers of the upper curve $\mu$. We will compute $\theta_\mu(p_i)$ at the $I$-points, while finding patterns for $\s_D$ and $s_i = 2(\theta_\eta(p_i) - \theta_\mu(p_i))+\tfrac{1}{2}$. Reading directly from $D$, we find $p^1=p_n$ or $\s_D(1)=n$. Moreover, Observation \ref{obs:pts} implies that when traveling along $\mu$, we always go from $p_{i}$ to $p_{\overline{i+k}}$. Thus, $\s_D(i) = \ol{k\cdot (i-1)}$. At $p_n$, we have $\theta_{\mu}(p_n) = 0$. So we have 
\[s_n = 2(\theta_{\eta}(p_n) - \theta_{\mu}(p_n))+\frac{1}{2} = 2\cdot (\frac{1}{4}-0)+\frac{1}{2} = 1\,.\]
We have the following cases when computing $\theta_\mu(p_i)$ for $1 \le i \le n-1$. 

Starting from $p_n$, the next $I$-point we encounter is $p_{\ol{n+k}} = p_k$. The local picture is illustrated in Figure \ref{fig:fR-S-power-0} where only the relevant part of $\mu$ is drawn. 
We go downwards from $p_n$ and enter the handle at the attaching circle on the right, during which $\theta_\mu$ increases by $\tfrac{1}{2}$. Then we exit from the left attaching circle to reach $p_k$, and $\theta_\mu$ decreases by $\tfrac{3}{4}$. Therefore, $\theta_\mu(p_k) = \theta_{\mu}(p_n) + \tfrac{1}{2}-\tfrac{3}{4} = -\tfrac{1}{4}$, and we have 
\[s_{\ol{n+k}} = s_{k} = 2(\theta_\eta(p_k) - \theta_\mu(p_k)) + \frac{1}{2} = 2(0-(-\frac{1}{4})) + \frac{1}{2} = 1 = s_n\,.\]

\begin{figure}[ht]
    \centering
    \includegraphics[width=280pt]{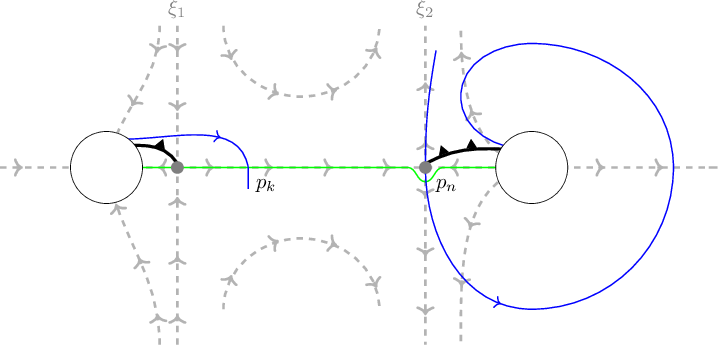}
    \caption{$\theta_\mu(p_{k}) = -\tfrac{1}{4}$.}
    \label{fig:fR-S-power-0}
\end{figure}

For $n-k+1 \leq i \le n-1$, we have $\ol{i+k} = i+k-n$, and we can calculate the change of $\theta_\mu$ using the local configuration shown in Figure \ref{fig:fR-S-power-1}. Precisely, in Figure \ref{fig:fR-S-power-1}, we first go downwards from $p_i$, and then enter the attaching circle on the right. This process increases $\theta_\mu$ by $\tfrac{3}{4}$. Then we exit from the left attaching circle to reach $p_{i+k-n}$, and $\theta_\mu$ decreases by $\tfrac{3}{4}$ in this part of $\mu$. So in this case, we have $\theta_\mu(p_{\ol{i+k}}) = \theta_{\mu}(p_i)$, and $\theta_{\eta}(p_{\ol{i+k}}) = \theta_{\eta}(p_i) = 0$, which means $s_{\ol{i+k}} = s_{i}$. 
\begin{figure}[ht]
    \centering
    \includegraphics[width=280pt]{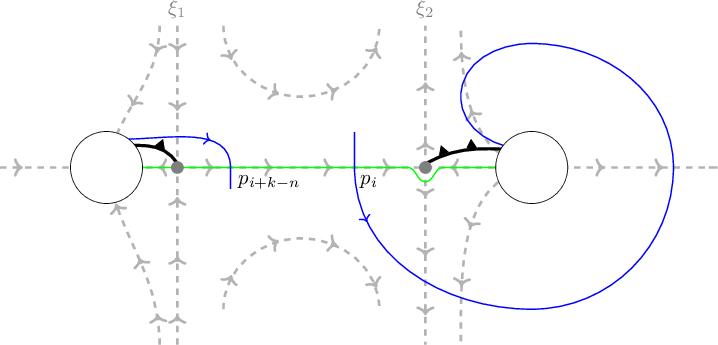}
    \caption{For $n-k+1 \leq i \le n-1$, $\theta_\mu(p_{\ol{i+k}}) = \theta_{\mu}(p_i)$.}
    \label{fig:fR-S-power-1}
\end{figure}

Summarizing the above two cases, we have shown that if $n-k+1\leq i \le n$, then $s_{\ol{i+k}} = s_i$. In particular, if $k = n-1$ or $k_0=0$, then we have $s_n = s_{n-1} = 1$, and $s_{i-1} = s_i$ for all $2 \le i\le n-1$, which means  $s_i = 1$ for all $1 \le i \le n$. Therefore, we have finished the proof of the proposition for $k = n-1$. So for the rest of the proof, we assume that $n-k > 1$, i.e., $k_0 \ge 1$.

For $1\leq i \leq k_0$, we have $\ol{i+k} = i+k$, and we use Figure \ref{fig:fR-S-power-2} to analyze the change of $\theta_\mu$ from $p_i$ to $p_{i+k}$. Start from $p_i$, we go downwards to enter the left attaching circle, and $\theta_\mu$ decreases by $\tfrac{3}{4}$. Then we exit the right attaching circle from the left, and re-enter the handle from right, and $\theta_\mu$ increases by $\tfrac{1}{2}$ in this process. Finally, we exit the left attaching circle from the left to reach $p_{i+k}$, during which $\theta_\mu$ decreases by $\tfrac{3}{4}$. So in this case, we have $\theta_{\mu}(p_{i+k}) = \theta_{\mu}(p_i) -\tfrac{3}{4}+\tfrac{1}{2}-\tfrac{3}{4} = \theta_{\mu}(p_i) - 1$, which implies that 
\[s_{\ol{i+k}} = 2(0-\theta_{\mu}(p_{i+k})) +\frac{1}{2} =  2(0-(\theta_{\mu}(p_{i})-1)) +\frac{1}{2} = s_{i} + 2\,.\]
\begin{figure}[ht]
    \centering
    \includegraphics[width=280pt]{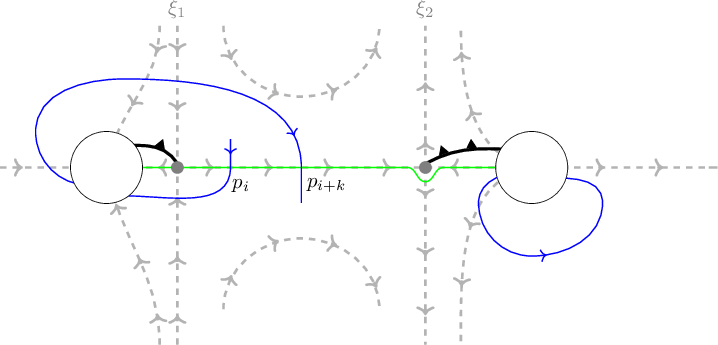}
    \caption{For $1 \le i \le k_0$, $\theta_{\mu}(p_{i+k})= \theta_{\mu}(p_i) - 1$.}
    \label{fig:fR-S-power-2}
\end{figure}

For $k_0+1 \leq i \leq n-k-1$, $\ol{i+k} = i+k$, and the local configuration of $\mu$ is depicted in Figure \ref{fig:fR-S-power-3}, where we go downwards from $p_i$ to reach $p_{i+k}$. Using the similar argument as above, we can easily see that $\theta_\mu(p_{i+k}) = \theta_{\mu}(p_i) + 1$, and so $s_{\ol{i+k}}=s_{i}-2$. 
\begin{figure}[ht]
    \centering
    \includegraphics[width=280pt]{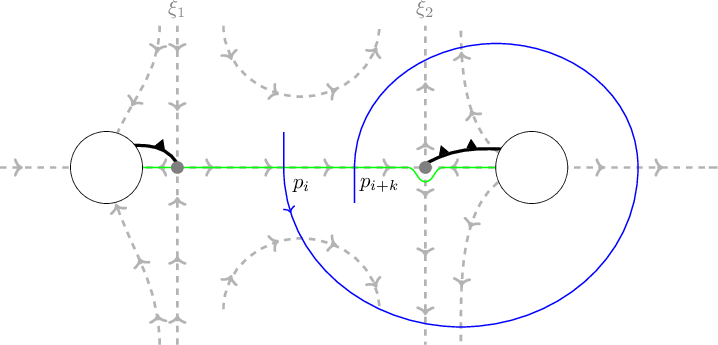}
    \caption{For $k_{0}+1 \le i \le n-k-1$, $\theta_{\mu}(p_{i+k})= \theta_{\mu}(p_i)+1$.}
    \label{fig:fR-S-power-3}
\end{figure}

Now we compute $\theta_{\mu}(p_{n-k})$ and $s_{n-k}$. 
By the expression of $\s_D$, when we start from $p_n$ and travel along $\mu$, the last $I$-point we encounter (before we go back to $\xi_2$) is $p_{n-k}$, as $k$ is assumed to be coprime to $n$. Since $\{\{1, \dots, k_{0}\}, \{k_{0}+1, \dots, n-k-1\}, \{n-k+1, \dots, n\}\}$ is a partition of $\{1, \dots, n\}\setminus \{n-k\}$, the recursive relations above to completely determine $\theta_\mu(p_i)$ and  $s_i$ for all $i \ne n-k$. Consequently, we can compute $\theta_{\mu}(p_{n-k})$ inductively starting from $\theta_\mu(p_k) = -\tfrac{1}{4}$. More precisely, following $\mu$, we go from $p_k$ to $p_{\ol{2k}}$ and all the way to $p_{\ol{n+(n-1)k}} = p_{n-k}$. Note that in this process, there are exactly $k_{0}$ points at which $\theta_\mu$ increases by $1$, and $k_{0}$ points at which $\theta_\mu$ decreases by $1$, and at the rest of the points $\theta_\mu$ remains unchanged, so we can see that 
\[\theta_\mu(p_{n-k}) = \theta_{\mu}(p_k) = -\tfrac{1}{4}\,,\quad\text{and}\quad 
s_{n-k} = s_k = s_n = 1\,.\]
Therefore, the sequence $\{s_i\}_{i=1}^{n}$ satisfies the relations (i)-(iii) in the statement, and it is uniquely determined by these recursive relations.

\begin{figure}[ht]
    \centering
    \includegraphics[width=280pt]{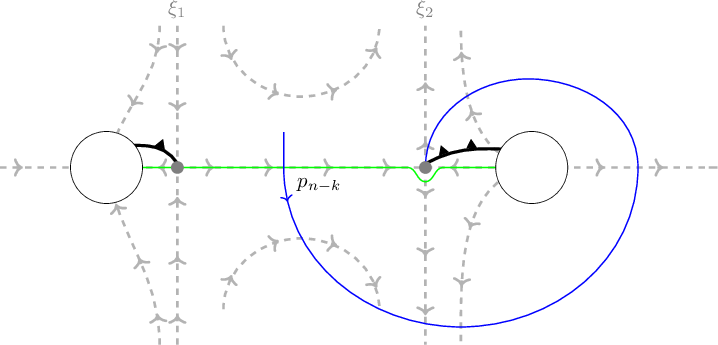}
    \caption{The total rotation number $\theta_\mu = \tfrac{1}{2}$.}
    \label{fig:fR-S-power-4}
\end{figure}
Finally, we use the Figure \ref{fig:fR-S-power-4} to compute the total rotation $\theta_\mu$ when we return to the base point $\xi_2$ following $\mu$. Going from $p_{n-k}$ downwards, we can see that $\theta_\mu$ increases by $\tfrac{3}{4}$, and so $\theta_\mu = \theta_\mu(p_{n-k}) + \tfrac{3}{4} = \tfrac{1}{2}$. By design, there is no intersection between $\mu$ and the twist front, and the direction of the twist front is opposite to the orientation of $\mu$, so $\phi_\mu = -\tfrac{1}{2} = -\theta_\mu$, and we can conclude that $D$ is a framed Heegaard diagram of $L(n,k)$ in this case as desired.
\end{proof}

Recall that the diagram framing associated to the framed Heegaard diagram $D$ in Figure \ref{fig:fR-detail} is denoted by $\frmR$. We  now derive the expression of the Kuperberg invariant $K(L(n,k),\frmR,H)$ for any finite-dimensional Hopf algebra $H$.

\begin{definition}\label{d:cp} Let $n$ be a positive integer and $l$ an integer coprime to $n$. 
\begin{enumerate}
    \item[(a)] We define $\s_{(n,l)} \in \Sym_n$ by $\s_{(n,l)}(i)=\ol{il}$ for $i \in \{1, \dots, n\}$. Since $\s_{(n,l)}(n)=n$, we also consider $\s_{(n,l)} \in \Sym_{n-1}$ by restriction. 
\item[(b)] Let $H$ be a finite-dimensional Hopf algebra, and $n > k >0$ an integer with $(n-k)$ odd. We define the $\kk$-linear operator $P^{(n,-k)}: H \to H$ as 
\[
P^{(n,-k)} := m \circ \s_{(n,-k)}\circ \left(\bigotimes_{i=1}^{n-1}S^{2 c_i}\right)\circ\Delta^{n-1}
\]
where the sequence $\{c_i\}_{i=1}^{n}$ is determined by the following recursive relations with $c_n=0$:
\begin{enumerate}[label=\rm{(\roman*)}] 
    \item If $1\leq i\leq k_0$, then $c_{i+k}=c_i+1$ where $k_0=\frac{n-k-1}{2}$;
    \item If $k_0+1\leq i \leq 2k_0$, then $c_{i+k}=c_i-1$;
     \item If $2k_0+1 \leq i \leq n$, then $c_{\ol{i+k}}=c_i$.
\end{enumerate}
\item[(c)] Finally, we define $\delta(i):=c_{\ol{i+k}} -c_{i}$ for $i=1, \dots, n$.
\end{enumerate}
\end{definition}

\begin{remark}
For any coprime integers $n,k$ with $(n-k)$ odd, it is obvious that the sequences $\{c_i\}_{i=1}^{n}$ above and $\{s_i\}_{i=1}^{n}$ in Proposition \ref{p:S-exponents} are related by $s_i = 2 c_i+1$ for $i=1, \dots, n$. In particular, $c_n = c_{n-k} =0$.
\end{remark}

\begin{thm}\label{t:KInd}
Let $H$ be any finite-dimensional Hopf algebra. For any coprime integers $n>k>0$ such that $(n-k)$ is odd, we have
\begin{align*}
K(L(n, k), \frm_R, H)=\Tr(S \circ P^{(n,-k)})\,.
\end{align*}
\end{thm}
\begin{proof}
By \eqref{eq:KInv}, the Kuperberg invariant is given by
\begin{equation*}
   K(L(n, k), \frm_R, H)=\ld_{-\theta(\mu)}\left(m\circ \s_D\circ \left(\bigotimes_{i=1}^nS^{s_i}\right)\circ\Delta^{n}(\Ld_{\theta(\eta)})\right)
\end{equation*} 
where $\ld \in \int^r_{H^*}$ and $\Ld \in \int^l_{H}$ are a pair of normalized integrals. Recall \eqref{eq:twisted-ld} that $\ld_{-\frac{1}{2}}=\ld$ and $\Ld_{\frac{1}{2}}=\Ld$, by Proposition \ref{p:S-exponents}, we have  
\begin{equation}\label{eq:fR-form}
\begin{aligned}
&K(L(n, k), \frm_R, H)
=\ld_{-\frac{1}{2}} \left(m\circ \s_D\circ \left(\bigotimes_{i=1}^n S^{s_i}\right) \circ \Delta^{n}(\Ld_{\frac{1}{2}})\right) \\
=&\ld \left(m\circ \s_D\circ \left(\bigotimes_{i=1}^n S^{s_i}\right) \circ \Delta^{n}(\Ld)\right) 
=\ld\left(S^{s_n}(\Ld_{(n)})S^{s_{k}}(\Ld_{(k)}) \cdots  S^{s_{n-k}}(\Ld_{(n-k)})\right)\\
=&\ld\left( S(\Ld_{(n)}) S\left(S^{2c_{n-k}}(\Ld_{(n-k)}) S^{2c_{\ol{n-2k}}}(\Ld_{(\ol{n-2k})}) \cdots S^{2 c_k}(\Ld_{(k)})\right)\right)\\
=&\ld\left( S(\Ld_{(2)})\left(S \circ P^{(n,-k)}\right)(\Ld_{(1)})\right)
=\Tr(S \circ P^{(n,-k)})\,,
\end{aligned} 
\end{equation}
where the last equality follows from Theorem \ref{t:t2} (v).
\end{proof}

Now we are ready to prove the gauge invariance of $K(L(n,k),\frmR, H)$.

\begin{thm}\label{t:Lnk-fR}
Let $H$ be a finite-dimensional Hopf algebra over a field $\kk$, $F \in H \o H$ a 2-cocycle of $H$, and $H_F$ the Drinfeld twist of $H$. For any coprime integers $n >k> 0$ such that $(n-k)$ is odd, denote by $P_F^{(n,-k)}$ the $\kk$-linear operator for $H_F$ defined in Definition \ref{d:cp}. Then
\[
\Tr(S \circ  P^{(n,-k)}) = \Tr(S_F \circ  P^{(n,-k)}_F)\,.
\]
In particular, 
\[
K(L(n, k), \frm_R, H) =  K(L(n, k), \frm_R, H_F)\,.
\]
\end{thm}
\begin{proof}
Write $F = \ff{i}{1} \o \ff{i}{2} \in H\o H$. Note that $H = H_F$ as $\kk$-algebras, so $P^{(n,-k)}_F \in \End_\kk(H)$. 
Recall from \eqref{eq:S_F} that $S_F(h)=u S(h) u^{-1}$ where $u = \ff{i}{1}S(\ff{i}{2})$. Let  $Q=uS(u^{-1})$, and define $Q_0=1$, $Q_{s+1}=Q_{s}S^{2s}(Q)$ for $s\in \mathbb{Z}$. In particular, $Q_{-1}=S^{-2}(Q^{-1})$. One can show directly by induction on $s \ge 0$ and $s \le -1$ the equalities:
\begin{equation}\label{eq:Qs}
S_F^{2s}(x)=Q_sS^{2s }(x)Q_{s}^{-1}, \, Q_{s+1}^{-1} Q_s = S^{2s}(Q^{-1}),\,  Q_{s-1}^{-1}Q_s =S^{2s-2}(Q)\, \text{ and }\, Q_{s-1}=Q_{s}S^{2s}(Q_{-1})
\end{equation} 
for $s\in \mathbb{Z}$. Now, by the Radford trace formula (Theorem \ref{t:t2} (v)), 
\begin{equation}\label{eq:tr-fr}
\Tr(S_F \circ  P^{(n,-k)}_F) =  \ld\left(S(\Ld_{(2)})  S_F(P^{(n,-k)}_F(\Ld_{(1)})) \right) \\
\end{equation} 
where  the integrals $\ld \in \int^r_{H^*}$ and $\Ld \in \int_H^l$ are normalized, i.e., $\ld(\Ld)=1$. 

We will use Lemma \ref{l:t1}, which is an adaptation from \cite{KMN12}, to simplify the right hand side of \eqref{eq:tr-fr}. More precisely, let $l = n-k$. Applying Lemma \ref{l:t1} for $Y = m \circ \s_{(n,-k)} \circ \bigotimes_{p=1}^{n-1} S_F^{2 c_p}$, we have
\begin{align*}
\Tr(S_F \circ  P^{(n,-k)}_F) =&\ld \circ S\Biggl(\dd{i}{1} \biggl(\prod_{p=1}^{n-1}S_F^{2c_{\ol{lp}}}(\ff{j}{\ol{lp}}\Lda{\ol{lp}}\dd{i}{\ol{lp}+1})\biggl)\ff{j}{n}\Lda{n}\Biggl) \\ 
=&\ld \circ S\Biggl(\prod_{p=1}^{n}S_F^{2c_{\ol{lp}}}\left( S_F^{2\delta(\ol{lp})}(\dd{i}{\ol{l(p-1)+1}}) \ff{j}{\ol{lp}}\Lda{\ol{lp}}\right) \Biggl)
\end{align*}
which follows from $c_{\ol{l(p-1)}} = c_{\ol{lp+k}} = c_{\ol{lp}}+\delta(\ol{lp})$. Note that $\delta(n)=c_n = 0$, which explains the last term of the preceding expression. By Definition \ref{d:cp} and \eqref{eq:Qs}, we have 
\begin{itemize}
\item $\delta(p) = 1$ and $Q_{c_{p}}S^{2c_{p}}(Q) = Q_{c_{p} + 1} = Q_{c_{\ol{p+k}}}$ if $1 \le p \le k_0$;
\item $\delta(p) = -1$ and $Q_{c_{p}}S^{2c_{p}}(Q_{-1}) = Q_{c_{p} - 1} = Q_{c_{\ol{p+k}}}$ if $k_0+1 \le p \le 2k_0$;
\item $\delta(p) = 0$ and $Q_{c_{p}}S^{2c_p}(Q_0) = Q_{c_{p}} = Q_{c_{\ol{p+k}}}$ if $2k_0+1 \le p \le n$,
\end{itemize} 
where $k_0 = \tfrac{n-k-1}{2}$. Therefore, for all $1 \le p \le n$, we have $Q_{c_{p}}S^{2c_{p}}(Q_{\delta(p)}) = Q_{c_{\ol{p+k}}}$. Let $Y' := \ld \circ S \circ m \circ \s_{(n,-k)}$. Then we can rewrite $\Tr(S_F \circ  P^{(n,-k)}_F)$ as 
\begin{align*}
= &\  Y' \Biggl(\bigotimes_{p=1}^n 
S_F^{2c_{p}}\!\!\left(\! S_F^{2\delta(p)}(\dd{i}{\ol{p+k+1}}) \ff{j}{p}\Lda{p}\!\right)\!\!
\Biggl) =Y' \Biggl(\bigotimes_{p=1}^{n}
Q_{c_{p}} S^{2c_{p}}(Q_{\delta(p)}S^{2\delta(p)}(\dd{i}{\ol{p+k+1}}) Q^{-1}_{\delta(p)}\ff{j}{p}\Lda{p})Q_{c_p}^{-1} \!\!\Biggl)\\ 
=&\ Y' \Biggl(\bigotimes_{p=1}^{n}
Q_{c_{\ol{p+k}}} S^{2c_{p}}(S^{2\delta(p)}(\dd{i}{\ol{p+k+1}}) Q^{-1}_{\delta(p)}\ff{j}{p}\Lda{p})Q_{c_p}^{-1} \Biggl)\\
=&\ \ld S \Biggl(\prod_{p=1}^{n} Q_{c_{\ol{l(p-1)}}} S^{2c_{\ol{lp}}}(S^{2\delta(\ol{lp})}(\dd{i}{\ol{l(p-1)+1}}) Q^{-1}_{\delta(\ol{lp})}\ff{j}{\ol{lp}}\Lda{\ol{lp}})Q_{c_{\ol{lp}}}^{-1}\Biggl)\\
=&\ \ld S \Biggl(\prod_{p=1}^{n}  S^{2c_{\ol{lp}}}(S^{2\delta(\ol{lp})}(\dd{i}{\ol{l(p-1)+1}}) Q^{-1}_{\delta(\ol{lp})}\ff{j}{\ol{lp}}\Lda{\ol{lp}})\Biggl) \quad\text{since $Q_{c_{\ol{0}}} = Q_0 = 1$}\\
=&\ \ld S\Biggl(\dd{i}{1}\ff{j}{l}\Lda{l}\prod_{p=2}^{n} S^{2c_{\ol{lp}}}\left(S^{2\delta({\ol{lp}})}(\dd{i}{\ol{l(p-1)+1}}) Q^{-1}_{\delta({\ol{lp}})}\ff{j}{{\ol{lp}}}\Lda{{\ol{lp}}}\right)\Biggl)\,.
\end{align*}

Applying Theorem \ref{t:t2} (vi), the last expression becomes
\begin{align*}
&\ld S\Biggl(\Lda{l} \biggl(\prod_{p=2}^{n} S^{2 c_{\ol{l(p-1)}}}\left(S(\dd{i}{1}\ff{j}{l})_{(\ol{l(p-1)})}\right) S^{2c_{\ol{lp}}}\left(S^{2\delta({\ol{lp}})}(\dd{i}{\ol{l(p-1)+1}}) Q^{-1}_{\delta({\ol{lp}})}\ff{j}{{\ol{lp}}}\Lda{{\ol{lp}}}\right)\biggl) \Biggl)\\
=& \ \ld S\Biggl(\Lda{l} \cdot
\prod_{p=2}^{n} S^{2c_{\ol{lp}}}\left(S^{2 \delta(\ol{lp})}\bigl(S(\dd{i}{1}\ff{j}{l})_{(\ol{l(p-1)})}\dd{i}{\ol{l(p-1)+1}}\bigl) Q^{-1}_{\delta({\ol{lp}})}\ff{j}{{\ol{lp}}}\Lda{{\ol{lp}}}\right)
\Biggl)\\
=&\ Y'\Biggl(
\bigotimes_{p=1}^{l-1} S^{2c_{p}}\left(S^{2 \delta(p)}(S(\dd{i}{1}\ff{j}{l})_{(p+k)}\dd{i}{p+k+1}) Q^{-1}_{\delta({p})}\ff{j}{{p}}\Lda{{p}}\right)
\o \Lda{l} \\
&\quad\quad \o \bigotimes_{q=1}^{k} S^{2c_{q+l}}\left(S(\dd{i}{1}\ff{j}{l})_{(q)}\dd{i}{q+1} \ff{j}{{l+q}}\Lda{{l+q}}\right) 
\Biggl)\,.
\end{align*}

By Lemma \ref{lem:trick-3} (i), we have 
\[F_n =  \left(\bigotimes_{p=1}^{l-1} \ff{j}{p}\right) \o (F_{k+1} \De^{k+1}(\ff{j}{l})) \text{ and } F_n^{-1} =  (\De^{k+1}(\dd{i}{1}) F_{k+1}^{-1}) \o \left(\bigotimes_{p=2}^{l}\dd{i}{p}\right)\,.\]
Moreover, since $c_p=c_{\ol{p+k}}$ for $n-k+1 \le p \le n$, $c_q=c_{q+l}$ for $1\le q \le k$ and $\Tr(S_F \circ  P^{(n,-k)}_F) $ equals to 
\begin{align*}
&\ Y'\Biggl(\biggl(\bigotimes_{p=1}^{l-1} S^{2c_{p}}\Bigl(S^{2 \delta(p)}\Bigl(S(\dd{i\,(1)}{1}\ff{j\, (1)}{l})_{(p+k)}\dd{i}{p+1}\Bigl) Q^{-1}_{\delta({p})}\ff{j}{{p}}\Lda{{p}}\biggl)
\o \Lda{l} \\
&\quad\quad \o \biggl(\bigotimes_{q=1}^{k} S^{2c_{q}}\Bigl(S(\dd{i\, (1)}{1}\ff{j\, (1)}{l})_{(q)}\dd{i\, (q+1)}{1}\ff{j\, (q+1)}{l}\Lda{q+l}\biggl) \Biggl) \\
=&\ Y'\Biggl(\biggl(\bigotimes_{p=1}^{l-1} S^{2c_{p}}\Bigl(S^{2 \delta(p)}\Bigl(S(\dd{i\,(l-p)}{1}\ff{j\, (l-p)}{l})\dd{i}{p+1}\Bigl) Q^{-1}_{\delta({p})}\ff{j}{{p}}\Lda{{p}}\biggl)
\o \Lda{l} \\
&\quad\quad \o \biggl(\bigotimes_{q=1}^{k} S^{2c_{q}}\Bigl(S(\dd{i\, (n-q)}{1}\ff{j\, (n-q)}{l})\dd{i\, (n-1+q)}{1}\ff{j\, (n-1+q)}{l}\Lda{l+q}\biggl) \Biggl) \\
= &\ Y'\Biggl(\biggl(\bigotimes_{p=1}^{l-1} S^{2c_{p}}(S^{2 \delta(p)}(S(\ff{j\, (l-p)}{l}) S(\dd{i\,(l-p)}{1})\dd{i}{p+1})
Q^{-1}_{\delta(p)} \ff{j}{p} \Lda{p})\biggl) \o \bigotimes_{p=l}^n S^{2 c_p}(\Lda{p}) \Biggl)\,,
\end{align*}
where the last equation follows from the antipode conditions and $(\id^{\o (n-1)} \o \e)(F_n)=F_{n-1}$. It follows from Lemma \ref{lem:trick-3} (v) that $\bigotimes_{p=1}^{l-1}S(\dd{i\,(l-p)}{1})\dd{i}{p+1}=\bigotimes_{p=1}^{l-1}S(\ff{t}{l-p}) u^{-1}$. Now the last expression equals
\begin{align*}
& Y'\Biggl(
\bigotimes_{p=1}^{l-1} S^{2c_{p}}
\left(S^{2 \delta(p)}(S(\ff{j\, (l-p)}{l}) S(\ff{t}{l-p}) u^{-1})
Q^{-1}_{\delta(p)} \ff{j}{p} \Lda{p}\right) \o 
\bigotimes_{p=l}^n S^{2 c_p}(\Lda{p})  
\Biggl)\\
=&\ Y'\Biggl(
\bigotimes_{p=1}^{l-1} S^{2c_{p}}\left(S^{2 \delta(p)}(S(\ff{t}{l-p}\ff{j\, (l-p)}{l}) u^{-1}) Q^{-1}_{\delta(p)} \ff{j}{p} \Lda{p}\right)
\o 
\bigotimes_{p=l}^n S^{2 c_p}(\Lda{p})  
\Biggl)\\
=&\ Y'\Biggl(
\bigotimes_{p=1}^{l-1} S^{2c_{p}}\left(S^{2 \delta(p)}(S(\ff{i}{2l-p-1})  u^{-1}) Q^{-1}_{\delta(p)} \ff{i}{p} \Lda{p}\right) \o 
\bigotimes_{p=l}^n S^{2 c_p}(\Lda{p}) \Biggl)\\
=&\ Y'\Biggl(
\bigotimes_{p=1}^{k_0} S^{2c_{p}}\left(S^3(\ff{i}{4k_0-p+1})S^2(u^{-1})Q^{-1} \ff{i}{p} \Lda{p}\right) 
\o \bigotimes_{p=k_0+1}^{2k_0} S^{2c_p}\left(S^{-1}(S(\ff{i}{p} )u^{-1} \ff{i}{4k_0-p+1})\Lda{p}\right)\\
&\quad\quad  \o \bigotimes_{p=l}^n S^{2 c_p}(\Lda{p})  \Biggl) 
\end{align*}
since $\delta(p)=1$ for $1 \le p \le k_0$,  $\delta(p)=-1$ for $k_0+1 \le p \le 2k_0$ and $Q_{-1}^{-1}=S^{-2}(Q)$. 
Repeatedly applying Lemma \ref{lem:trick-3} (iii) starting from $p=2k_0$ in the second term of the last expression, we find it equals to 
\begin{align*}
& Y'\Biggl(
\bigotimes_{p=1}^{k_0} S^{2c_{p}}\left(S^3(\ff{i}{2k_0-p+1})S^2(u^{-1})Q^{-1} \ff{i}{p} \Lda{p}\right) \o 
\bigotimes_{p=k_0+1}^{n} S^{2c_p}(\Lda{p}) \Biggl)  \\
=&\ Y'\Biggl(
\bigotimes_{p=1}^{k_0} S^{2c_{p}} \left(S^3(\ff{j}{k_0-p+2})S^2(u^{-1})Q^{-1} \ff{i}{p}\ff{j\,(p)}{1} \Lda{p}\right) \o 
\bigotimes_{p=k_0+1}^{n} S^{2c_p}(\Lda{p}) \Biggl) \text{ by Lemma \ref{lem:trick-3} (i)}\\
=&\ Y'\Biggl(
\bigotimes_{p=1}^{k_0} S^{2c_{p}}\left(S^2(\ff{j\,(p)}{1})S^3(\ff{j}{k_0-p+2})S^2(u^{-1})Q^{-1} \ff{i}{p} \Lda{p}\right) \o \bigotimes_{p=k_0+1}^{n} S^{2c_p}(\Lda{p}) \Biggl) \text{ by Lemma \ref{l:t3}.}
\end{align*}
It follows from the second equality of \eqref{eq:u-f-Sf-3} in Lemma \ref{lem:trick-3} (iv), we have
\begin{align*}
& \Tr(S_F\circ P_F^{(n,-k)}) 
=  Y'\Biggl(\bigotimes_{p=1}^{k_0} S^{2c_{p}}\biggl(S^2(\dd{j}{p} u)S^2(u^{-1})Q^{-1} \ff{i}{p} \Lda{p}\biggl) \o \bigotimes_{p=k_0+1}^{n} S^{2c_p}(\Lda{p})\Biggl) \\
=&\ Y'\Biggl(
\bigotimes_{p=1}^{k_0} S^{2c_{p}}\left(S^2(\dd{j}{p})Q^{-1} \ff{i}{p} \Lda{p}\right) \o \bigotimes_{p=k_0+1}^{n} S^{2c_p}(\Lda{p}) \Biggl) \\
=&\ Y'\Biggl(
\bigotimes_{p=1}^{k_0} S^{2c_{p}}(Q^{-1}_{(p)} \Lda{p}) \o \bigotimes_{p=k_0+1}^{n} S^{2c_p}(\Lda{p}) 
\Biggl) \text{ by \eqref{eq:f-DQ-Q-Sf} of Lemma \ref{lem:trick-3} (v)}\\
=&\ Y'\Biggl(
\bigotimes_{p=1}^{k_0} S^{2c_{p}}\left(S(u)_{(p)} S(\dd{i}{1})_{(p)}\dd{i\,(p)}{2} \Lda{p}\right) \o 
\bigotimes_{p=k_0+1}^{n} S^{2c_p}(\Lda{p})\Biggl) \text{ as $u^{-1} = S(\dd{i}{1})\dd{i}{2}$.}
\end{align*}
So by Lemma \ref{l:t3}, we have 
\begin{align*}
 &\Tr(S_F\circ P_F^{(n,-k)}) 
=\ Y'\Biggl(
\bigotimes_{p=1}^{k_0} S^{2c_{p}}\left(S^2(\dd{i\,(p)}{2})S(u)_{(p)} S(\dd{i}{1})_{(p)} \Lda{p}\right) \o 
\bigotimes_{p=k_0+1}^{n} S^{2c_p}(\Lda{p})\Biggl) \\
=&\ Y'\Biggl(
\bigotimes_{p=1}^{k_0} S^{2c_{p}}\biggl(\Bigl(S^2(\dd{i}{2})S(u) S(\dd{i}{1})\Bigl)_{(p)} \Lda{p}\biggl) \o \bigotimes_{p=k_0+1}^{n} S^{2c_p}(\Lda{p})\Biggl)  \text{ as $\dd{i}{1} u S(\dd{i}{2}) =1$}\\
=&\ Y'\Biggl(\bigotimes_{p=1}^{k_0} S^{2c_{p}}(\Lda{p}) \o \bigotimes_{p=k_0+1}^{n} S^{2c_p}(\Lda{p})\Biggl)\ = \Tr(S \circ P^{(n,-k)}) 
\end{align*}
This completes the proof of that $\Tr(S \circ P^{(n,-k)}) = K(L(n, k), \frm_R, H)$ is a gauge invariant. 
\end{proof}

\subsubsection{Gauge invariance of \texorpdfstring{$K(L(n,k), \frmL, H)$}{} for odd \texorpdfstring{$k$}{}} 
Let $n>k>0$ be coprime integers with $k$ odd. Similar to the previous subsection, we read the rotation numbers from the 2-combed Heegaard diagram in Figure \ref{fig:fL-detail}, show it is a framed Heegaard diagram, and study the gauge invariance of the associated Kuperberg invariants. 

We start with the rotation numbers $\theta$. For the lower curve $\eta$, it is easy to see that $\theta_\eta(p_1) = \tfrac{1}{4}$, $\theta_\eta(p_j) = \tfrac{1}{2}$ for all $2 \le j \le n$, and $\theta_\eta = \tfrac{1}{2}$.

We now obtain the change of rotation number from $p_i$ to the next $I$-point on the upper curve $\mu$. For $i = 1$,  we can see that $\theta_{\mu}(p_1) = 0$, and we have already had $\theta_{\eta} = \tfrac{1}{4}$, so $s_1 = 2(\tfrac{1}{4}-0)+\tfrac{1}{2} = 1$. As we go along $\mu$, the next $I$-point we encounter will be $p_{k+1}$. As is illustrated in Figure \ref{fig:fL-S-power-0}, we have $\theta_{\eta}(p_{k+1})=\tfrac{1}{2}$, and $\theta_{\mu}(p_{k+1})=\tfrac{1}{4}$, so $s_{k+1} = 2(\tfrac{1}{2}-\tfrac{1}{4})+\tfrac{1}{2}=1=s_1$. 
\begin{figure}[ht]
    \centering
    \includegraphics[width=280pt]{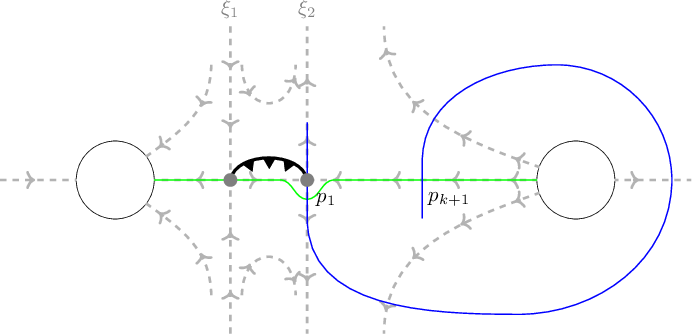}
    \caption{$\theta_\mu(p_1) = 0$, $\theta_{\mu}(p_{k+1}) = \tfrac{1}{4}$.}
    \label{fig:fL-S-power-0}
\end{figure}

For $2\leq i\leq n-k$, the local picture is illustrated in Figure \ref{fig:fL-S-power-1}, and it is immediately seen that $\theta_{\eta}(p_i) = \theta_{\eta}(p_{i+k})$ and $\theta_\mu(p_i) = \theta_\mu(p_{i+k})$, and consequently, $s_{i+k}=s_{i}$. 
\begin{figure}[ht]
    \centering
    \includegraphics[width=280pt]{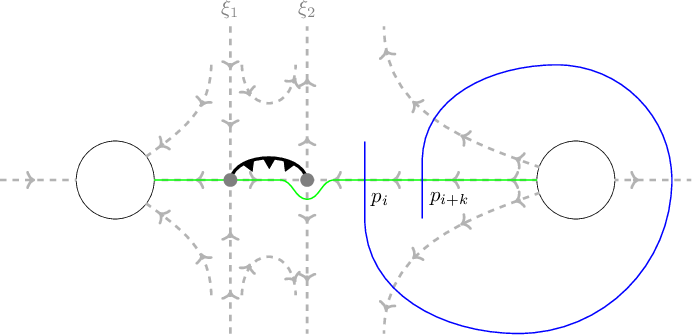}
    \caption{For $1 \le i \le n-k$, $\theta_{\mu}(p_i) = \theta_{\mu}(p_{i+k})$.}
    \label{fig:fL-S-power-1}
\end{figure}
Recall that when $k=1$, there is no $I$-point between $\xi_1$ and the left attaching circle, and by direct computation, we have $\theta_{\mu}(p_{n}) = \tfrac{1}{4} = \theta_{\mu}(p_{i})$ for all $2 \le i \le n-1$, and $\theta_{\eta}(p_n) = \tfrac{1}{2}$. Moreover, the total rotation number of the upper curve is $\theta_\mu=\tfrac{1}{2}$, so in this case, $s_1 = \cdots = s_n$.

Now we assume that $k \ge 3$, and denote $k_1 := \tfrac{k-1}{2}$. For $n-k+2\leq i \leq n-k_1$, the local picture from $p_{i}$ to $p_{\ol{i+k}}$ along $\mu$ is given in Figure \ref{fig:fL-S-power-2}. Reading from the local picture, we have $\theta_{\mu}(p_{\ol{i+k}}) = \theta_{\mu}(p_i)-1$, and $\theta_{\eta}(p_{\ol{i+k}}) = \theta_{\eta}(p_i)$, which implies that $s_{\ol{i+k}}=s_{i}+2$. 
\begin{figure}[ht]
    \centering
    \includegraphics[width=280pt]{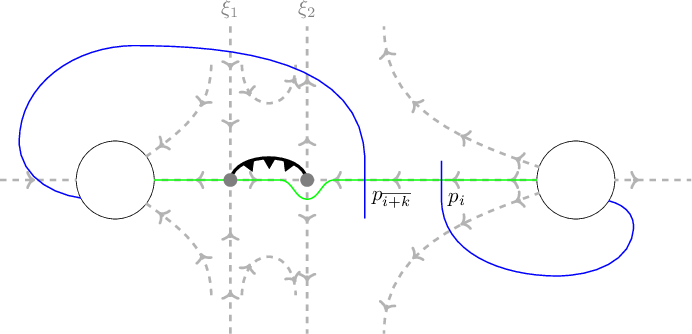}
    \caption{For $n-k+2\leq i \leq n-k_1$, $\theta_{\mu}(p_{\ol{i+k}}) = \theta_{\mu}(p_i)-1$.}
    \label{fig:fL-S-power-2}
\end{figure}

For $n-k_1+1 \leq i \leq n$, the local picture from $p_{i}$ to $p_{\ol{i+k}}$ along $\mu$ is given in Figure \ref{fig:fL-S-power-3}. Reading from it, we have $\theta_{\mu}(p_{\ol{i+k}}) = \theta_{\mu}(p_i)+1$, and $\theta_{\eta}(p_{\ol{i+k}}) = \theta_{\eta}(p_i)$, which implies that $s_{\ol{i+k}}=s_{i}-2$. 
\begin{figure}[ht]
    \centering
    \includegraphics[width=280pt]{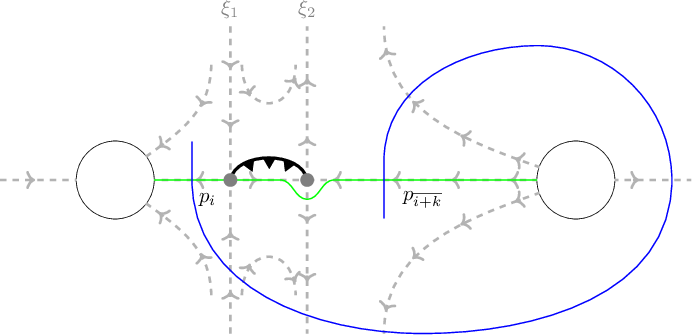}
    \caption{For $n-k_1+1 \leq i \leq n$, $\theta_{\mu}(p_{\ol{i+k}}) = \theta_{\mu}(p_i)+1$.}
    \label{fig:fL-S-power-3}
\end{figure}

Now we are left with $p_{n-k+1}$, which is the last $I$-point we visit on $\mu$ before going back to the base point $\xi_2$. Note that $J = \{\{p_2, ..., p_{n-k}\}, \{p_{n-k+2}, ..., p_{n-k_1}\}, \{p_{n-k_1+1}, ..., n\}\}$ is a partition of $I \setminus \{p_1, p_{n-k+1}\}$. Starting from $p_{k+1}$, we will go  through every point in $J$ and finally reach $p_{n-k+1}$ without passing $p_1$.  Moreover, by the discussions above, upon reaching the next $I$-point, the $\theta_\mu$-value increases by 1 for $k_1$ of the points in $J$, decreases by 1 for $k_1$ of the points in $J$ and remains unchanged for the rest of the points in $J$ (the reason we exclude $p_1$ from $J$ is because it does not follow this pattern). 
Therefore, the total change of $\theta_\mu$-value of the $I$-points when we go from $p_{1+k}$ to $p_{n-k+1}$ is 0, which means $\theta_\mu(p_{n-k+1}) = \theta_{\mu}(p_{1+k}) = \tfrac{1}{4}$, and so $s_{n-k+1} = s_{k+1} = s_1 = 1$. 

After reaching $p_{n-k+1}$, we will follow $\mu$ to go into the right attaching circle, passing through the vertical axis of $\xi_1$ in Figure \ref{fig:fL-detail}, and go back to $\xi_2$. In this process, the $\theta_\mu$-value decreases by $\tfrac{3}{4}$, i.e., the total rotation number $\theta_\mu$ is equal to $-\tfrac{1}{2}$. 

Finally, by design, we have $\phi_\eta(p_i) = \phi_\mu(p_i) = 0$ for all $1 \le i \le n$, and $\phi_\eta = \tfrac{1}{2} = \phi_\mu$. So we can see that $\theta_{\mu} = -\phi_{\mu}$, and $\theta_{\eta} = \phi_{\eta}$. 

We summarize the above discussions in the following lemma:

\begin{lemma}\label{lem:fL-diag}
Let $n>k \ge 1$ be coprime integers such that $k$ is odd. Then the 2-combed Heegaard diagram in Figure \ref{fig:fL-detail} is a framed Heegaard diagram. Moreover, the exponent $s_i$'s are computed recursively with $s_1= 1$ as follows:
\begin{enumerate}[label=\rm{(\roman*)}] 
\item If $1\leq i\leq n-k+1$, then $s_{\ol{i+k}}=s_i$;
\item If $n-k+2\leq i \leq n-\tfrac{k-1}{2}$, then $s_{\ol{i+k}}=s_i+2$;
\item If $n-\tfrac{k-1}{2}+1 \leq i \leq n$, then $s_{\ol{i+k}}=s_i-2$. \qed
\end{enumerate}
\end{lemma}

Recall that the diagram framing associated to Figure \ref{fig:fL-detail} is denoted by $\frmL$.

\begin{prop} \label{p:duality}  
Let $n>k>0$ be  coprime integers such that  $k$ is odd. For any finite-dimensional Hopf algebra $H$, we have 
\[K(L(n, k),  \frm_L, H)=K(L(n, n-k),  \frm_R, H^{\op})\,.\]
\end{prop}
\begin{proof} 
Let $\ld \in \int_{H^*}^r$ and $\Ld \in \int_H^l$ be normalized integrals for $H$, and recall from \eqref{eq:twisted-ld} that $\ld_{\frac{1}{2}}=\ld\circ S^{-1}$ and $\Ld_{\frac{1}{2}}=\Ld$. By definition, the Kuperberg invariant of $L(n,k)$ for the framing $\frm_L$ as
\begin{align*}
	K(L(n, k), \frm_L, H)=\ld_{\frac{1}{2}}\left(m\circ \tau_{n,k}\circ (\otimes_{i=1}^nS^{s_i})\circ\Delta^{(n)}(\Ld_{\frac{1}{2}})\right)=\ld S^{-1}\left(m\circ \tau_{n,k}\circ (\otimes_{i=1}^nS^{s_i})\circ\Delta^{(n)}(\Ld)\right)
\end{align*}
where $\tau_{n,k}$ be the permutation $\tau_{n,k}(i)=\ol{1+(i-1)k}$ for $i \in \{1,\dots, n\}$.                     

Note that $S(\Ld)$ is the left integral of $H^{\op}$ such that $\ld \circ S(\Ld) =1$. Thus, $\ld$ and $S(\Ld)$ are normalized  integrals for $H^{\op}$, and we find
\begin{align*}
K(L(n, k),  \frm_L, H^{\op})
=&\ld S(m^{\op}\circ \tau_{n,k}\circ (\otimes_{i=1}^nS^{-s_i})\circ\Delta^{(n)}(S(\Ld)))\\
=&\ld S(m^{\op}\circ \tau_{n,k}\circ (\otimes_{i=1}^nS^{-s_i+1}(\Ld_{(n+1-i)})))\\
=&\ld S(S^{-s_{n+1-k}+1}(\Ld_{(k)})\cdots S^{-s_{1+k}+1}(\Ld_{(n-k)})S^{-s_1+1}(\Ld_{(n)}))
\end{align*}
Let $\tilde{k}=n-k$, $\tilde{k}_0 =\frac{n-\tilde{k}-1}{2} =\frac{k-1}{2}$, and $\tilde{c}_i=-s_{n+1-i}+1$ for $1\leq i \leq n$. Then $\tilde{c}_n=-s_1+1=0$ and $\tilde{c}_i$'s satisfy the following properties.
\begin{enumerate}[label=\rm{(\roman*)}] 
\item  If $1\leq i\leq \tilde{k}_0$, then $n-\tfrac{k-1}{2}+1 \leq n+1-i \leq n$ and $s_{n+1-i+k-n}=s_{n+1-i}-2$. Thus, we have $\tilde{c}_{i+\tilde{k}}=\tilde{c}_i+2$.

\item If $\tilde{k}_0+1\leq i \leq 2 \tilde{k}_0$, then $n-k+2\leq n+1-i \leq n-\tfrac{k-1}{2}$ and $s_{n+1-i+k-n}=s_{n+1-i}+2$. Thus, $\tilde{c}_{i+\tilde{k}}=\tilde{c}_i-2$.

\item  If $n-\tilde{k}+1 \leq i \leq n$, then $1\leq n+1-i\leq n-k$ and $s_{n+1-i+k}=s_{n+1-i}$. Thus, $\tilde{c}_{\overline{i+\tilde{k}}}=\tilde{c}_i$.
\end{enumerate}
Comparing to the recursive relation of the sequence $\{2c_i\mid 1 \le i \le n\}$ for the pair $(n, \tilde{k})$ in Definition \ref{d:cp}, we can see that $\{\tilde{c}_i\mid 1\le i \le n\}$ satisfies the same recursive relations and initial condition, so $\tilde{c}_i=2c_i$ for all $1 \le i\le n$. Hence, 
\begin{align*}
K(L(n, k),  \frm_L, H^{\op}) 
=& \ld S(S^{\tilde{c}_{n-\tilde{k}}}(\Ld_{(n-\tilde{k})})\cdots S^{\tilde{c}_{\tilde{k}}}(\Ld_{(\tilde{k})})S^{\tilde{c}_n}(\Ld_{(n)}))\\
=& \ld S(S^{2c_{n-\tilde{k}}}(\Ld_{(n-\tilde{k})})\cdots S^{2c_{\tilde{k}}}(\Ld_{(\tilde{k})})S^{2c_n}(\Ld_{(n)}))
= K(L(n,\tilde{k}), \frm_R, H)
\end{align*}
by the fourth equality of \eqref{eq:fR-form} in the proof of Theorem \ref{t:KInd}.
\end{proof}
 Since $K(L(n, n-k), \frm_R, H^{\op})$ is a gauge invariant by Theorem \ref{t:Lnk-fR}, so is $K(L(n, k), \frm_L, H)$.
\begin{thm}\label{t:Lnk-fL}
Let $n>k>0$ be coprime integers such that  $k$ is odd. Then for any finite-dimensional Hopf algebra $H$, the Kuperberg invariant $K(L(n, k), \frm_L, H)$ is a gauge invariant of $H$. \qed
\end{thm}

In \cite{Shi15}, the author used the extended Sweedler power to generalize the indicators defined in \cite{KMN12}:
\[
\nu_n(H) = \ld S(P^{(n)}(\Ld)) \text{ for $n \in \BZ$}\,,
\]
where $P^{(n)}$ is the $n$-th convolution power of the identity of $H$. In particular, for any positive integer $n$ and $x \in H$, $P^{(-n)}(x) = S(x_{(1)}) \cdots S(x_{(n)})$. We can view $\nu_n$ as special cases of the Kuperberg invariants in the following sense. Recall that for any nonnegative integer $n$, the lens space $L(n,n-1)$ is homeomorphic to $L(n,1)$ (cf. \eqref{eq:lens_iso}), then we can pull back $\frmR$ on $L(n,n-1)$ to obtain a framing on $L(n,1)$, denoted by $\frm_n$. When $n<0$, $L(n,1) \cong L(|n|, 1)$. By \cite{Kup96}, we can perform a spiral move to $\frmL$ and pull it back to from $L(|n|,1)$ to $L(n,1)$ to get another framing $\frm_n$ so that $K(L(n,1), \frm_n', H)=\alpha(g) K(L(|n|,1), \frmL, H)$ for all $H$. 

\begin{cor}\label{cor:Kup-and-nu}
For any integer $n$ and  finite-dimensional Hopf algebra $H$, we have 
\[
\nu_n(H) = K(L(n,1),\frm_n,H)=
\begin{cases}
K(L(n,n-1),\frmR,H) & \text{ if $n \ge 0$}\,,\\
\alpha(g) K(L(|n|,1),\frmL,H) & \text{ if $n < 0$}.
\end{cases}
\]
In particular, $\nu_n$ is a gauge invariant.
\end{cor}
\begin{proof}
Recall from \cite{Shi15} that $\nu_{-n}(H)=\alpha(g)\nu_n(H^{\op})$. The equalities follow from the fact that the Kuperberg invariant is an invariant of framed 3-manifolds up to homeomorphism (\cite[Thm.~4.1]{Kup96}), Theorems \ref{t:KInd}, \ref{t:Lnk-fR} and Proposition \ref{p:duality}.
\end{proof}

In view of Proposition \ref{p:duality}, Theorems \ref{t:Lnk-fL} and \ref{t:Lnk-fR}, we define the gauge invariant $\nu_{n,k}(H)$ for any finite-dimensional Hopf algebras $H$ and coprime integers $n > k >0$ as
\[
\nu_{n,k}(H):=\left\{\begin{array}{ll}
     K(L(n,n-k), \frm_R, H) & \text{if $k$ is odd;} \\
     K(L(n,n-k), \frm_L, H^{\op}) &  \text{if $k$ is even}.
\end{array}\right.
\]
Note that $\nu_{n,k}$ generalizes the $n$-th Frobenius-Schur indicator (see Corollary \ref{cor:Kup-and-nu}), namely, we have $\nu_n(H) = \nu_{n,1}(H)$. One could also define
\[\nu'_{n,k}(H):=\left\{\begin{array}{ll}
     K(L(n,k), \frm_L, H) & \text{if $k$ is odd;} \\
     K(L(n,k), \frm_R, H^{\op}) &  \text{if $k$ is even}.
\end{array}\right.\]
These gauge invariants satisfy the relation $\nu'_{n,k}(H) = \nu_{n,k}(H^{\op})$ for any coprime integers $n > k >0$ by Proposition \ref{p:duality}. When $H$ is defined over the complex number $\BC$, it is shown in \cite[Prop.\ 3.6]{Shi15} that $\ol{\nu_{n,1}(H)} = \nu_{n,1}(H^{\op}) = \nu'_{n,1}(H)$, and $\nu_{n,1}(H)$ is a cyclotomic integer. We do not know  whether similar properties hold for other pairs of coprime integers $n$, $k$. 
\begin{question}
Does the following  properties hold for any coprime positive integers $n,k$ for any 
finite-dimensional Hopf algebra $H$ over $\BC$?
\begin{enumerate}[label=\rm{(\roman*)}] 
\item $\ol{\nu_{n,k}(H)} = \nu_{n,k}(H^\op)$.
\item $\nu_{n,k}(H)$ is a cyclotomic integer.
\item $\nu_{n,k}(D(H)) = |\nu_{n,k}(H)|^2$.
\end{enumerate}
\end{question}
Note that (iii) follows from (i), cf.\ Corollary \ref{cor:DH}.

\subsection{Gauge invariance of \texorpdfstring{$K(L(n,k), \frm, H)$}{} with arbitrary framing \texorpdfstring{$\frm$}{}}\label{subsec:gen-frm}
In this subsection, we study the Kuperberg invariants of lens spaces with general framings and prove their gauge invariance. 
In order to do it, we start by giving a brief description of the homotopy classification of framings on a general 3-manifold $M$. 
Once a framing on $M$ is fixed, the set of framings up to homotopy can be identified with $[M, \SO(3)]$, the homotopy set of maps from $M$ to $\SO(3)$. Any map from $M$ to $\SO(3)$ can be viewed as an extension from the 0-skeleton $M^{0}$ of $M$ to $M = M^{3}$. By obstruction theory, up to homotopy, maps from $M^{1}$ to $\SO(3)$ which can be extended to $M^{2}$ are classified by $H^{1}(M, \pi_{1}\SO(3)) = H^{1}(M, \mathbb{Z}/2\mathbb{Z})$. Since $\pi_{2}\SO(3) = 0$, any of the extendable map above can be further extended to $M = M^{3}$, and up to homotopy, the extensions of such a map to the whole manifold $M$ are classified by $H^{3}(M, \pi_{3}\SO(3)) = H^{3}(M, \mathbb{Z})$. Therefore, each map from $M^2$ to $\SO(3)$ has $|H^{3}(M, \mathbb{Z})|$ extensions to $M$ up to homotopy, i.e., $[M, \SO(3)]=H^{1}(M, \mathbb{Z}/2\mathbb{Z}) \times H^{3}(M, \mathbb{Z})$. 
Relative to the fixed choice of framing, we call the $H^{1}(M, \mathbb{Z}/2\mathbb{Z})$-component of any framing its \emph{spin class}, and the $H^{3}(M, \mathbb{Z})$-component its \emph{degree}, see \cite[Sec.\ 2]{Kup96} for details.

Let $H$ be a finite-dimensional Hopf algebra. By \cite{Kup96}, when the degree of a framing $\frm$ of $M$ changes by +1, the corresponding Kuperberg invariant $K(M, \frm, H)$ gets a factor of $\alpha(g)$, which is a gauge invariant. Therefore, if $K(M, \frm, H)$ is a gauge invariant for some framing $\frm$, then $K(M, \frm', H)$ is also a gauge invariant for any framing $\frm'$ that has the same spin class as $\frm$. Consequently, in order to prove the gauge invariance of the Kuperberg invariant associated to an arbitrary framing on $M$, we only need to show that for any spin class, there is a framing $\frm$ on $M$ in that spin class such that $K(M, \frm, H)$ is a gauge invariant.

The spin class of the diagram framing $\frm$ of $M$ associated to a framed Heegaard diagram with combings $b_1$ and $b_2$ can be read from $b_1$ as follows. By the universal coefficient theorem and the Poincar\'e duality, we have the following split short exact sequence 
\[0 \to H^1(M, \BZ)\otimes \BZ/2\BZ \to H^1(M, \BZ/2\BZ) \to \Tor(H^2(M,\BZ), \BZ/2\BZ) \to 0\,,\]
and the spin class of $\frm$ is determined by its components in $H^1(M, \BZ)\ot \BZ/2\BZ$ and $\Tor(H^2(M,\BZ), \BZ/2\BZ)$ (``the torsion part''). In particular, if $M$ satisfies $H^{1}(M, \BZ) \cong H_{2}(M, \BZ) = 0$, then the spin class of $\frm$ is completely determined by its torsion part. In \cite{Kup96}, it is argued that the torsion part of $\frm$ is given by the characteristic class $c \in H^{2}(M, \BZ)$ of $b_1$ (by viewing $b_1$ as a map from $M$ to $S^{2}$). Moreover, the reason for $c \in \Tor(H^2(M,\BZ), \BZ/2\BZ)$ is that $2c$ is the obstruction of the extension of $b_1$ to $\frm$, which has to vanish. 

In this way, we are able to compare the spin classes of two diagram framings of a 3-manifold with vanishing second integral cohomology, and lens spaces can be easily shown to satisfy this condition (for example by the Mayer-Vietoris sequence associated to a genus 1 Heegaard diagram such as the one in Figure \ref{fig:HD-Lnk}). 
This enables us to obtain our main theorem on the gauge invariance of the Kuperberg invariants for lens spaces.

\begin{thm} \label{t:genus_1}
Let $n$, $k$ be coprime integers. For any framing $\frm$ of lens space $L(n, k)$, the Kuperberg invariant $K(L(n,k), \frm, H)$ is a gauge invariant of finite-dimensional Hopf algebras.
\end{thm}
\begin{proof}
Denote $L(n,k)$ by $\BL$ for short. Since $H^1(\BL, \BZ) = 0$, $H^1(\BL,\BZ/2\BZ) \cong \Tor(H^2(\BL,\BZ), \BZ/2\BZ) \cong \Tor(\BZ/n\BZ, \BZ/2\BZ)$ by Poincar\'e duality.  We have the following two cases.

When $n$ is odd, $H^1(\BL, \BZ/2\BZ) = 0$. Therefore, up to degree, any framing $\frm$ on $\BL$ is homotopic to either $\frm_L$ (when $k$ is odd) or $\frm_R$ (when $k$ is even). Thus, up to a power of $\alpha(g)$, $K(\BL, \frm, H)$ is equal to $K(\BL, \frm_L, H)$ or $K(\BL, \frm_R, H)$, and we are done by Theorems \ref{t:Lnk-fR} and \ref{t:Lnk-fL}.

When $n$ is even, then $k$ has to be odd, and $H^1(\BL, \BZ/2\BZ)\cong \BZ/2\BZ$. In light of the above discussions, by Theorems \ref{t:Lnk-fR} and \ref{t:Lnk-fL}, it suffices to show that $\frm_L$ and $\frm_R$ have different spin classes. Denote the 2 combings in the framed Heegaard diagram for $\frmL$ by $b_1$ and $b_2$, and those for $\frmR$ by $b_1'$ and $b_2'$ respectively. To compare the spin classes of $\frm_L$ and $\frm_R$, we only need to compare the characteristic classes $c$, $c' \in H^2(\BL, \BZ)$ of $b_1$ and $b_1'$ respectively. The relative characteristic class $c-c'$ can be presented by its Poincare dual $[C]\in H_1(\BL, \mathbb{Z})$. According to \cite[Lem.\ 2.14]{Les15}, this homology class is represented geometrically by a closed curve $C \subset \BL$ on which $b_1=-b'_1$. Apparently, if $[C] \ne 0$ in $H_1(\BL, \mathbb{Z})$, then $\frmL$ and $\frmR$ have different spin classes in $H^1(\BL, \BZ/2\BZ)$. 

To find the $C$, we compare the framed Heegaard diagrams $D$ (resp.\ $D'$) with combing $b_1$ (resp.\ $b_1'$) in Figure \ref{fig:fL-no-move} (resp.\ Figure \ref{fig:fR-no-move}) for $\frmL$ (resp.\ $\frmR$). Note that $D$ and $D'$ have identical underlying Heegaard diagram without combings. For the reader's convenience, we present the information relevant to our discussions here in Figure \ref{fig:compare} below. Note that by design, our combings $b_1$ and $b_1'$ are completely determined by the positions of the base points up to homotopy, so we only draw the combings near the base points in the figure.
\begin{figure}[ht]
    \centering
    \includegraphics[width=0.7\linewidth]{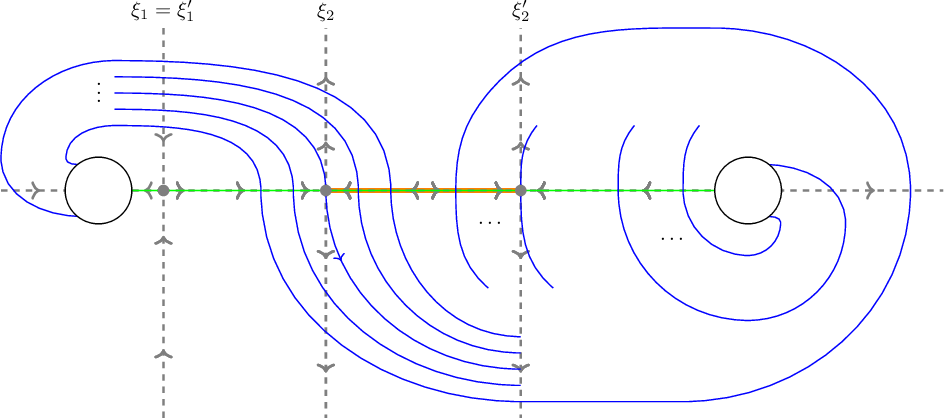}
    \caption{}
    \label{fig:compare}
\end{figure}
In Figure \ref{fig:compare}, we put $b_1$ and $b_1'$ in the same Heegaard diagram, and by construction, we place their left base points (labeled by $\xi_1$ in Figures \ref{fig:fL-no-move} and \ref{fig:fR-no-move}) at the same position, and we write $\xi_1=\xi_1'$ for the common left base point. 

Denote the right base point of $b_1$ by $\xi_2$ and that of $b_1'$ by $\xi_2'$. Then, by construction, there are $\tfrac{k-1}{2}$ $I$-points to the left of $\xi_2$ and $\tfrac{n-k-1}{2}$ $I$-points to the right of $\xi_2'$, which means there are $\tfrac{n}{2}$ $I$-points between $\xi_2$ and $\xi_2'$. Moreover, it can be immediately seen that along the segment of the lower curve $\eta$ connecting $\xi_2$ and $\xi_2'$ (colored in orange in Figure \ref{fig:compare}), $b_1$ is opposite to $b_1'$, so this part of $\eta$ is a segment $C_2$ of the \emph{anti-parallel locus} $C$. As is demonstrated in \cite{Kup96}, the other segment $C_1$ of $C$ lies in the disc $\mathbb{B}$ that the upper curve $\mu$ bounds in $\BL$. We can assume without loss of generality that $C_1$ is a segment in $\mathbb{B}$ connecting $\xi_2$ and $\xi_2'$. We then isotope $C_1$ to align with a segment of the upper curve $\mu$, and still denote that segment of $\mu$ by $C_1$. See the left part of Figure \ref{fig:euler} for an illustration of this situation. In this way, we get a closed curve $C_1 \cup C_2$ representing $[C] \in H_1(\BL, \BZ)$. See the right part of Figure \ref{fig:euler} for a picture in the Heegaard diagram of $\BL$.
\begin{figure}[ht]
    \centering
    \includegraphics[width=0.2\linewidth]{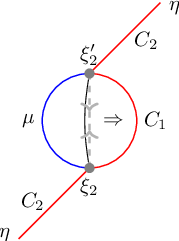}
    \qquad
    \includegraphics[width=0.5\linewidth]{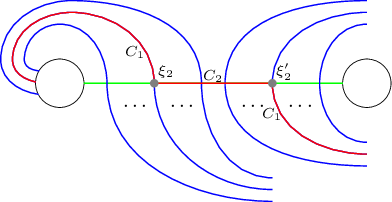}
    \caption{}
    \label{fig:euler}
\end{figure}

Now, to compute the homology class of $C$, we can perform an isotopy to $\BL$ so that $C_2$ is shrunk to a point that identifies $\xi_2$ and $\xi_2'$, and $C_1$ becomes a loop $\gamma$. Then we have $[C] = [\gamma] \in H_1(\BL,\BZ)$. Since $\mu$ bounds a disk in $\BL$, so $\gamma$ cannot bound a disk in $\BL$ because it is only a part of $\mu$. Therefore, $[C] = [\gamma] \ne 0 \in H_1(\BL,\BZ)$. In other words, the two framings $\frmL$ and $\frmR$ have different spin classes.
\end{proof}

\begin{cor}\label{cor:DH}
Let $H$ be a finite-dimensional Hopf algebra over a field $\kk$, and let $D(H)$ be the Drinfeld double of $H$. Let $n, k \in \BZ$ be a pair of coprime integers, and denote $L(n,k)$ by $\BL$.
\begin{enumerate}[label=\rm{(\roman*)}]
\item For any framing $\frm$ of $\BL$, $K(\BL, \frm, D(H)) = K(\BL, \frm, H) \cdot K(\BL, \frm, H^{\op})$.
\item The value $K(\BL, \frm, D(H))$ {is independent of the framing $\frm$ of $\BL$.}
\end{enumerate}
\end{cor}
\begin{proof}
(i) By \cite{DT94}, $D(H)$ is gauge equivalent to $H \ot (H^*)^{\cop} \cong H \ot (H^{\op})^{*}$, so by Theorem \ref{t:genus_1} and the multiplicativity of the Kuperberg invariant \cite[Sec.~5]{Kup96}, we have 
\[K(\BL, \frm, D(H)) = K(\BL, \frm, H) \cdot K(\BL, \frm, (H^{\op})^*) = K(\BL, \frm, H) \cdot K(\BL, \frm, H^{\op})\,,\]
where the second equality follows from the invariance of the Kuperberg invariant under taking duals of Hopf algebras, which is argued in  \cite[Sec.~5]{Kup96}.

(ii) As is argued in \cite{Kup96}, when we fix the spin class of $\frm$ and change its Hopf degree, the resulting Kuperberg invariant differs only by a power of $\alpha(g)$ of $H$. {Since $D(H)$ is unimodular, the Kuperberg invariant is independent of the degree of the framing $\frm$. It remains to show that it is also independent of the spin class of $\frm$. Since $H^1(\BL, \BZ/2\BZ) = 0$ for odd $n$, the framings $\frm$, $\frmL$, $\frmR$ belong to the same spin class of $\BL$ whenever they are well-defined.  In particular, we have
\[
K(\BL, \frm, D(H))=\begin{cases}
    K(L(|n|,\ol{k}), \frmL, D(H)) & \text{if $\ol{k}$ is odd};\\
    K(L(|n|,\ol{k}), \frmR, D(H)) & \text{if $\ol{-k}$ is odd},
\end{cases}
\]
where $\ol{k}$ again denotes the least positive residue of $k$ modulo $n$. 

Now, we assume $n$ is even. Then both $\ol k$ and $\ol {-k}$ are odd. If $\frm$ has the same spin class as $\frmL$, then
\[
K(\BL, \frm, D(H))= K(L(|n|,\ol{k}), \frmL, D(H)) = K(L(|n|,\ol{k}), \frmL, H) K(L(|n|,\ol{-k}), \frmR, H)
\]
by Proposition \ref{p:duality}. On the other hand, if $\frm$ and $\frmR$ are in the same spin class, then 
\[
K(\BL, \frm, D(H))= K(L(|n|,\ol{-k}), \frmR, D(H)) = K(L(|n|,\ol{-k}), \frmR, H) K(L(|n|,\ol{k}), \frmL, H)\,.
\]
Therefore, $K(\BL, \frm, D(H))$ is also independent of the spin class of $\frm$.
}

\end{proof}

\begin{exmp}
Let $n\ge 2$ be an integer. For any primitive $n$-th root of unity $\zeta \in \BC$, let $T(\zeta)$ be the Hopf algebra over $\BC$ with the algebra generated by $x, g$ subject to the relations:
\[
 x^n=0,\, g^n=1,\, gx=\zeta xg\,.
\]
The coalgebra structure and antipode of $T(\zeta)$ are given by
\[\begin{split}
&\De(g)=g\ot g, \,\quad \De(x)=x\ot g + 1 \ot x,\,  \\
&\e(g)=1,\,\, \e(x)=0,\,\, S(g)=g^{-1},\,\, S(x)=-xg^{-1}.
\end{split}\]
This Hopf algebra, known as the Taft algebra \cite{Taft1971}, has dimension $n^2$ and $\{x^i g^j \mid 0 \le i,j \le n-1 \}$ forms a basis for $T(\zeta)$. Moreover,
$\Ld=(\sum_{i=1}^{n} g^i)x^{n-1} \in \int^l_{T(\zeta)}$ and $\ld=\delta_{x^{n-1}} \in \int^r_{T(\zeta)^*}$ form a pair of normalized integral for $T(\zeta)$, i.e., $\ld(\Ld)=1$. 
\begin{enumerate} [label=\rm (\roman{*}), itemsep=3pt, leftmargin=*]
\item
Consider the lens spaces $L(7, 1)$ and $L(7, 2)$. As is explained in the proof of Theorem \ref{t:genus_1}, since $H^1(L(7,1),\BZ/2\BZ)=H^1(L(7,2),\BZ/2\BZ)=0$, the spin classes of all framings of $L(7,1)$ are equal, which means any two framings of $L(7,1)$ are different only by their Hopf degrees. The same is true for $L(7,2)$. In particular, for any finite-dimensional Hopf algebra $H$, the Kuperberg invariant for $L(7, 1)$ with any framing is equal to $K(L(7, 1), \frm_L, H)$ up to a root of unity, and the invariant for $L(7, 2)$ with any framing equals $K(L(7, 2), \frm_R, H)$ up to some root of unity. By Theorem \ref{t:KInd} and Proposition \ref{p:duality}, we have
\begin{align*}
K(L(7, 1), \frm_L, H)=&\lambda(\Lambda_{(7)}\Lambda_{(6)}\Lambda_{(5)}\Lambda_{4}\Lda{3}\Lda{2}\Lda{1})\,,\\
K(L(7, 2), \frm_R, H)=&\lambda S(\Lambda_{(5)}S^2(\Lda{3})\Lda{1}\Lambda_{(6)}S^2(\Lda{4})\Lda{2}\Lambda_{(7)})\,,
\end{align*}
where $\ld \in H^*$, $\Ld \in H$ is a pair of normalized integrals. With the help of the \texttt{GAP} algebra system, we obtain that $K(L(7, 1), \frm_L, T(\zeta_7))= -42\zeta_7-35\zeta_7^2-28\zeta_7^3-21\zeta_7^4-14\zeta_7^5-7\zeta_7^6$, while $K(L(7, 2), \frm_R, T(\zeta_7))=0$, where $\zeta_7$ is a primitive 7-th root of unity. Consequently, the Kuperberg invariants of $L(7, 1)$ and $L(7, 2)$ for any choices of framings cannot be equal to each other. Therefore, the Kuperberg invariant can distinguish $L(7, 1)$ and $L(7, 2)$ up to homeomorphism, although they are homotopic (see, for example, \cite{Rol76}). 
\item
Moreover, the Kuperberg invariant can also distinguish spin classes of framings of a manifold. For example, for any finite-dimensional Hopf algebra $H$, we have
\begin{align*}
&K(L(4, 1), \frmR, H)=\ld S(\Lda{3} S^2(\Lda{2}) \Lda{1}\Lda{4})\\
&K(L(4, 1), \frmL, H)=\ld (\Lda{4}\Lda{3}\Lda{2}\Lda{1})\,,
\end{align*}
where $\ld \in H^*$ and $\Ld \in H$ form a pair of normalized integrals. Let $H = T(i)$, where $i= \sqrt{-1}$. By direct computation, we have $K(L(4,1), \frmR, H)=0$ and $K(L(4,1), \frmL, H)=8(1-i)$. Consequently, $\frmR$ and $\frmL$ cannot be in the same spin class because otherwise, the corresponding Kuperberg invariant could only be different by a multiple of a root of unity. In particular, the discussion here is consistent with the proof of Theorem \ref{t:genus_1}, where we showed that the framings $\frmR$ and $\frmL$ of the lens space $L(n,k)$ have different spin classes for any positive even $n$.
\end{enumerate}
\end{exmp}

\section{Technical Lemmas}\label{sec:technical}
In this section, we provide the proofs for the technical lemmas that we used for our major results. 
\subsection{Lemmas on Hopf algebras}
Let $H$ be a finite-dimensional Hopf algebra over a field $\kk$ with antipode $S$ and counit $\e$. We continue to assume that $\ld \in \int^r_{H^*}$ and $\Ld \in \int^l_H$ are normalized integrals for $H$.

\begin{lemma}\label{l:t3} 
Let $n > k >0$ be coprime integers such that $(n-k)$ is odd, and let $k_0 = \tfrac{n-k-1}{2}$. Let $Y' = \ls \circ m \circ \s_{(n,-k)}$, and $\{c_\ell\}$ the sequence defined in Definition \ref{d:cp}.  For $x, w_\ell\in H$ with $1\leq \ell \leq k_0$, we have 
\begin{equation*}
Y'\Biggl(\bigotimes_{\ell=1}^{k_0}S^{2c_\ell}(w_\ell x_{(\ell)}\Lda{\ell})\o \bigotimes_{\ell=k_0+1}^{n}S^{2c_\ell}(\Lda{\ell})\Biggl)
= Y' \Biggl(\bigotimes_{\ell=1}^{k_0}S^{2c_\ell}(S^2(x_{(\ell)})w_\ell\Lda{\ell}))\o \bigotimes_{\ell=k_0+1}^n S^{2c_\ell}(\Lda{\ell})\Biggl)\,.
\end{equation*}
\end{lemma}
\begin{proof}
Let $Z': H\sot{(n-2)} \to H$ be a $\kk$-linear map such that 
\[Y'(v) =\ld S\Biggl(v_i^{[n-k]} Z'\biggl(\bigotimes\limits_{\substack{\ell=1 \\ \ell \ne n-k}}^{n-1} v_i^{[\ell]}\biggl)v_i^{[n]}\Biggl) = \ld S\bigg(v_i^{[\ol{n-k}]}v_i^{[\ol{n-2k}]}\cdots v_i^{[\ol k]}v_i^{[\ol n]}\biggl)\] 
for all $v \in H\sot{n}$. Then, by Theorem \ref{t:t2} (iii), we have 
\begin{align*}
&Y'\Biggl(\bigotimes_{\ell=1}^{k_0}S^{2c_\ell}(w_\ell x_{(\ell)}\Lda{\ell})\o \bigotimes_{\ell=k_0+1}^{n}S^{2c_\ell}(\Lda{\ell})\Biggl)
\! =\! Y'\Biggl(\bigotimes_{\ell=1}^{k_0}S^{2c_\ell}(w_\ell\Lda{\ell})\o \!\bigotimes_{\ell=k_0+1}^{n}S^{2c_\ell}(S^{-1}(x)_{(\ell-k_0)}\Lda{\ell})\Biggl)\,.
\end{align*}
Since $c_n=c_{n-k}=0$, the right hand side of the preceding equation can be rewritten as follows:
\begin{align*}
&\ld S\Biggl(S^{-1}(x)_{(k_0+1)}\Lda{n-k}Z'\Biggl(\bigotimes_{\ell=1}^{k_0} S^{2 c_\ell}(w_\ell\Lda{\ell})\o\bigotimes_{\substack{\ell=k_0+1 \\ \ell\ne n-k}}^{n-1} S^{2 c_\ell}(S^{-1}(x)_{(\ell-k_0)}\Lda{\ell})\Biggl) S^{-1}(x)_{(n-k_0)}\Lda{n}\Biggl)\\
=&\ \ld S\Biggl(S^{-1}(x_{(k+1)})\Lda{n-k}Z'\Biggl(\bigotimes_{\ell=1}^{k_0} S^{2 c_\ell}(w_\ell\Lda{\ell})\o\bigotimes_{\substack{\ell=k_0+1 \\ \ell\ne n-k}}^{n-1} S^{2 c_\ell}(S^{-1}(x_{(n+1-\ell)})\Lda{\ell})\Biggl) S^{-1}(x_{(1)})\Lda{n}\Biggl)\\
=&\ \ld S  \Biggl(\Lda{n-k} (x_{(k+1)})_{(n-k)} Z'\Biggl(\bigotimes_{\ell=1}^{k_0} S^{2 c_\ell} (w_\ell\Lda{\ell}(x_{(k+1)})_{(\ell)})\o \\
&\qquad\qquad  \bigotimes_{\substack{\ell=k_0+1 \\ \ell\ne n-k}}^{n-1} S^{2 c_\ell}(S^{-1}(x_{(n+1-\ell)})\Lda{\ell} (x_{(k+1)})_{(\ell)})\Biggl)\cdot S^{-1}(x_{(1)})\Lda{n} \Biggl) \quad \text{by Theorem \ref{t:t2} (vi)}\\
=&\ Y' \Biggl(\bigotimes_{\ell=1}^{k_0} S^{2 c_{\ol{\ell+k}}}((x_{(k+1)})_{(\ol{\ell+k})}) S^{2 c_\ell}(w_\ell\Lda{\ell})\o\! \bigotimes_{\ell=k_0+1}^{2k_0} S^{2 c_{\ol{\ell+k}}}((x_{(k+1)})_{(\ol{\ell+k})})S^{2c_\ell}(S^{-1}(x_{(n+1-\ell)})\Lda{\ell} )\\
&\ \o \Lda{n-k} \o \bigotimes_{\ell=2k_0+2}^{n} S^{2c_{\ol{\ell+k}}}((x_{(k+1)})_{(\ol{\ell+k})})S^{2c_{\ell}}(S^{-1}(x_{(n+1-\ell)})\Lda{\ell})  \Biggl)\\
=&\ Y' \Biggl(\bigotimes_{\ell=1}^{k_0} S^{2 c_\ell}\Big(S^{2}((x_{(k+1)})_{(\ol{\ell+k})})w_\ell\Lda{\ell}\Big)\o\! \bigotimes_{\ell=k_0+1}^{2k_0} S^{2c_\ell}\Big(S^{-2}\big((x_{(k+1)})_{(\ol{\ell+k})} S(x_{(n+1-\ell)})\big)\Lda{\ell} \Big) \\
&\qquad \o  \Lda{n-k} \o\!\bigotimes_{\ell=2k_0+2}^{n} S^{2c_\ell}((x_{(k+1)})_{(\ol{\ell+k})}S^{-1}(x_{(n+1-\ell)})\Lda{\ell})  \Biggl)\,.
\end{align*}
To complete the proof, it suffices to show
\[\De^{k_0}(x)  \o 1\sot{(n-k_0-1)} = 
\bigotimes_{\ell=1}^{k_0} (x_{(k+1)})_{(\ol{\ell+k})}   \o \bigotimes_{\substack{\ell=k_0+1\\ \ell \ne n-k}}^n (x_{(k+1)})_{(\ol{\ell+k})}S^{\w(\ell)}
(x_{(n+1-\ell)})\]
where $\w(\ell)=1$ if $k_0+1\le \ell \le 2k_0$ and $\w(\ell)=-1$ if $2k_0+2\le \ell \le n$.  One can rewrite the right hand side as
    \begin{align*}
     & \  \bigotimes_{\ell=k+1}^{k+k_0} (x_{(k+1)})_{(\ell)}  \o \bigotimes_{\ell=k+k_0+1}^{n-1} (x_{(k+1)})_{(\ell)}S(x_{(n+1+k-\ell)}) \o \bigotimes_{\ell=1}^k (x_{(k+1)})_{(\ell)}S^{-1}(x_{(k+1-\ell)}) \\
     =& \  \bigotimes_{\ell=k+1}^{k+k_0} (x_{(2)})_{(\ell)}  \o \bigotimes_{\ell=k+k_0+1}^{n-1} (x_{(2)})_{(\ell)}S(x_{(n+2-\ell)}) \o \bigotimes_{\ell=1}^k (x_{(2)})_{(\ell)}S^{-1}(x_{(1)})_{(\ell)} \\
     =& \  \bigotimes_{\ell=1}^{k_0} (x_{(2)})_{(\ell)}  \o \bigotimes_{\ell=k_0+1}^{n-k-1} (x_{(2)})_{(\ell)}S(x_{(n-k+2-\ell)}) \o \e(x_{(1)})\De^k(1)  \quad\text{ by the antipode condition} \\
     =& \ \De^{k_0} (x_{(1)})  \o \De^{n-k-k_0-1} (x_{(2)}S(x_{(3)})) \o \De^k(1) = \De^{k_0}(x) \o 1^{\o(n-k_0-1)}. 
    \end{align*}
    Here, the second last equality follows from the counit condition and Sweedler's notation, and the last equality is a 
    consequence of the  antipode condition.
\end{proof}

\subsection{Lemmas on 2-cocycles}
We state and prove lemmas on 2-cocycles that are relevant in this paper. Let $F=\ff{i}{1} \o \ff{i}{2}\in H\o H$ be a 2-cocycle. Then, by definition, $F^{-1} = \dd{i}{1} \o \dd{i}{2}  \in H^{\op} \o H^{\op}$ is a 2-cocycle of $H^{\op}$. Recall that $u = \ff{i}{1}S(\ff{i}{2})$ an invertible element in $H$ with $u^{-1} = S(\dd{i}{1})\dd{i}{2}$. The antipode of $H^{\op}$ is $S^{-1}$. Let $\bullet$ denote the multiplication of $H^{\op}$. Then
\begin{equation}\label{eq:op_trick}
    \dd{i}{1} \bullet S^{-1}(\dd{i}{2}) = S^{-1} (\dd{i}{2}) \dd{i}{1} = S^{-1}(u^{-1})\,.
\end{equation}
For elaboration of the properties of these 2-cocycles, we  continue to use these notations and introduce the reverse tensor  and insertion operators $(-)\rev$ and $I_j(-, -)$ as follows. Let $m, n \ge 1$ be integers. For $v= v^{[1]}_i \ot \cdots \ot v^{[n]}_i \in V\sot{n}$, $v\rev : =v^{[n]}_i \ot \cdots \ot v^{[1]}_i$. For  $w \in V^{\ot m}$, we define the $j$-th insertion operator of $w$, $I_j(w, -): V^{\ot n} \to V^{\ot (n+m)}$, as 
 \begin{equation}\label{eq:def-insert}I_j(w,v) := v^{[1]}_i \ot \cdots \ot v^{[j-1]}_i \ot w \ot v^{[j]}_i \ot \cdots \ot v^{[n]}_i \text{ for all $v \in V^{\ot n}$}\,.
 \end{equation}
The cases when $j=1$ or $n+1$ are simply the juxtapositions: $I_1(w,v) = w \ot v$, $I_{n+1}(w,v) = v \ot w$ for any $v \in V^{\ot n}$.

The following lemma is essential to the proof of the gauge invariance of the Kuperberg invariants discussed in the previous section, and its Statement (i) is a generalization of Lemma \ref{lem:F-n-plus-1}, which was proved in \cite{KMN12}. Recall the Convention \ref{conv:k0} and our definition of $F_n$ in \eqref{eq:F-n-plus-1}: $F_1= 1$, $F_2=F$, $F_{n+1}=(1\o F_n)(\id\o \Delta^{n})(F)$ for $n\geq 2$. 
\begin{lemma}\label{lem:trick-3}
Suppose $F_n = f^{[1]}_i  \o \cdots \o f_i^{[n]}$, $F_n^{-1} =  d^{[1]}_i \o \cdots \o d_i^{[n]}$, $u= f^{[1]}_i S(f^{[2]}_i)$, and $Q = u S(u^{-1})$. 
\begin{enumerate}[label={\rm (\roman*)}]
\item 
For any integers $m$, $n \ge 1$, we have
\begin{equation}\label{eq:insert}
F_{m+n} =  I_{j}(F_m, 1^{\ot n}) \cdot (\id^{\o (j-1)} \o \Delta^{m} \o \id^{\o (n+1-j)})(F_{n+1}) \quad \text{for all }j \in \{1, \dots, n+1\}.
\end{equation} 
Equivalently, we have $F^{-1}_{m+n} = (\id^{\o (j-1)} \o \Delta^{m} \o \id^{\o (n+1-j)})(F^{-1}_{n+1}) \cdot I_{j}(F^{-1}_m, 1^{\ot n})$ for all $j \in \{1, \dots, n+1\}$.

\item 
For any $m$, $n \ge 0$, we have $F_{m+n} = (F_m \o F_n)\cdot (\D^{m} \o \D^{n})(F)$. 

\item 
For any $n \ge 2$, and $n-1 \ge m \ge 1$, we have 
\begin{eqnarray}
I_{m}(1, F_{n-2})&=&\left(\bigotimes_{s=1}^{m-1} \ff{i}{s} \right) \o S (\ff{i}{m}) u^{-1} \ff{i}{m+1} \o \left(\bigotimes_{s=m+2}^{n} \ff{i}{s}\right) \,,\label{eq:reduce_f}\\
I_m(1, F^{-1}_{n-2})&=&\left(\bigotimes_{s=1}^{m-1} \dd{i}{s} \right) \o \dd{i}{m} u S(\dd{i}{m+1}) \o \left(\bigotimes_{s=m+2}^{n} \dd{i}{s}\right)\,.\label{eq:reduce_d}
\end{eqnarray}

\item 
For any $n \ge 1$, we have
  \begin{equation}
      1 \o u\sot{n}\! =\! \ff{i}{1} \o \bigotimes_{\ell=2}^{n+1} \ff{i}{\ell}S(\ff{i}{2n+3-\ell}) 
    \ =\  F_{n+1} \cdot(\ff{j (1)}{1} \o \bigotimes_{\ell=2}^{n+1} \ff{j (\ell)}{1} S(\ff{j}{n+3-\ell}))\,, \label{eq:u-f-Sf-1}
    \end{equation}
     \begin{equation}
     1 \o (S^{-1}(u^{-1}))\sot{n} = \dd{i}{1} \o \bigotimes_{\ell=2}^{n+1}S^{-1}(\dd{i}{2n+3-\ell})  \dd{i}{\ell}
    \ =\  (\dd{j (1)}{1} \o \bigotimes_{\ell=2}^{n+1} S^{-1}(\dd{j}{n+3-\ell})\dd{j (\ell)}{1} )F_{n+1}^{-1}\,,\label{eq:u-f-Sf-2}
  \end{equation} 
\begin{equation}
    u\sot{n} =  \bigotimes_{\ell=1}^n \ff{i}{\ell}S(\ff{i}{2n+1-\ell}) = F_n \Big(\bigotimes_{\ell=1}^{n} \ff{j (\ell)}{1} S(\ff{j}{n+2-\ell})\Big)  =\Big(\bigotimes_{\ell=1}^{n} \ff{j}{n+1-\ell} S(\ff{j\,(\ell)}{n+1})\Big) S\sot{n}(F_n)
   \,, \label{eq:u-f-Sf-3} 
\end{equation}
\begin{equation}  \label{eq:u-f-Sf-4}
    (u^{-1})\sot{n} \!= \! \bigotimes_{\ell=1}^n S(\dd{i}{\ell}) \dd{i}{2n+1-\ell} =  S\sot{n}(F_n^{-1})\Big(\!\bigotimes_{\ell=1}^{n}  S(\dd{j (\ell)}{1})\dd{j}{n+2-\ell} \Big)\!=\!\Big(\!\bigotimes_{\ell=1}^{n} S(\dd{j}{n+1-\ell}) \dd{j\,(\ell)}{n+1}\Big)F_n^{-1}.
\end{equation}
\item  For any $n \ge 1$, 
\begin{equation}\label{eq:f-Du-u-Sd}
\D^n(u) = F_n^{-1} \cdot u\sot{n} \cdot (S\sot{n} (F_n^{-1}))\rev\,.
\end{equation}
Consequently, we have 
\begin{equation}\label{eq:Sdd-vd}
\bigotimes_{\ell=1}^{n}S(d^{[1]}_{j\,(n+1-\ell)})d^{[\ell+1]}_{j}
 =\bigotimes_{\ell=1}^{n}S(\ff{j}{n+1-\ell})u^{-1}    
 =
\bigotimes_{\ell=1}^{n} u^{-1}_{(\ell)}d^{[\ell]}_j,
\end{equation}
\begin{equation}\label{eq:fSf-dv}
\bigotimes_{\ell=1}^{n}\ff{j\,(\ell)}{1}S(\ff{j}{n+2-\ell}) = \bigotimes_{\ell=1}^{n} \dd{j}{\ell} u = \bigotimes_{\ell=1}^{n} u_{(\ell)} S(\ff{j}{n+1-\ell}),
\end{equation}
\begin{equation}\label{eq:fSf-vSd}
\bigotimes_{\ell=1}^{n}\ff{j}{n+1-\ell}S(\ff{j\,(\ell)}{n+1}) = \bigotimes_{\ell=1}^{n} uS(\dd{j}{\ell}) = \bigotimes_{\ell=1}^{n} \ff{j}{n+1-\ell} u_{(n+1-\ell)}\,.
\end{equation}

and
\begin{equation}\label{eq:f-DQ-Q-Sf}
F_n \D^{(n)}(Q) = Q\sot{n} \cdot (S^2)\sot{n} (F_n),
\end{equation}
\end{enumerate}
\end{lemma}
 \begin{proof}
 (i) Fix $m \ge 1$. We proceed to prove the \eqref{eq:insert} by induction on $n$. For $n = 1$, and $j=1$ or $2$. When $j = 1$, \eqref{eq:insert} follows from Lemma \ref{lem:F-n-plus-1}, and when $j = 2$, \eqref{eq:insert} follows from the definition of $F_{m+1}$.

Suppose \eqref{eq:insert} holds for all positive integers $k\le n$. For any positive integer $j\le n+1$,  by Lemma \ref{lem:F-n-plus-1}, coassociativity of $\De$ and the induction hypothesis, we have
\begin{align*}
& F_{m+n+1} = (F_{m+n} \o 1) \cdot (\D^{m+n} \o \id)(F) \\
=& \left[\left((1^{\o (j-1)} \o F_m \o 1^{\o (n+1-j)}) \cdot (\id^{\o (j-1)} \o \D^{m} \o \id^{\o (n+1-j)})(F_{n+1})\right) \o 1\right] \cdot (\D^{m+n} \ot \id)(F)\\
=& (1^{\o (j-1)} \o F_m \o 1^{\o (n+2-j)}) \cdot (\id^{\o (j-1)} \o \D^{m} \o \id^{\o (n+2-j)})(F_{n+1} \ot 1) \\
\qquad & \cdot (\id^{\o (j-1)} \o \D^{m} \o \id^{\o (n+2-j)})(\D^{n+1} \ot \id)(F)\\
=& (1^{\o (j-1)} \o F_m \o 1^{\o (n+2-j)}) \cdot (\id^{\o (j-1)} \o \D^{m} \o \id^{\o (n+2-j)})[(F_{n+1} \ot 1) \cdot (\D^{n+1} \ot \id)(F)]\\
=& (1^{\o (j-1)} \o F_m \o 1^{\o (n+2-j)}) \cdot (\id^{\o (j-1)} \o \D^{m} \o \id^{\o (n+2-j)})(F_{n+2})  \,.
\end{align*}
For $j = n+2$, by definition and induction hypothesis, we can perform similar calculation as above
\begin{align*}
F_{m+n+1} =&\ (1 \o F_{m+n}) \cdot (\id \ot \D^{m+n})(F)\\
=&\ (1^{\o (n+1)} \o F_m) \cdot (\id^{\o (n+1)} \o \D^{m})(1\o F_{n+1}) \cdot (\id \o \D^{m+n})(F)\\
=&\ (1^{\o (n+1)} \o F_m) \cdot (\id^{\o (n+1)} \o \D^{m}) (F_{n+2}) \quad\text{by the coassociativity.}
\end{align*}
Therefore, \eqref{eq:insert} holds for $n+1$. 

Statement (ii) follows immediately from (i).

(iii) When $n=2$, $m$ can only be 1, and the statement follows from 
\begin{align*}
S(f_i^{[1]}) u^{-1} f_i^{[2]} 
= 
S(f_i^{[1]}) S(d_j^{[1]}) d_j^{[2]}f_i^{[2]} = m\circ (S \ot \id)(F^{-1} \cdot F) = 1\,.
\end{align*}

For any $x \in H$,  the term $F_0\D^{0}(x)$, if it appears in the following deductions, is considered as an empty tensorand. Now let $n\geq 3$ and apply (ii) repeatedly. Then we have 
\begin{align*}
F_{n}=&\ \Bigl(F_{m-1} \o 1^{\o 2} \o F_{n-m-1}\Bigl)\cdot (\D^{m-1} \o \id^{\o 2} \o \D^{n-m-1})(F_4)\\
=&\ \Bigl(F_{m-1} \o 1^{\o 2} \o F_{n-m-1}\Bigl)\cdot \Bigl((\D^{m-1} \o \id^{\o 2} \o \D^{n-m-1})(\ff{i}{1} \o \ff{j}{1} \ff{i\,(1)}{2} \o \ff{j}{2} \ff{i\,(2)}{2} \o \ff{i}{3})\Bigl)\\
=&\ \Bigl(F_{m-1} \cdot \D^{m-1}(\ff{i}{1})\Bigl) \o \Bigl( \ff{j}{1} \ff{i\,(1)}{2} \o \ff{j}{2} \ff{i\,(2)}{2} \Bigl) \o \Bigl(F_{n-m-1}\cdot \D^{n-m-1}(\ff{i}{3})\Bigl)\,.
\end{align*}
Therefore, 
\begin{align*}
&f_i^{[1]} \o \cdots \o S(f_i^{[m]})u^{-1}f_i^{[m+1]} \o \cdots \o f_i^{[n]}\\
=&\ \Bigl(F_{m-1} \cdot \D^{m-1}(f^{[1]}_i)\Bigl) \o S(\ff{i\,(1)}{2})\ff{i\,(2)}{2} \o \Bigl(F_{n-m-1}\cdot \D^{n-m-1}(\ff{i}{3})\Bigl) \quad\text{by the case $n=2$}\\
=&\ (F_{m-1} \cdot \D^{m-1}(f^{[1]}_i)) \o 1 \o (F_{n-m-1} \cdot \D^{n-m-1}(f^{[2]}_{i})) = I_m(1, F_{n-2})\,.
\end{align*}
The second equality can be proved similarly or using $H^\op$ in the beginning remark and \eqref{eq:op_trick}. 

(iv) We prove the first equality of \eqref{eq:u-f-Sf-1} by induction on $n$. For $n = 1$, by definition and (i), we have 
\[\ff{i}{1} \ot \ff{i}{2}S(\ff{i}{3}) 
= \ff{i}{1} \ot \ff{j}{1}\ff{i(1)}{2}S(\ff{j}{2}\ff{i(2)}{2}) = \ff{i}{1}\e(\ff{i}{2}) \o \ff{j}{1}S(\ff{j}{2}) = 1 \ot u\,.\]
Now assume  first equality of \eqref{eq:u-f-Sf-1} holds for any positive integer $n$. By a similar calculation as above, we have
\begin{align*}
\ff{i}{1} \o \left(\bigotimes_{\ell=2}^{n+2} \ff{i}{\ell}S(\ff{i}{2n+5-\ell})\right) = &\  \ff{i}{1} \o \left(\bigotimes_{\ell=2}^{n+1} \ff{i}{\ell}S(\ff{i}{2n+5-\ell})\right) \o \ff{i}{n+2}S(\ff{i}{n+3})\\
=\ \ff{i}{1} \o \left(\bigotimes_{\ell=2}^{n+1} \ff{i}{\ell}S(\ff{i}{2n+4-\ell})\right) & \o \ff{j}{1} \ff{i(1)}{n+2} S(\ff{j}{2}\ff{i(2)}{n+2}) \quad \text{ by (i)} \\
=\ \ff{i}{1} \o \left(\bigotimes_{\ell=2}^{n+1} \ff{i}{\ell}S(\ff{i}{2n+3-\ell})\right)  &\o u = 1 \o u\sot{n} \o u = 1 \o u^{\o (n+1)}
\end{align*}
by the induction hypothesis. The second equality above is a consequence of  the fact that $F_{2n+1}=(\id^{\o (n+1)} \o \e \o \id\sot{n})(F_{2n+2}) $ by the definition of $F_{2n+1}$ and (i).  The second equality of \eqref{eq:u-f-Sf-1} follows immediately from (ii) that $F_{2n+1} = (F_{n+1} \o 1\sot{n})\cdot (\D^{n+1} \o \id \sot{n})(F_{n+1})$. The first and the second equalities of \eqref{eq:u-f-Sf-3} are obtained by applying $\e \o \id\sot{n}$ to  \eqref{eq:u-f-Sf-1}. The third equality of \eqref{eq:u-f-Sf-3} follows from $F_{2n} = (1\sot{n} \o F_n)(\id\sot{n} \o \Delta^{n})(F_{n+1})$ and the first equality by reversing the tensor.

By the beginning remark of this subsection and \eqref{eq:op_trick}, \eqref{eq:u-f-Sf-3} implies
\begin{align*}
    (S^{-1}(u^{-1}))\sot{n} &=\  \bigotimes_{\ell=1}^{n}S^{-1}(\dd{i}{2n+1-\ell})  \dd{i}{\ell}
    \ =\  \bigg(\bigotimes_{\ell=1}^{n} S^{-1}(\dd{j}{n+2-\ell})\dd{j (\ell)}{1}\bigg) \cdot F_{n}^{-1}\\
    =&\ (S^{-1})^{\o n}(F_n^{-1}) \bigg(\bigotimes_{\ell=1}^n S^{-1}(\dd{j (\ell)}{n+1}) \dd{j}{n+1-\ell}\bigg)\,.
\end{align*}
Equation \eqref{eq:u-f-Sf-4} follows from applying $S\sot{n}$ to these equations.

(v) Note that $\D^n(u) = \bigotimes_{\ell=1}^n \ff{j\,(\ell)}{1}S(\ff{j\, (n+1-\ell)}{2})$. Thus, we have
\begin{align*}
    \D^n(u) \cdot  (S\sot{n}(F_n))\rev \ = &\ 
    \bigotimes_{\ell=1}^n \ff{j\,(\ell)}{1}S\Bigl(\ff{j\, (n+1-\ell)}{2}\Bigl) S\Bigl(\ff{i}{n+1-\ell}\Bigl)\ = \
    \bigotimes_{\ell=1}^n \ff{j\,(\ell)}{1}S\Bigl(\ff{i}{n+1-\ell}\ff{j\, (n+1-\ell)}{2}\Bigl) \\
    = &\ \bigotimes_{\ell=1}^n \ff{j\,(\ell)}{1}S\Bigl(\ff{j}{n+2-\ell}\Bigl) \quad \text{ since } F_{n+1} = (1 \o F_n)\cdot (\id \o \D^n)(F) \text{ by(i)} \\
    = &\ F_n^{-1} u^n \quad \text{ by \eqref{eq:u-f-Sf-3}}\,. 
\end{align*}
This implies \eqref{eq:f-Du-u-Sd}. Now, we take inverse of both sides of Equation \eqref{eq:f-Du-u-Sd} to obtain
\[\D^n(u^{-1}) F_n^{-1} = (S\sot{n}(F_n))\rev \cdot (u^{-1})\sot{n}\,,\]
which is the second equality of \eqref{eq:Sdd-vd}. Apply the reverse tensor operator to \eqref{eq:u-f-Sf-4}, we find
\[(u^{-1})\sot{n} = (S\sot{n}(F_n^{-1}))\rev \cdot \Biggl(\bigotimes_{\ell=1}^n S(\dd{j\ (n+1-\ell)}{1})\dd{j}{\ell+1}\Biggl)\,,\]
and the first equality of \eqref{eq:Sdd-vd} follows. The equations \eqref{eq:fSf-dv} and \eqref{eq:fSf-vSd} are obtained similarly from \eqref{eq:f-Du-u-Sd} and \eqref{eq:u-f-Sf-3}.

Finally, we show how to derive \eqref{eq:f-DQ-Q-Sf} from \eqref{eq:f-Du-u-Sd}. Since $Q = uS(u^{-1})$, we have 
\[\begin{split}
& F_n\cdot \D^{n}(Q) = F_n \cdot\D^{n}(u) \cdot \D^{n}(S(u^{-1})) = u\sot{n} \cdot \bigotimes_{\ell=1}^{n} S(\dd{i}{n+1-\ell}) \cdot \D^{n}(S(u^{-1}))\\ 
=&\ u\sot{n} \cdot S\sot{n} \left(u^{-1}_{(n)}\dd{i}{n} \ot \cdots \ot u^{-1}_{(1)}\dd{i}{1} \right) = u\sot{n} \cdot S\sot{n} \left(S\sot{n}(F_n) \cdot (u^{-1})\sot{n}\right) \\
=&\ Q\sot{n} \cdot (S^2)\sot{n}(F_n)\,,
\end{split}\]
where the second last equality is obtained from applying the reverse tensor operator on both sides of \eqref{eq:f-Du-u-Sd}, and last equality follows from that $S$ is an antihomomorphism of algebras.
\end{proof}

\begin{remark}
When $n = 2$, Equation \eqref{eq:f-Du-u-Sd} of (v) is well-known. See, for example, \cite[(2.17)]{Maj95}, \cite[(5)]{AEGN02} or \cite[Lem.~5.1]{KC17}.
\end{remark}

\begin{lemma}\label{l:t1}
Let $F \in H\sot{2}$ be a 2-cocycle of $H$. For any integer $n>0$ and any $\kk$-linear map $Y: H\sot{(n-1)} \to H$, we have
\begin{align*}
     \ld\left(S(\Ld_{(2)}) S_F\circ Y\circ \De_F^{n-1}(\Ld_{(1)})\right) & = \ld  S\left(m\circ (\id\o Y \o \id)\left((1\o F_n)(1 \o \De^n(\Ld))(F_n^{-1} \o 1) \right)\right) \\
      & = \ld  S\left(\dd{i}{1} Y\left(\ff{j}{1} \Ld_{(1)}\dd{i}{2}  \o \cdots \o (\ff{j}{n-1} \Ld_{(n-1)}\dd{i}{n})\right) \ff{j}{n}\Ld_{(n)}\right) \,.
\end{align*}
\end{lemma}
\begin{proof} Recall that $S_F(h) = u S(h) u^{-1}$ with $u=\ff{j}{1}S(\ff{j}{2})$ and $u^{-1}=S(\dd{i}{1})\dd{i}{2}$. Thus, we have
 \begin{equation*}
\begin{split}
& \ld \left(S(\Ld_{(2)}) S_F(Y\circ\Delta_F^{n-1}(\Ld_{(1)}))\right) = \ld \left(S(\Ld_{(2)}) u \cdot S(Y\circ\Delta_F^n(\Ld_{(1)}))\cdot  u^{-1}\right)\\
= & \ld \left(S(\Ld_{(2)}) \ff{j}{1}S(\ff{j}{2}) \cdot S(Y\circ\Delta_F^{n-1}(\Ld_{(1)})) \cdot S(\dd{i}{1})\dd{i}{2} \right)\\
= & \ls S\left( S^{-1}(\dd{i}{2}) \dd{i}{1} \cdot Y(\Delta_F^{n-1}(\Ld_{(1)})) \ff{j}{2} S^{-1}(\ff{j}{1})\Ld_{(2)} \right)\\
  = & \ld S \left( \dd{i}{1} \cdot Y(\Delta_F^{n-1}(\ff{j}{1}\Ld_{(1)}\dd{i}{2}))\ff{j}{2}\Ld_{(2)} \right)  \text{ by Theorem \ref{t:t2} (iii) and (vi).}
\end{split}
\end{equation*}
 By \eqref{eq:D_F(h)} and Lemma \ref{lem:trick-3} (i), we have 
\begin{equation*}
  \begin{split}
&\ \dd{i}{1} \ot (\Delta_F^{n-1}(\ff{j}{1}\Ld_{(1)}\dd{i}{2})) \ot \ff{j}{2} \Ld_{(2)} =  \dd{i}{1} \ot \biggl(\bigotimes_{m=1}^{n-1}\ff{p}{m} \ff{j\, (m)}{1}\Ld_{(m)} \dd{i\,(m)}{2}\dd{q}{m} \biggl)\ot \ff{j}{2}\Ld_{(n)}\\
= &\ \dd{i}{1} \ot \biggl(\bigotimes_{m=1}^{n-1}\ff{j}{m} \Ld_{(m)} \dd{i}{m+1} \biggl)\ot \ff{j}{n}\Ld_{(n)} = (1 \ot F_n) \left( 1 \ot \Delta^n(\Ld) \right)  \cdot (F_{n}^{-1} \ot 1) 
  \end{split}
\end{equation*}
by Lemma \ref{lem:F-n-plus-1}. The statement follows by applying $\ls\circ m\circ (\id\o Y \o \id)$ to these expressions.
\end{proof}

\section{Gauge invariance of \texorpdfstring{$\tilde\nu_{n, k}(H)$}{}}
In this section, we introduce another gauge invariant $\tilde\nu_{n,k}(H)$ for finite-dimensional Hopf algebras $H$, which are similar to $K(L(n,k), \frm, H)$, but not directly obtained from the Kuperberg invariants of $3$-manifolds. It is unclear whether or how   $\tilde\nu_{n,k}(H)$ and $\nu_{n,k}(H)$ are related. 

\begin{definition}\label{d:shifted_sp} 
Let $H$ be a  finite-dimensional Hopf algebra $H$. For any a coprime integers $n, k$ with $n > 1$, we define the ``shuffled'' Sweedler power 
\[\tP^{(n,k)} = m \circ \s_{(n,k)} \circ \D^{n-1}\,,\] 
and
\[\tnu_{n,k}(H) := \Tr(S \circ \tP^{(n, k)})\,.\]
\end{definition}
Note that 
\begin{align*}
	\tP^{(n,k)}(x)= \sum\limits_{(x)} x_{(\ol{k})}x_{(\ol{2k})}\cdots x_{(\ol{(n-1)k})} \quad\text{ for }x \in H\,.
\end{align*}
In particular, $\tP^{(n,1)} = P^{(n,1)}$, and so $\tnu_{n,1}(H) = \nu_{n,1}(H)$.

It follows from Radford's trace formula (Theorem \ref{t:t2} (v)) that  
\[\tnu_{n, k}(H)
=\sum \ld ( S(\Ld_{(2)}) S(\tP^{(n,k)}(\Ld_{(1)}))) = \sum \ld S(\Ld_{(\ol{k})}\Ld_{(\ol{2k})}\cdots \Ld_{(\ol{(n-1) k})} \Ld_{(n)})\]
for any normalized integrals $\ld \in \int_{H^*}^r$ and $\Ld \in \int_{H}^l$ for $H$. Thus, $\nu_{n,1}(H)=\tnu_{n,1}(H) = \Tr(S \circ P^{(n)})$, which is equal to the higher indicators $\nu^{\rm KMN}_n(H)$ defined in \cite{KMN12}. Therefore, both $\nu_{n,k}(H)$ and $\tnu_{n,k}(H)$ are generalizations of $\nu^{\rm KMN}_n(H)$. We will show that $\tnu_{n,k}(H)$ is also a gauge invariant of $H$  as follows.

\begin{thm}
Let $H$ be a finite-dimensional Hopf algebra over any field, and $F \in H \ot H$ a 2-cocycle of $H$. Then for any pair of coprime integers $n, k$ with $n>1$, we have
\[\tnu_{n,k}(H_F) = \tnu_{n,k}(H)\,.\]
\end{thm}
\begin{proof}
By definition, $\tnu_{n,k}(H) = \tnu_{n, \ol{k}}(H)$ for all $k$ coprime to $n$, and the gauge invariance of $\tnu_{n,1}$ is proved in \cite{KMN12}. Therefore, we assume without loss of generality that $2 \le k <n$.
It follows from Radford's trace formula (Theorem \ref{t:t2} (v)) that
\[\tnu_{n,k}(H_F) = \Tr(S_F \circ \tP^{(n,k)}_F) = \ld(S(\Ld_{(2)})(S_F\circ \tP^{(n,k)}_F(\Ld_{(1)}))) = \ld(S(\Ld_{(2)})(S_F\circ Y(\Delta_F^{n-1}\Ld_{(1)})))\]
for any normalized integrals $\ld \in \int^r_{H^*}$ and $\Ld \in \int^l_H$ for $H$, where $Y = m \circ \s_{(n,k)}$. Let  $F_n= f^{[1]}_j \ot \cdots \ot f^{[n]}_j$ and  $F_n^{-1}= d_i^{[1]} \o \cdots\o d^{[n]}_i$. By Lemma \ref{l:t1}, 
\begin{align*}
&\ \tnu_{n,k}(H_F) \ =\  \ls\Bigl(\dd{i}{1} Y\left(\ff{j}{1} \Ld_{(1)}\dd{i}{2}  \o \cdots \o \ff{j}{n-1} \Ld_{(n-1)}\dd{i}{n}\right) \ff{j}{n}\Ld_{(n)}\Bigl) \\
=&\ \ls  \biggl(\dd{i}{1}\ff{j}{k} \Lda{k} \prod_{p=2}^{n} d^{[\ol{1+(p-1)k}]}_i f^{[\ol{pk}]}_j \Ld_{\left(\ol{pk}\right)}\biggl)\stackrel{(\dagger)}{=}\  \ls \biggl( \Ld_{(k)} \, \prod_{p=2}^{n} S(d^{[1]}_i f^{[k]}_j)_{(\ol{(p-1)k})} d^{[\ol{1+(p-1)k}]}_i f^{[\ol{pk}]}_j \Ld_{(\ol{pk})} \biggl) \\
=&\  \ls \circ Y\biggl(  \bigg(\bigotimes_{p=1}^{k-1} S(d^{[1]}_i f^{[k]}_j)_{(\ol{p-k})} \dd{i}{\ol{p-k}+1} \ff{j}{p} \Lda{p} \bigg)\o \Lda{k}\o \bigotimes_{p=k+1}^{n} S(d^{[1]}_i f^{[k]}_j)_{(p-k)} \dd{i}{p-k+1} \ff{j}{p} \Lda{p} \biggl)\\
=&\  \ls \circ  Y\biggl(\bigl(\bigotimes_{p=1}^{k-1} S(\ff{j\ (k-p)}{k}) S(\ff{i}{k-p})u^{-1} \ff{j}{p} \Lda{p} \bigl)\o \Lda{k}\o \bigotimes_{p=k+1}^{n} S(\ff{j\ (\ol{k-p})}{k}) S(\ff{i}{\ol{k-p}}) u^{-1} \ff{j}{p} \Lda{p}\biggl)
\end{align*}
where the third equality ($\dagger$) follows from Theorem \ref{t:t2} (vi), and the last equality is a consequence of \eqref{eq:Sdd-vd} of Lemma \ref{lem:trick-3} (v). Therefore, we have $\tnu_{n,k}(H_F)$ is equal to
\begin{align*}
&\  \ls \circ  Y\biggl(\bigl(\bigotimes_{p=1}^{k-1}  S(\ff{i}{k-p}\ff{j\ (k-p)}{k})u^{-1} \ff{j}{p} \Lda{p} \bigl)\o \Lda{k}\o \bigotimes_{p=k+1}^{n}  S(\ff{i}{\ol{k-p}}\ff{j\ (\ol{k-p})}{k}) u^{-1} \ff{j}{p} \Lda{p}\biggl) \\
=&\  \ls \circ  Y\biggl(\bigl(\bigotimes_{p=1}^{k-1}  S(\ff{j}{2k-p-1})u^{-1} \ff{j}{p} \Lda{p} \bigl)\o \Lda{k}\o \bigotimes_{p=k+1}^{n}  S(\ff{j}{n+2k-p-1}) u^{-1} \ff{j}{p+n-2} \Lda{p}\biggl)\,.
\end{align*}

Now, we apply \eqref{eq:reduce_f} of Lemma \ref{lem:trick-3} (iii) to the last expression repeatedly starting from $p=k+1$, and  obtain 
\begin{align*}
&\ \tnu_{n,k}(H_F)\  = \ \ls \circ  Y\biggl(\bigotimes_{p=1}^{k-1}  S(\ff{j}{2k-p-1})u^{-1} \ff{j}{p} \Lda{p}  \o \bigotimes_{p=k}^{n}  \Lda{p}\biggl) \\
=& \ \ls \circ  Y\biggl(\bigotimes_{p=1}^{k-1}  S(\ff{j}{k+1-p})u^{-1} \ff{i}{p}\ff{j\, (p)}{1} \Lda{p}  \o \bigotimes_{p=k}^{n}  \Lda{p}\biggl) \text{ by Lemma \ref{lem:trick-3} (i)} \\
=& \ \ls \circ  Y\biggl(\bigotimes_{p=1}^{k-1}  S(\ff{j}{k+1-p})u^{-1} \ff{i}{p}\Lda{p}  \o \bigotimes_{p=k}^{n}  S^{-1}(\ff{j\, (n+1-p)}{1})\Lda{p}\biggl)  \text{ by Theorem \ref{t:t2} (iii).}
\end{align*}
Let $Y': H^{\o (n-2)} \to H$ be the linear map $Y'(w) = Y(I_k(1,w) \o 1)$ for $w \in H\sot{n-2}$, then
\[Y\bigg(\bigotimes_{p=1}^{n} v^{[p]}\bigg) = v_i^{[k]}Y'\bigg(\bigotimes_{p=1}^{k-1} v_i^{[p]} \o \bigotimes_{p=k+1}^{n-1} v_i^{[p]}\bigg) v_i^{[n]}\] 
for all $v\in H\sot{n}$. Then  $\tnu_{n,k}(H_F)$ is equal to
\begin{align*}
&\ \ls \biggl(S^{-1}(\ff{j\,(\ol{1-k})}{1})\Lda{k}Y'\Bigl(\bigotimes_{p=1}^{k-1}  S(\ff{j}{k+1-p})u^{-1} \ff{i}{p}\Lda{p}  \o \bigotimes_{p=k+1}^{n-1}  S^{-1}(\ff{j\, (\ol{1-p})}{1})\Lda{p}\Bigl)S^{-1}(\ff{j\, (1)}{1})\Lda{n}\biggl) \\ 
=& \ \ls \biggl(\Lda{k} \ff{j\, (n)}{1}Y'\Bigl(\bigotimes_{p=1}^{k-1}  S(\ff{j}{k+1-p})u^{-1} \ff{i}{p}\Lda{p} \ff{j\, (n-k+p)}{1} \o \bigotimes_{p=k+1}^{n-1}  S^{-1}(\ff{j\, (\ol{1-p})}{1})\Lda{p}\ff{j\, (n-k+p)}{1}\Bigl)\\
& \qquad\qquad \cdot S^{-1}(\ff{j\, (1)}{1})\Lda{n}\biggl) \text{ by Theorem \ref{t:t2} (vi)}\\
=& \ {\ls\circ Y\biggl(\bigotimes_{p=1}^{k-1}  \ff{j\, (n-k+\ol{p-k})}{1}S(\ff{j}{k+1-p})u^{-1} \ff{i}{p}\Lda{p}  \o \Lda{k}\o \bigotimes_{p=k+1}^{n} \ff{j\,(n-k+\ol{p-k})}{1} S^{-1}(\ff{j\, (n+1-p)}{1})\Lda{p}\biggl)}\\
=& \ \ls\circ Y\biggl(\bigotimes_{p=1}^{k-1}  \ff{j\, (2n-2k+p)}{1}S(\ff{j}{k+1-p})u^{-1} \ff{i}{p}\Lda{p}  \o \Lda{k}\o \bigotimes_{p=k+1}^{n} \ff{j\,(n-2k+p)}{1} S^{-1}(\ff{j\, (n+1-p)}{1})\Lda{p}\biggl)\,.
\end{align*}
By repeatedly applying the antipode condition starting at $p=k+1$, we find
\begin{align*}
 \tnu_{n,k}(H_F) =& \ \ls\circ Y\biggl(\bigotimes_{p=1}^{k-1}  \ff{j\, (p)}{1}S(\ff{j}{k+1-p})u^{-1} \ff{i}{p}\Lda{p} \o \bigotimes_{p=k}^{n} \Lda{p}\biggl) \\
 =& \ \ls\circ Y\biggl(\bigotimes_{p=1}^{k-1}  \dd{j}{p} u u^{-1} \ff{i}{p}\Lda{p} \o \bigotimes_{p=k}^{n} \Lda{p}\biggl), \quad \text{ by \eqref{eq:u-f-Sf-3} of Lemma \ref{lem:trick-3} (iv),}\\
 =&\   \ \ls\circ Y\biggl(\bigotimes_{p=1}^{n}  \Lda{p} \biggl)\, = \,\tnu_{n,k}(H). \qedhere
\end{align*}
\end{proof}

\section{Kuperberg invariants of genus 2 framed 3-manifolds}\label{sec:Kmn}
So far we have been working with 3-manifolds of genus 1. It is then natural to ask whether there exist framed 3-manifolds of higher genera whose Kuperberg invariants are gauge invariant. In this section, we answer these questions by studying a family of 3-manifolds of genus 2, whose fundamental groups are certain central extensions of triangle groups studied by Milnor in his investigation of the Brieskorn manifolds \cite{Mil75}. 

We start with a brief review on the triangle groups. Let $p$, $q$, $r \ge 2$ be integers. The \emph{full (Schwarz) triangle group} $\Theta^{*}(p, q, r)$ admits the following presentation \cite[Thm.~2.2]{Mil75}
\[\Theta^{*}(p, q, r) \cong \langle a, b, c \mid a^{2} = b^{2} = c^{2} = (ab)^{p} = (bc)^{q} = (ca)^{r} = 1 \rangle\,.\]
Geometrically, $\Theta^{*}(p, q, r)$ is the group generated by reflections along the edges of the triangle with interior angles $\pi/p$, $\pi/q$ and $\pi/r$ in $\BP(p,q,r)$, the 2-dimensional simply-connected Riemannian manifold of constant Gauss curvature. Namely, 
\[\BP(p,q,r) = 
\begin{cases}
  \BS^{2} & \text{ if }\ p^{-1}+q^{-1}+r^{-1} > 1\,,\\
  \BH^{2} & \text{ if }\ p^{-1}+q^{-1}+r^{-1} < 1\,,\\
  \BR^{2} & \text{ if }\ p^{-1}+q^{-1}+r^{-1} = 1\,,
\end{cases}\]
where $\BS^{2}$, $\BH^{2}$ and $\BR^{2}$ stand for the unit sphere, the hyperbolic plane and the Euclidean plane respectively. For simplicity, we write $\BP$ for $\BP(p, q, r)$. By \cite[Cor.~2.5]{Mil75}, the index 2 subgroup $\Theta(p,q,r)$ of $\Theta^{*}(p,q,r)$ consisting of all orientation preserving elements of $\Theta^{*}(p,q,r)$ admits the following presentation  
\[\Theta(p, q,r) = \langle u, v, w \mid u^{p} = v^{q} = w^{r} = uvw = 1 \rangle\,. \]
Denote the connected Lie group of orientation preserving isometries of $\BP$ by $\Isom^{+}(\BP)$. Then we have
\[\Isom^{+}(\BP) = 
  \begin{cases}
    \SO(3) & \text{if }\ \BP = \BS^{2}\,,\\
    \PSL(2,\BR) & \text{if }\ \BP = \BH^{2}\,,\\
    \BR^{2} \rtimes \SO(2) & \text{if }\ \BP = \BR^{2}\,.
  \end{cases}
\]
As is listed in \cite[Tbl.~1]{RV81}, the quotient space $X(p,q,r): = \Theta(p,q,r)\backslash\Isom^{+}(\BP)$ is a compact 3-manifold whose fundamental group, called the \emph{centrally extended triangle group} in \cite{Mil75}, admits the following presentation \cite[Lem.~3.1]{Mil75}
\begin{equation}\label{eq:Gamma} 
\Gamma(p,q,r) := \pi_{1}(X(p,q,r)) \cong \langle x, y, z \mid x^{p} = y^{q} = z^{r} = xyz \rangle \,.
\end{equation}
If we denote the universal covering group of $\Isom^{+}(\BP)$ by $\widetilde{\Isom^{+}}(\BP)$, then we have  
\[X(p,q,r) = \Theta(p,q,r)\backslash \Isom^{+}(\BP) \cong \Gamma(p,q,r)\backslash\widetilde{\Isom^{+}}(\BP)\,,\]
and $\Gamma(p,q,r)$ is the inverse image of $\Theta(p,q,r)$ under the covering map $\widetilde{\Isom^{+}}(\BP) \to \Isom^{+}(\BP)$.

In this section, we study the Kuperberg invariant of 
\[\BM_{m,n} := X(2, m+1, n+1)\]
for any pairs of integers $m \ge 1$, $n \ge 1$. Such a manifold can be obtained by (Dehn) surgery on $\BS^3$ along framed links whose planar diagrams are depicted by either (a), (b), (c) or (d) in Figure \ref{fig:Mmn-link}, where the number near each strand represents the number of twists (or framing) on the corresponding strand. Note that since one can get the framed link diagram (b) by the Kirby (or the Fenn-Rourke) move \cite{Kir78, FR79} from (a), and similarly one can obtain (c) from (a) and (d) from (c), we indeed have framed link representations of the same manifold.
\begin{figure}
    \centering
    \includegraphics[width=0.9\linewidth]{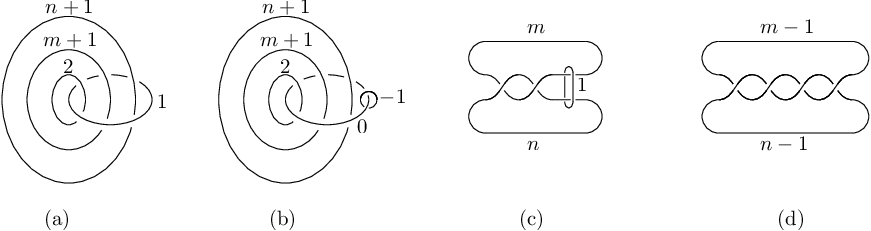}
    \caption{Framed link presentations of $\BM_{m,n}$.}
    \label{fig:Mmn-link}
\end{figure}
\begin{remark}\label{rmk:Mmn}
For completeness, we provide the following remarks for the interested reader.
\begin{enumerate}[label=\rm (\roman{*}), itemsep=3pt, leftmargin=*]
\item 
For nonnegative integers $m$, $n$ such that $mn=0$, the framed link presentations in Figure \ref{fig:Mmn-link} are still valid, and the corresponding 3-manifold can be described as follows. Without loss of generality, assume $n=0$, then in Figure \ref{fig:Mmn-link}(a), the outer vertical $(n+1)$-framed unknot wraps around the 1-framed horizontal unknot, and the handle-slide move implies that we can delete the $(n+1)$-framed unknot and change the framing of the horizontal unknot to 0. Then the cutting property of the 0-framed unknot in Kirby calculus (see, for example, \cite{Kir78, MP94}) implies that we end up with a lens space $L(m+3,1)$. In fact, we can see that the group presentation in \eqref{eq:Gamma} can still be used in this case: Take $p = 2$, $q = m+1$ and $r = n+1 = 1$, then $z^r = xyz$ implies $xy = 1$, and $x^2 = y^{m+1}$ implies that 
\[\langle x, y, z \mid x^{p} = y^{q} = z^{r} = xyz \rangle \cong \langle y \mid y^{m+3} \rangle \cong \BZ/(m+3)\BZ \cong \pi_1(L(m+3,1))\,.\]

\item 
When $m = n = 1$, $\widetilde{\Isom^{+}}(\BP) \cong \SU(2)$, which is  simply the 3-sphere $\BS^3$ topologically. Using \eqref{eq:Gamma}, it is easy to see that $\BM_{1,1} \cong \BS^3/Q_8$, and Figure \ref{fig:Dn-local-3} indeed reduces to Figure \ref{fig:Q8}.

\item For any $m \ge 1$, $n \ge 1$, consider the framed link presentation of $\BM_{m,n}$ given by Figure \ref{fig:Mmn-link}(b). Since the Dehn surgery on $\BS^3$ along a 0-framed unknot results in $\BS^2\times \BS^1$ (see, for example, \cite{Rol76}), we can see that $\BM_{m,n}$ is a Seifert manifold over $\BS^2$ with $4$ singular fibers. Such a manifold is called a \emph{large Seifert manifold} \cite[Sec.~5.3]{Orl72}, and by Thm.~6 in loc.~cit., they are completely determined by their fundamental groups. 
\end{enumerate}    
\end{remark}

It is illustrated in \cite[Prop.\ 7]{Rol76} that one can convert a framed link diagram presentation of the Poincare homology sphere into a Heegaard diagram, and this can be generalized to arbitrary 3-manifolds. We  convert the link diagram Figure \ref{fig:Mmn-link}(d) into a Heegaard diagram of $\BM_{m,n}$, which is shown in Figure \ref{fig:Dn-local-3} (see also \cite{GH07}). 

As before, in Figure \ref{fig:Dn-local-3}, the four black circles are attaching circles indicating the positions of handles. The horizontal green and violet lines represent the lower curves, and the blue and red strands stand for the upper curves. For the discussions below, we also added combings and twist fronts on the Heegaard diagram, where the gray dashed arrows represent the combing $b_1$, and the four base points, colored in gray, are connected by the black twist fronts. This makes Figure \ref{fig:Dn-local-3} a 2-combed Heegaard diagram for $\BM_{m,n}$, denote by $D$.

By an abuse of notations, denote the fundamental group of the 3-manifold that $D$ represents be denoted by $\pi_1(D)$. We show that $\pi_1(D)$ is isomorphic to $\pi_1(\BM_{m,n}) \cong \Gamma(2, m+1, n+1)$, which, by the discussions above, provides an alternative argument that $D$ is indeed a Heegaard diagram of $\BM_{m,n}$. Start by choosing a base point, say, a point slightly below $p_1$ in Figure \ref{fig:Dn-local-3}. Then we choose two generators of $\pi_1(D)$ to be the loops at the base point that winds once around the upper right handle clockwise (denoted by $a$) and the lower right handle anticlockwise (denoted by $b$) respectively. By reading the homotopy class of the two upper curves in terms of these chosen generators,  we can write down the presentation of $\pi_1(D)$ as
\[\pi_1(D) = \langle a, b\mid aba^{-m}b=ab^{-n}ab=1 \rangle \cong \langle a, b\mid a^m = bab, b^n = aba \rangle\,.\]

Now we can compare the above presentation with the presentation in \eqref{eq:Gamma} when $p = 2$, $q = m+1$, and $r = n+1$. More precisely, in this case, we have 
\[\Gamma(2, m+1, n+1) = \langle x, y, z \mid x^2 = y^{m+1} = z^{n+1} = xyz \rangle\,. \]
Therefore, we have $x = yz$, which implies that $y^{m+1} = yzyz$, so $y^m=zyz$. Similarly, $z^{n+1} = yzyz$ implies $z^{n} = yzy$. Therefore, by assigning $y \mapsto a$ and $z \mapsto b$, we can establish an isomorphism of fundamental groups $\pi_1(D) \cong \Gamma(2, m+1, n+1) \cong \pi_1(\BM_{m,n})$ as desired.

\begin{figure}[ht]
    \centering
    \includegraphics[width=300pt]{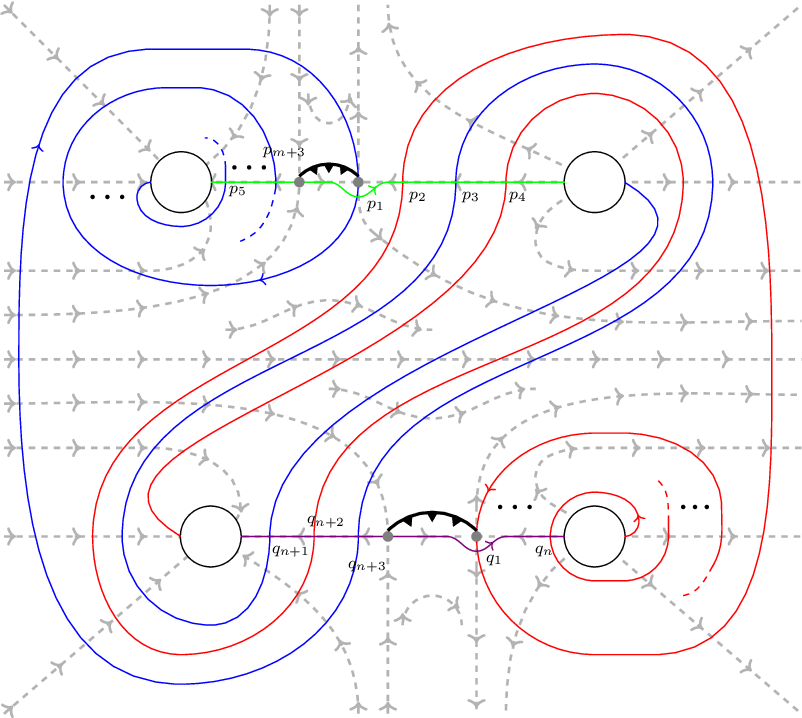}
    \caption{A 2-combed Heegaard diagram for $\BM_{m,n}$.}
    \label{fig:Dn-local-3}
\end{figure}

We list the rotation numbers of $D$ in the tables below. First of all, it is easy to read the rotation numbers of the $I$-points along the lower curves.

\begin{table}[ht]
\begin{tabular}{|c|c|c|c|c|}
\hline
\multirow{3}{*}{$\eta_1$ (green)} & 
$\eta_1 \cap I$ & $p_1$ & $p_2, ..., p_{m+3}$ & Total \\
\cline{2-5}& $\theta$ & $\frac{1}{4}$ & $\frac{1}{2}$ & $\frac{1}{2}$ \\
\cline{2-5}& $\phi$ & $0$ & $0$ & $\frac{1}{2}$  \\
\hline
\multirow{3}{*}{$\eta_2$ (violet)} & 
$\eta_2 \cap I$ & $q_{1}$ & $q_2, ..., q_{n+3}$ & Total \\
\cline{2-5}& $\theta$ & $\frac{1}{4}$ & $\frac{1}{2}$ & $\frac{1}{2}$ \\
\cline{2-5}& $\phi$ & $0$ & $0$ & $\frac{1}{2}$  \\
\hline
\end{tabular}
\caption{}
\label{tbl:Mmn-1}
\end{table}

Then we list the rotation numbers on the upper curve and give the $S$-term in \eqref{eq:PMDH} (note that by design, there is no $T$-term). The $I$-points below are listed in accordance with the orientations of the upper curves, starting from their base points. By the description of the evaluation process of the Kuperberg invariant (below Definition \ref{def:Kup}), this will help us write down the formula for the Kuperberg invariant.

\begin{table}[ht]
\begin{tabular}{|c|c|c|c|c|c|c|c|}
\hline
\multirow{4}{*}{$\mu_1$ (blue)} & 
$\mu_1 \cap I$ & $p_1$ & $p_{m+3}, ..., p_{5}$& $q_{n+1}$ & $p_{3}$ & $q_{n+3}$ & Total \\
\cline{2-8}& $\theta$ & $0$ & $-\frac{3}{4}$ &$-\frac{3}{4}$   &$-\frac{1}{4}$ & $\frac{1}{4}$ & $-\frac{1}{2}$ \\
\cline{2-8}& $S$-term & $S$ & $S^3$ & $S^3$ & $S^2$ & $S$ & \\
\cline{2-8}& $\phi$ & $0$ & $0$ & $0$ & $0$ & $0$ & $\frac{1}{2}$  \\
\hline
\end{tabular}
\caption{}
\label{tbl:Mmn-2}
\end{table}
\begin{table}[ht]
\begin{tabular}{|c|c|c|c|c|c|c|c|}
\hline
\multirow{4}{*}{$\mu_2$ (red)} & 
$\mu_2 \cap I$ & $q_{1}$ & $p_{2}$ & $q_{n+2}$ & $p_{4}$ & $q_{n}, ..., q_{2}$ & Total \\
\cline{2-8}& $\theta$ & $0$ & $\frac{1}{4}$ & $-\frac{1}{4}$ & $-\frac{3}{4}$ & $-\frac{3}{4}$ & $-\frac{1}{2}$ \\
\cline{2-8}& $S$-term & $S$ & $S$ & $S^2$ & $S^3$ & $S^3$ &\\
\cline{2-8}& $\phi$ & $0$ & $0$ & $0$ &$0$& $0$ &$\frac{1}{2}$  \\
\hline
\end{tabular}
\caption{}
\label{tbl:Mmn-3}
\end{table}

The following lemma follows immediately from the tables above.

\begin{lemma}
The 2-combed Heegaard diagram in Figure \ref{fig:Dn-local-3} is a framed Heegaard diagram.\qed
\end{lemma}

We can now state the main theorem of this section. Recall the Sweedler power $P^{(n)}$ of a Hopf algebra $H$ is $P^{(n, 1)}: H \to H$, $x \mapsto m\Delta^n(x)$.

\begin{thm}\label{t:genus_2}
Let $H$ be a finite-dimensional Hopf algebra over a field $\kk$. For any positive integers $m, n$, denote the diagram framing of $\BM_{m,n}$ associated to the framed Heegaard diagram in Figure \ref{fig:Dn-local-3} by $\frm_{m,n}$, then we have
\begin{equation}
\begin{split}
\label{eq:Kmn}
&K(\BM_{m,n}, \frm_{m,n}, H) =  \Tr\left((S \o S)\circ \Psi^{m,n}\right) \\
= &\lambda S\left(S^{-2}(\Lambda^2_{(4)})S^{-1}(\Lambda^1_{(2)})\Lambda^2_{(2)}P^{(m)}(\Lambda^1_{(4)})\right)\cdot 
\lambda S\left(P^{(n-1)}(\Lambda^2_{(1)})\Lambda^1_{(3)}S^{-1}(\Lambda^2_{(3)})S^{-2}(\Lambda^1_{(1)})\Lambda^2_{(5)}\right)
\end{split}
\end{equation}
where $\Psi^{m,n}: H^{\otimes 2} \to  H^{\otimes 2}$ is the $\kk$-linear operator defined by
\begin{equation} \label{eq:Psi}
    \Psi^{m, n}(x\otimes y)=S^{-2}(y_{(4)})S^{-1}(x_{(2)})y_{(2)}P^{(m-1)}(x_{(4)})\otimes P^{(n-1)}(y_{(1)})x_{(3)}S^{-1}(y_{(3)})S^{-2}(x_{(1)})
\end{equation}
for all $x$, $y \in H$.
Moreover, $K(\BM_{m,n}, \frm, H)$ above is a gauge invariant of $H$. 
\end{thm}
\begin{proof}
By Tables \ref{tbl:Mmn-1}, \ref{tbl:Mmn-2} and \ref{tbl:Mmn-3}, we can write down the formula for $K(\BM_{m,n}, \frm_{m,n}, H)$ as follows. As is indicated in the tables, by our choice of orientation, starting from their base points, $\mu_1$ passes $p_1$, $p_{m+3}$, ..., $p_5$, $q_{n+1}$, $p_3$, $q_{n+3}$ and $\mu_2$ passes $q_1$, $p_2$, $q_{n+2}$, $p_4$, $q_n$, ..., $q_2$ sequentially. Moreover, recall that in \eqref{eq:twisted-ld}, we have $\ld_{-\theta(\mu_1)}=\ld_{-\theta(\mu_2)}=\ld_{\tfrac{1}{2}}=\ld\circ S^{-1}$ {and $\Ld_{\frac{1}{2}} = \Ld$.} So 
\[
\begin{split}
&\, K(\BM_{m,n}, \frm_{m,n}, H) \\
=&\, \ld S^{-1}\left(S(\Lda{1}^{1}) S^3(\Lda{m+3}^{1}) \cdots S^3(\Lda{5}^{1}) S^3(\Lda{n+1}^{2}) S^2(\Lda{3}^{1}) S(\Lda{n+3}^{2})\right)\\
&\quad \cdot \ld S^{-1}\left(S(\Lda{1}^2) S(\Lda{2}^{1}) S^2(\Lda{n+2}^{2}) S^3(\Lda{4}^{1}) S^3(\Lda{n}^{2}) \cdots S^3(\Lda{2}^{2}) \right)\\
=&\, \ld \left(\Ld^{2}_{(5)} S(\Ld^{1}_{(3)}) S^{2}(\Ld^{2}_{(3)}) S^{2}(P^{m-1}(\Ld^{1}_{(5)})) \Ld^{1}_{(1)}\right)\cdot 
\ld \left(S^{2}(P^{n-1}(\Ld^{2}_{(2)})) S^{2}(\Ld^{1}_{(4)}) S(\Ld^{2}_{(4)}) \Ld^{1}_{(2)} \Ld^{2}_{(1)}  \right)\\
= &\, (\ld \o \ld)\bigg((S^2 \o S^2)\Psi^{m,n}\Big(\Lda{2}^1 \o \Lda{2}^2\Big) (\Lda{1}^1 \o \Lda{1}^2) \bigg) \\
= &\, \Tr((S \ot S) \circ \Psi^{m,n}) 
\end{split}
\]
by {the second} Radford's trace formula (Theorem \ref{t:t2} (v)). The second equation of \eqref{eq:Kmn} follows immediately from the first trace formula of Theorem \ref{t:t2} (v) for the operator $(S \o S)\circ \Psi^{m,n}$ on $H \o H$, and  this finishes the proof of the the first part of Theorem \ref{t:genus_2}.

Now, we prove the gauge invariance of $K(\BM_{m,n}, \frm_{m,n}, H)$ which will be abbreviated as  $K^{m,n}$ for simplicity. Let $F = \ff{i}{1} \o \ff{i}{2} \in H\o H$ be any 2-cocycle of $H$.  We denote by $\Psi^{m,n}_F: H_F \o H_F \to H_F \o H_F$ the associated operator defined in \eqref{eq:Psi}, $K^{m,n}_F : = K(\BM_{m,n}, \frm_{m,n}, H_F)$ and $F^{-1} = \dd{j}{1}\o \dd{j}{2}$.  In particular, if $\Delta_F(x)=\tilde{x}_{(1)} \o \tilde{x}_{(2)}$ and $\Delta_F(y)=\tilde{y}_{(1)} \o \tilde{y}_{(2)}$, then
\begin{align*}
\Psi_F^{m,n}(x \o y) =&\, S_F^{-2}(\tilde{y}_{(4)})S_F^{-1}(\tilde{x}_{(2)})\tilde{y}_{(2)}P_F^{(m-1)}(\tilde{x}_{(4)})\otimes P_F^{(n-1)}(\tilde{y}_{(1)})\tilde{x}_{(3)}S_F^{-1}(\tilde{y}_{(3)})S_F^{-2}(\tilde{x}_{(1)})\,.
\end{align*}
 By the first Radford’s trace formula of Theorem \ref{t:t2} (v), we have
\begin{align*}
K^{m, n}_F =&\, \Tr((S_{F} \ot S_{F}) \circ \Psi^{m,n}_{F}) \\
            = & \, (\ld \o \ld)\bigg( (S(\Lda{2}^1) \o S(\Lda{2}^2))(S_F\o S_F)\Psi^{m,n}_F(\Lda{1}^1 \o \Lda{1}^2)\bigg) \\
            =&  \, \ld\bigg(S(\Lda{2}^1) S_F \Big(S_F^{-2}(\tilde{\Ld}^2_{(1,4)})S_F^{-1}(\tilde{\Ld}^1_{(1,2)})\tilde{\Ld}^2_{(1,2)}P_F^{(m-1)}(\tilde{\Ld}^1_{(1,4)})\Big)\bigg) \\
            &\  \cdot \ld\bigg(S(\Lda{2}^2) S_F\Big(P_F^{(n-1)}(\tilde{\Ld}^2_{(1,1)})\tilde{\Ld}^1_{(1,3)}S_F^{-1}(\tilde{\Ld}^2_{(1,3)})S_F^{-2}(\tilde{\Ld}^1_{(1,1)})\Big), 
\end{align*}
where $\Delta^4_F(\Lda{1}^i)\o \Lda{2}^i= \tLd{1,1}^i \o \tLd{1,3}^i\o \tLd{1,3}^i\o\tLd{1,4}^i \o \Lda{2}^i$, $i=1,2$.
Now, we apply Lemma \ref{l:t1} to $\Ld^1$ to obtain
\begin{align*}
K_F^{m,n}=&  \, \ls\bigg(\dd{j}{1} S_F^{-2}(\tLd{1,4}^2)S_F^{-1}(\ff{i}{2}\Lda{2}^1 \dd{j}{3})\tLd{1,2}^2\Big(\prod_{w=4}^{m+2} \ff{i}{w} \Lda{w}^1 \dd{j}{w+1}\Big) \ff{i}{m+3}\Lda{m+3}^1\bigg) \\
            &\  \cdot \ld\bigg(S(\Lda{2}^2) S_F\Big(P_F^{(n-1)}(\tLd{1,1}^2)\ff{i}{3}\Lda{1,3}^1\dd{j}{4}S_F^{-1}(\tLd{1,3}^2)S_F^{-2}(\ff{i}{1}\Lda{1}^1 \dd{j}{2})\Big)
\end{align*}
and then to $\Ld^2$ to rewrite the last expression as
\begin{align*}
&\ls \bigg(\!
\dd{j}{1} S_F^{-2}(\ff{p}{n+2} \Ld^2_{(n+2)} \dd{q}{n+3}) S_F^{-1}(\ff{i}{2}\Ld^1_{(2)}\dd{j}{3}) \ff{p}{n} \Ld^2_{(n)} \dd{q}{n+1}
\Big(\prod_{w=4}^{m+2} \ff{i}{w} \Ld_{(w)}^{1} \dd{j}{w+1}\Big) f^{[m+3]}_i\Lambda^1_{(m+3)}
\bigg)\\
& \cdot \lambda S \bigg(
\dd{q}{1}\Big(\prod_{\ell=1}^{n-1} \ff{p}{\ell}\Ld_{(\ell)}^{2} \dd{q}{\ell+1}\Big)
 \ff{i}{3} \Lambda^1_{(3)} \dd{j}{4} S_F^{-1}(\ff{p}{n+1} \Ld^2_{(n+1)} \dd{q}{n+2}) S_F^{-2}(\ff{i}{1}\Lambda^1_{(1)}\dd{j}{2}) \ff{p}{n+3} \Lambda^2_{(n+3)} \bigg)\\
 =&\lambda S\bigg(
\dd{j}{1} S_F^{-2}(\ff{p}{n+2} \Ld^2_{(n+2)} \dd{q}{n+3}) S_F^{-1}(\ff{i}{2}\Ld^1_{(2)}\dd{j}{3}) \ff{p}{n} \Ld^2_{(n)} \dd{q}{n+1}\ff{i}{4}\Ld_{(4)}^{1}
\Big(\prod_{w=5}^{m+3}  \dd{j}{w}\ff{i}{w}\Lambda^1_{(w)}\Big)
\bigg)\\
& \cdot \lambda S \bigg(
\Big(\prod_{\ell=1}^{n-1} \dd{q}{\ell}\ff{p}{\ell}\Ld_{(\ell)}^{2} \Big)
\dd{q}{n} \ff{i}{3} \Lambda^1_{(3)} \dd{j}{4} S_F^{-1}(\ff{p}{n+1} \Ld^2_{(n+1)} \dd{q}{n+2}) S_F^{-2}(\ff{i}{1}\Lambda^1_{(1)}\dd{j}{2}) \ff{p}{n+3} \Lambda^2_{(n+3)} \bigg).
\end{align*}
By Lemma \ref{lem:trick-3} (i), we have 
\[F_{n+3}\!=\!(F_{n-1} \o 1\sot{4})(\Delta^{n-1} \o \id\sot{4})(F_5) \quad\text{and}\quad F_{m+3}\!= \!(1\sot{4}\o F_{m-1})(\id\sot{4} \o \Delta^{m-1})(F_5)\,,\] 
which means $K^{m,n}_F$ is equal to
\begin{align*}
&\lambda S\left(
\dd{j}{1} S_F^{-2}(\ff{p}{4} \Ld^2_{(n+2)} \dd{q}{5}) S_F^{-1}(\ff{i}{2}\Ld^1_{(2)}\dd{j}{3}) \ff{p}{2} \Ld^2_{(n)} \dd{q}{3}\ff{i}{4}\Ld_{(4)}^1
\left(\prod_{w=5}^{m+3}   \dd{j\,(w-4)}{5} \ff{i\,(w-4)}{5}\Lambda^1_{(w)}\right) 
\right)\\
& \cdot \lambda S \left(
\left(\prod_{\ell=1}^{n-1} \dd{q\, (\ell)}{1} \ff{p\, (\ell)}{1}\Ld_{(\ell)}^{2} \right)
\dd{q}{2} \ff{i}{3} \Lambda^1_{(3)} \dd{j}{4} S_F^{-1}(\ff{p}{3} \Ld^2_{(n+1)} \dd{q}{4}) S_F^{-2}(\ff{i}{1}\Lambda^1_{(1)}\dd{j}{2}) \ff{p}{5} \Lambda^2_{(n+3)} \right)\\
=&\lambda S\left(
\dd{j}{1} S_F^{-2}(\ff{p}{4} \Ld^2_{(4)} \dd{q}{5}) S_F^{-1}(\ff{i}{2}\Ld^1_{(2)}\dd{j}{3}) \ff{p}{2} \Ld^2_{(2)} \dd{q}{3}\ff{i}{4}\Ld_{(4)}^1
P^{(m-1)}  \left(\dd{j}{5} \ff{i}{5} \Lambda^1_{(5)}\right)
\right)\\
& \cdot \lambda S \left(
P^{(n-1)} \left(\dd{q}{1} \ff{p}{1}\Lda{1}^2 \right)
\dd{q}{2} \ff{i}{3} \Lambda^1_{(3)} \dd{j}{4} S_F^{-1}(\ff{p}{3} \Ld^2_{(3)} \dd{q}{4}) S_F^{-2}(\ff{i}{1}\Lambda^1_{(1)}\dd{j}{2}) \ff{p}{5} \Lambda^2_{(5)} \right)\,.
\end{align*}
Now, we apply Theorem \ref{t:t2} (iii) to $\dd{j}{5}\ff{i}{5}$ and then $\dd{q}{1}\ff{p}{1}$ in the last expression to obtain 
\begin{align*}
& \lambda S\left(
\dd{j}{1} S_F^{-2}\Big(\ff{p}{4} S^{-1}\big(\dd{q\,(2)}{1} \ff{p\, (2)}{1}\big) \Ld^2_{(4)} \dd{q}{5}\Big) S_F^{-1}\Big(\ff{i}{2}S\big(\dd{j\, (3)}{5} \ff{i\, (3)}{5}\big)\Ld^1_{(2)}\dd{j}{3}\Big) \right.\\
&\qquad \left.\ff{p}{2} S^{-1}\big(\dd{q\,(4)}{1} \ff{p\, (4)}{1}\big) \Ld^2_{(2)} \dd{q}{3}\ff{i}{4}S(\dd{j\, (1)}{5} \ff{i\, (1)}{5})
P^{(m)}  \Big( \Ld^1_{(4)}\Big)
\right)\\
& \cdot \lambda S \left(
P^{(n-1)} \Big(\Lda{1}^2 \Big)
\dd{q}{2} \ff{i}{3} S\big(\dd{j\, (2)}{5} \ff{i\, (2)}{5}\big)\Lda{3}^1 \dd{j}{4} S_F^{-1}\Big(\ff{p}{3} S^{-1}(\dd{q\,(3)}{1} \ff{p\, (3)}{1}) \Ld^2_{(3)} \dd{q}{4}\Big)\right.\\
&\qquad \left. S_F^{-2}\Big(\ff{i}{1}S\big(\dd{j\, (4)}{5} \ff{i\, (4)}{5}\big)\Lda{1}^1\dd{j}{2}\Big) \ff{p}{5} S^{-1}\big(\dd{q\,(1)}{1} \ff{p\, (1)}{1}\big) \Lda{5}^2 \right)\,.
\end{align*}
{Applying Lemma \ref{lem:trick-3} (v)  Equations \eqref{eq:fSf-dv} and \eqref{eq:fSf-vSd} to $\ff{p}{*}$ and $\ff{i}{*}$, we find $K^{m,n}_F$ is equal to}

\begin{align*}
&\lambda S\left(
\dd{j}{1} S_F^{-2}(S^{-1}(\dd{q\,(2)}{1}  \dd{p}{2} u  )\Ld^2_{(4)} \dd{q}{5}) S_F^{-1}(u S(\dd{i}{3}) S(\dd{j\, (3)}{5})\Ld^1_{(2)}\dd{j}{3}) \right.\\
&\qquad\left. S^{-1}(\dd{q\,(4)}{1} \dd{p}{4} u) \Ld^2_{(2)} \dd{q}{3}u S(\dd{i}{1})S(\dd{j\, (1)}{5})
P^{(m)}  \left( \Lda{4}^1\right)
\right)\\
& \cdot \lambda S \left(
P^{(n-1)} \left(\Lda{1}^2 \right)
\dd{q}{2} u S(\dd{i}{2}) S(\dd{j\, (2)}{5})\Lda{3}^1 \dd{j}{4} S_F^{-1}( S^{-1}(\dd{q\,(3)}{1} \dd{p}{3}u) \Ld^2_{(3)} \dd{q}{4})\right.\\
&\qquad\left. S_F^{-2}(u S(\dd{i}{4})S(\dd{j\, (4)}{5})\Lda{1}^1\dd{j}{2})  S^{-1}(\dd{q\,(1)}{1} \dd{p}{1}u) \Lda{5}^2 \right)\\
=&\, \lambda S\bigg(
\dd{j}{1} S_F^{-2}\Big(S^{-1}(\dd{p}{2} u  )\Ld^2_{(4)} \dd{p}{8}\Big) S_F^{-1}\Big(u S(\dd{j}{7})\Ld^1_{(2)}\dd{j}{3}\Big)  S^{-1}(\dd{p}{4} u) \Ld^2_{(2)} \dd{p}{6}u S(\dd{j}{5})
P^{(m)}  \Big( \Lda{4}^1\Big)
\bigg)\\
&  \cdot \lambda S \bigg(
P^{(n-1)} \Big(\Lda{1}^2 \Big)
\dd{p}{5} u S(\dd{j}{6}) \Lda{3}^1 \dd{j}{4} S_F^{-1}\Big( S^{-1}(\dd{p}{3}u) \Ld^2_{(3)} \dd{p}{7}\Big)   S_F^{-2}\Big(u S(\dd{j}{8})\Lda{1}^1\dd{j}{2}\Big)  S^{-1}(\dd{p}{1}u) \Lda{5}^2 \bigg) 
\end{align*}
by applying Lemma \ref{lem:trick-3} (i) to $\dd{q\, (*)}{1}\dd{p}{*}$.  Using Corollary \ref{c:t2} (i) with $x=\dd{j}{1}$ in the preceding expression  to obtain
\begin{align*}
K^{m,n}_F= &\lambda S\bigg(
 S_F^{-2}\Big(S^{-1}(\dd{p}{2} u  )\Ld^2_{(4)} \dd{p}{8}\Big) S_F^{-1}\Big(u S(\dd{j}{7})S^2(\dd{j\,(2)}{1})\Lda{2}^1 S(\dd{j\,(6)}{1})\dd{j}{3}\Big) \\
 & \qquad S^{-1}(\dd{p}{4} u) \Ld^2_{(2)} \dd{p}{6}u S(\dd{j}{5}){S^2(\dd{j\, (4)}{1})}
P^{(m)}  \Big( \Lda{4}^1\Big)
\bigg)\\
\cdot & \lambda S  \Big(P^{(n-1)} \big(\Lda{1}^2 \big)
\dd{p}{5} u S(\dd{j}{6}) S^2(\dd{j\,(3)}{1})\Lda{3}^1 S(\dd{j\,(5)}{1}) \dd{j}{4} S_F^{-1}\Big( S^{-1}(\dd{p}{3}u) \Ld^2_{(3)} \dd{p}{7}\Big) \\
& \qquad S_F^{-2}\Big(u S(\dd{j}{8})S^2(\dd{j\,(1)}{1})\Lda{1}^1 S(\dd{j\,(7)}{1})\dd{j}{2}\Big)  S^{-1}(\dd{p}{1}u) \Lda{5}^2 \bigg) \\
=&\lambda S\bigg(
 S_F^{-2}(S^{-1}(\dd{p}{2} u  )\Ld^2_{(4)} \dd{p}{8}) S_F^{-1}(u S(u^{-1}_{(6)} \dd{j}{6})\Lda{2}^1 u^{-1}_{(2)} \dd{j}{2})  S^{-1}(\dd{p}{4} u) \Ld^2_{(2)} \dd{p}{6}u {S(u^{-1}_{(4)} \dd{j}{4})}
P^{(m)}  \Big( \Lda{4}^1\Big)\bigg)\\
\cdot & \lambda S  \Big(P^{(n-1)} \big(\Lda{1}^2 \big)
\dd{p}{5} u S(u^{-1}_{(5)} \dd{j}{5})\Lda{3}^1 u^{-1}_{(3)} \dd{j}{3} S_F^{-1}\Big( S^{-1}(\dd{p}{3}u) \Ld^2_{(3)} \dd{p}{7}\Big) \\
& \quad S_F^{-2}(u S(u^{-1}_{(7)} \dd{j}{7})\Lda{1}^1 
u^{-1}_{(1)} \dd{j}{1} )S^{-1}(\dd{p}{1}u) \Lda{5}^2 \bigg) \,\text{ by using Lemma \ref{lem:trick-3}  \eqref{eq:Sdd-vd} on $S(\dd{j\,(*)}{1})\dd{j}{9-*}$.}
\end{align*}
It follows from Corollary \ref{c:t2} (i) that the preceding expression is equal to
\begin{align*}
&\lambda S\bigg(S^{-1}(u^{-1})
 S_F^{-2}\Big(S^{-1}(\dd{p}{2} u  )\Ld^2_{(4)} \dd{p}{8}\Big) S_F^{-1}\Big(u S(\dd{j}{6})\Lda{2}^1  \dd{j}{2}\Big)  S^{-1}(\dd{p}{4} u) \Ld^2_{(2)} \dd{p}{6}u S(\dd{j}{{4}})
P^{(m)}  \Big( \Lda{4}^1\Big)
\bigg)\\
\cdot & \lambda S  \bigg(P^{(n-1)} \Big(\Lda{1}^2 \Big)
\dd{p}{5} u S(\dd{j}{5})\Lda{3}^1 \dd{j}{3} S_F^{-1}\Big( S^{-1}(\dd{p}{3}u) \Ld^2_{(3)} \dd{p}{7}\Big)S_F^{-2}\Big(u S(\dd{j}{7})\Lda{1}^1 \dd{j}{1} \Big)S^{-1}(\dd{p}{1}u) \Lda{5}^2 \bigg).
 \end{align*}
Recall that $S_F^{-1}(x) = S^{-1}(u^{-1} x u)$, and so $S_{F}^{-2}(x) = S^{-1}(u)S^{-2}(u^{-1} x u)S^{-1}(u^{-1})$. We have $K^{m,n}_F$ is equal to
\begin{align*}
&\lambda S\bigg(
S^{-2}(u^{-1}S^{-1}(\dd{p}{2} u  )\Ld^2_{(4)} \dd{p}{8} u) S^{-1}(S(\dd{j}{6})\Lda{2}^1\dd{j}{2})S^{-1}(\dd{p}{4}u) \Ld^2_{(2)} \dd{p}{6}u S(\dd{j}{4}) 
P^{(m)}  \Big( \Lda{4}^1\Big)
\bigg)\\
\cdot &\lambda S \bigg(
P^{(n-1)} \Big(\Lda{1}^2 \Big)
\dd{p}{5} u S(\dd{j}{5})\Lda{3}^1  \dd{j}{3} S^{-1}( S^{-1}(\dd{p}{3}u) \Ld^2_{(3)} \dd{p}{7} u)S^{-2}(S(\dd{j}{7})\Lda{1}^1  \dd{j}{1} u)  S^{-1}(\dd{p}{1}) \Lda{5}^2 \bigg)\,.
\end{align*}
Now we apply Corollary \ref{c:t2}  (ii) with $x=S^{-1}(\dd{p}{1})$ to obtain
\begin{align*}
K^{m,n}_F=&\lambda S\bigg(
S^{-2}(u^{-1}S^{-1}(\dd{p}{2} u  ) \dd{p\,(7)}{1}  \Ld^2_{(4)} S(\dd{p\,(1)}{1})\ \dd{p}{8} u) S^{-1}(uS(\dd{j}{6})\Lda{2}^1\dd{j}{2})S^{-1}(\dd{p}{4}) \\
& \qquad \dd{p\,(5)}{1}\Ld^2_{(2)} S(\dd{p\,(3)}{1})\dd{p}{6}u S(\dd{j}{4}) 
P^{(m)}  \Big( \Lda{4}^1\Big)
\bigg)\\
\cdot &\lambda S \bigg(
P^{(n-1)} \Big(\Lda{1}^2 \Big)
S(\dd{p\,(4)}{1})\dd{p}{5} u S(\dd{j}{5})\Lda{3}^1  \dd{j}{3} S^{-1}( S^{-1}(\dd{p}{3}u) \dd{p\,(6)}{1}  \Ld^2_{(3)} S(\dd{p\,(2)}{1}) \dd{p}{7} u)\\
& \qquad S^{-2}(S(\dd{j}{7})\Lda{1}^1  \dd{j}{1} u)  \Lda{5}^2 \bigg)
\end{align*}
which is equal to
\begin{align*}
&\lambda S\bigg(
S^{-2}(u^{-1} \ff{p}{7}   \Ld^2_{(4)} S(\ff{p}{1})) S^{-1}(uS(\dd{j}{6})\Lda{2}^1\dd{j}{2}) S^{-1}(u^{-1}) \ff{p}{5} \Ld^2_{(2)} S(\ff{p}{3}) S(\dd{j}{4}) 
P^{(m)}  \left( \Lda{4}^1\right)
\bigg)\\
\cdot &\lambda S \bigg(
P^{(n-1)} (\Lda{1}^2 )
S(\ff{p}{4}) S(\dd{j}{5})\Lda{3}^1  \dd{j}{3} S^{-1}( \ff{p}{6}   \Ld^2_{(3)} S(\ff{p}{2}) ) S^{-2}(S(\dd{j}{7})\Lda{1}^1  \dd{j}{1} u)  \Lda{5}^2 \bigg)
 \end{align*}
 by Lemma applying \ref{lem:trick-3} (v) Equation \eqref{eq:Sdd-vd} to $S(\dd{p\, (*)}{1})\dd{p}{9-*}$. We now substitute $S^{-2}(u^{-1})=S^{-1}(\dd{i}{1})S^{-2}(\dd{i}{2})$ and  $S^{-2}(u)= S^{-2}(\ff{k}{1})S^{-1}(\ff{k}{2})$, and then apply Corollary \ref{c:t2} (i) with $x=S^{-1}(\dd{i}{1})$ and  (ii) with $x=S^{-1}(\ff{k}{2})$  to obtain
\begin{align*}
K^{m,n}_F\!=&\lambda S\bigg(\!
S^{-2}\Big(\dd{i}{2}\ff{p}{7}  \ff{k\,(7)}{2} \Ld^2_{(4)} S(\ff{k\,(1)}{2}) S(\ff{p}{1})\Big) S^{-1}\Big(S(\dd{j}{6})S(\dd{i\,(6)}{1})\Lda{2}^1\dd{i\,(2)}{1}\dd{j}{2}\Big)  \ff{p}{5} \ff{k\,(5)}{2}\Ld^2_{(2)}S(\ff{k\,(3)}{2}) \\
&\qquad S(\ff{p}{3}) S(\dd{j}{4}) S(\dd{i\,(4)}{1})
P^{(m)}  \Big( \Lda{4}^1\Big) \bigg)\\
\cdot &\lambda S \bigg(
P^{(n-1)} \Big(\Lda{1}^2 \Big) S(\ff{k\,(4)}{2})
S(\ff{p}{4}) S(\dd{j}{5})S(\dd{i\,(5)}{1})\Lda{3}^1 \dd{i\,(3)}{1}  \dd{j}{3} S^{-1}\Big( \ff{p}{6}  \ff{k\,(6)}{2} \Ld^2_{(3)} S(\ff{k\,(2)}{2}) S(\ff{p}{2}) \Big) \\
& \qquad S^{-2}\Big(S(\dd{j}{7})S(\dd{i\,(7)}{1})\Lda{1}^1 \dd{i\,(1)}{1}  \dd{j}{1} \ff{k}{1}\Big)  \Lda{5}^2 \bigg)
\end{align*}
 which can be further simplified by applying Lemma \ref{lem:trick-3} (i) to $\ff{p}{*}\ff{k\, (*)}{2}$ and  $\dd{i\,(*)}{1}\dd{j}{*}$ to
 \begin{align*}
K^{m,n}_F = &\lambda S\bigg(
S^{-2}\Big(\dd{i}{8} \ff{k}{8} \Ld^2_{(4)} S(\ff{k}{2}) \Big) S^{-1}\Big(S(\dd{i}{6})\Lda{2}^1\dd{i}{2}\Big)  \ff{k}{6} \Ld^2_{(2)}S(\ff{k}{4}) S(\dd{i}{4})
P^{(m)}  \Big( \Lda{4}^1\Big) \bigg)\\
&\, \cdot \lambda S \bigg(
P^{(n-1)} \Big(\Lda{1}^2 \Big) 
S(\ff{k}{5}) S(\dd{i}{5})\Lda{3}^1   \dd{i}{3} S^{-1}\Big( \ff{k}{7}   \Ld^2_{(3)} S(\ff{k}{3})\Big) S^{-2}\Big(S(\dd{i}{7})\Lda{1}^1 \dd{i}{1}   \ff{k}{1}\Big)  \Lda{5}^2 \bigg)
\end{align*}
\begin{align*}
=&\, \lambda S\left(
S^{-2}\Big(\dd{i}{8}\ff{k}{8} \Lda{4}^2\Big) S^{-1}(\dd{i}{2}\ff{k}{2})   S^{-1}(\Lda{2}^1) \dd{i}{6}\ff{k}{6}  \Ld^2_{(2)} S(\dd{i}{4}\ff{k}{4}) 
P^{(m)}  \left( \Lda{4}^1\right)
\right)\\
&\ \cdot \lambda S \big(
P^{(n-1)} \Big(\Ld_{(1)}^{2} \Big) S(\dd{i}{5}\ff{k}{5})\Lda{3}^1   \dd{i}{3} \ff{k}{3} S^{-1}(\Lda{3}^2) S^{-1}(\dd{k}{7}\ff{k}{7}) S^{-2}(\Lda{1}^1 \dd{i}{1} \ff{k}{1}) \Lda{5}^2 \big)\\
=&\, \lambda S\bigg(
S^{-2}(\Lda{4}^2 ) S^{-1}(\Lda{2}^1)   \Ld^2_{(2)} 
P^{(m)}  \Big( \Lda{4}^1\Big)
\bigg)
\cdot \lambda S \bigg(
P^{(n-1)} \Big(\Lda{1}^2 \Big) \Lda{3}^1    S^{-1}(\Lda{3}^2)  S^{-2}(\Lda{1}^1 ) \Lda{5}^2 \bigg)\\
=&\, K^{m,n}
\end{align*}
and this completes the proof.
\end{proof}

\bibliographystyle{abbrv}
\bibliography{MyRef}

\end{document}